\documentclass[12pt,a4paper]{article}
\usepackage{amssymb,amsmath,amsthm,amsfonts,epsfig,graphicx,psfrag}
\usepackage{tikz,subcaption}
\usepackage[T1]{fontenc}
\usepackage[utf8]{inputenc}
\usepackage[english]{babel}
\newcommand{\E}{\mathbb{E}}
\newcommand{\R}{\mathbb{R}}
\newcommand{\Px}{\mathbb{P}}
\newcommand{\N}{\mathbb{N}}
\newcommand{\cT}{\mathcal{T}}
\newcommand{\cI}{\mathcal{I}}
\newcommand{\cU}{\mathcal{U}}
\newcommand{\cZ}{\mathcal{Z}}
\newcommand{\cA}{\mathcal{A}}
\newcommand{\cE}{\mathcal{E}}
\newcommand{\cH}{\mathcal{H}}
\newcommand{\cX}{\mathcal{X}}
\newcommand{\card}{\textup{Card}}
\newtheorem{theorem}{Theorem}[section]
\newtheorem{lemma}[theorem]{Lemma}
\newtheorem{proposition}[theorem]{Proposition}

\newtheorem{definition}[theorem]{Definition}
\newtheorem{example}[theorem]{Example}

\newtheorem{remark}[theorem]{Remark}
\begin{document}

\title{A generic construction for high order approximation schemes of semigroups using random grids}
\author{Aurélien Alfonsi\footnote{Universit\'e Paris-Est, Cermics (ENPC), INRIA, F-77455 Marne-la-Vall\'ee, France. email: {\tt aurelien.alfonsi@enpc.fr}
  } ~and Vlad Bally\footnote{LAMA (UMR CNRS, UPEMLV, UPEC), MathRisk INRIA, Universit\'e Paris-Est. email: {\tt vlad.bally@u-pem.fr}}}

\maketitle

\noindent {AMS: 60H35, 65C30, 65C05, 65C20}
\begin{abstract}
Our aim is to construct high order approximation schemes for general semigroups of
linear operators $P_{t},t\geq 0$. In order to do it, we fix a time horizon $T
$ and the discretization steps $h_{l}=\frac{T}{n^{l}},l\in \N$ and we suppose
that we have at hand some short time approximation operators $Q_{l}$ such
that $P_{h_{l}}=Q_{l}+O(h_{l}^{1+\alpha })$ for some $\alpha >0$. Then, we
consider random time grids $\Pi (\omega )=\{t_0(\omega )=0<t_{1}(\omega )<...<t_{m}(\omega )=T\}$ such that for all $1\le k\le m$, $t_{k}(\omega )-t_{k-1}(\omega )=h_{l_{k}}$
for some $l_{k}\in \N$, and we associate the approximation discrete semigroup 
$P_{T}^{\Pi (\omega )}=Q_{l_{n}}...Q_{l_{1}}.$ Our main result is the following: for any approximation order $\nu $, we
 can construct random grids $\Pi_{i}(\omega )$ and coefficients $c_{i}$, with $i=1,...,r$ such that 
\[
P_{t}f=\sum_{i=1}^{r}c_{i}\E(P_{t}^{\Pi _{i}(\omega )}f(x))+O(n^{-\nu})
\]%
with the expectation concerning the random grids $\Pi _{i}(\omega ).$ Besides, $\card(\Pi _{i}(\omega ))=O(n)$
and the complexity of the algorithm is of order $n$, for any order
of approximation $\nu$. The standard example concerns diffusion processes, using the Euler approximation for~$Q_l$.
In this particular case and under suitable conditions, we are 
able to gather the terms in order to produce an estimator of $P_tf$ with finite variance.
However, an important feature of our approach is its universality in the sense that
it works for every general semigroup $P_{t}$ and approximations. Besides, approximation schemes sharing the same $\alpha$ lead to 
the same random grids $\Pi_{i}$ and coefficients $c_{i}$. Numerical illustrations are given for ordinary differential equations, piecewise deterministic Markov processes and diffusions. 
\end{abstract}
\bigskip

\section{Introduction}

We consider a semigroup of linear operators $(P_{t},t\geq 0)$ and we want to
construct high order approximation schemes based on some random grids.
Before presenting our general result, we would like to present the popular
example of diffusion processes and of approximation schemes of Euler type.
Consider the diffusion process%
\begin{equation}\label{i1}
dX_{t}=\sum_{j=1}^{d}\sigma _{j}(X_{t})dW_{t}^{j}+b(X_{t})dt,
\end{equation}
where $W$ is a $d$-dimensional Brownian motion and $\sigma_j,b:\R^d \to \R^d$ are smooth vector fields.
Our aim is to construct an approximation scheme for the semigroup $P_{t}f(x)=\E(f(X_{t}(x)))$, where $X_{t}(x)$ is the diffusion process starting
from $x$. Given the time horizon $T>0$ and the time step $h=\frac{T}{n}$ one
constructs the Euler scheme of step $h$ by%
\[
X_{(k+1)h}^{n}=X_{kh}^{n}+\sum_{j=1}^{d}\sigma
_{j}(X_{k}^{n})(W_{(k+1)h}^{j}-W_{kh}^{j})+b(X_{k}^{n})h
\]%
Then one constructs the approximation semigroup $%
P_{t}^{n}f(x)=\E(f(X_{kh}^{n}(x)))$ for $kh\leq t<(k+1)h$. It is well
known that%
\begin{equation}
\left\vert P_{T}f(x)-P_{T}^{n}f(x)\right\vert \leq \frac{C}{n}\left\Vert
f\right\Vert _{4,\infty }  \label{I0}
\end{equation}%
where $\left\Vert f\right\Vert _{4,\infty }$ is the supremum norm of $f$ and
its derivatives up to order four. The proof is based on Lindeberg method
(or Duhamel's principle):%
\begin{equation}
P_{T}f(x)-P_{T}^{n}f(x)=%
\sum_{k=0}^{n-1}P_{[n-(k+1)]h}(P_{h}-P_{h}^{n})P_{kh}^{n}f(x).  \label{I1}
\end{equation}%
Since%
\begin{equation}
\left\vert P_{h}f(x)-P_{h}^{n}f(x)\right\vert \leq C\left\Vert f\right\Vert
_{4,\infty }h^{2},  \label{I2}
\end{equation}%
the above inequality gives (\ref{I0}). If we want to go further we develop $P_{T-(k+1)h}=P_{[n-(k+1)]h}$ as well and we obtain%
\begin{eqnarray}
P_{T}f(x)
&=&P_{T}^{n}f(x)+%
\sum_{k=0}^{n-1}P_{[n-(k+1)]h}^{n}(P_{h}-P_{h}^{n})P_{kh}^{n}f(x)  \label{I3} \\
&&+%
\sum_{k_{1}<k_{2}<n}P_{[n-(k_{2}+1)]h}(P_{h}-P_{h}^{n})P_{(k_{2}-k_{1}-1)h}^{n}(P_{h}-P_{h}^{n})P_{k_{1}h}^{n}f(x).
\nonumber
\end{eqnarray}%
The last term is of order $n^{-2}$, which gives an error of order two if we only keep the two first terms. And one
may continue and go further in the development: one develops $%
P_{T-(k_{2}+1)n}$ and so on. This is similar to the development made in the
classical parametrix method. But now a problem appears: how to compute $%
P_{h}-P_{h}^{n}$  in the second term? In the classical parametrix method, one uses an
integration by parts formula based on the infinitesimal operator of the
diffusion semigroup. Here we follow another way: we develop $P_{h}-P_{h}^{n}$
itself in the same way as for $P_{t}-P_{t}^{n}$ in order to improve the
order of approximation. So, in our approach we have two simultaneous
developments: an "horizontal" one as in (\ref{I3}) and a "vertical" one which
is used in order to refine $P_{h}-P_{h}^{n}.$ And in both cases we continue
the development up to the moment that we have obtained the order of
approximation~$\nu \in \N^*$ that we desire. The control of this two folds "Taylor
expansion" gives rise to a rather intricate combinatorial problem. The
natural way to describe this is to use some trees which are constructed by
backward recurrence, which is a little bit tricky. A second problem
concerns the computation of the sum $\sum_{k=0}^{n-1}$ and more generally of sums of the form $\sum_{0\leq k_{1}<...<k_{m}<n}$. Computing all these terms would make the use of the development~\eqref{I3} inefficient for computational purposes. 
The idea is then to randomize $0\leq k_{1}<...<k_{m}<n$ by using order
statistics and then to use the Monte Carlo method in order to compute it.
This is the reason for which random grids come on in our schemes.

These developments are made in Sections~\ref{Sec_BD} and~\ref{Sec_TO}. Eventually, we show that for any order $\nu\in \N^*$, there exists $r\in \N^*$, coefficients $c_1,\dots, c_r\in \R$ and random grids $\Pi^{i}_\nu(\omega )\subset \{jT/n^{l}:j\leq n^{l}\},i=1,...,r$ with  $\card(\Pi _{i})\leq C_{\nu }\times n$ for some $C_\nu>0$ such that
$$P_Tf=\sum_{i=1}^r c_i \E[P^{\Pi^i_\nu}_Tf]+O(n^{-\nu}).$$
This is our first main result, stated precisely in Theorem~\ref{MAIN}, where we give an explicit construction of the coefficients~$c_i$ and of the time-grids~$\Pi^i_\nu$. Thus, the complexity of our algorithm remains of order $r\times C_\nu \times n$, for any order~$\nu$ of precision. However, $r\times C_\nu$ seriously
increases with $\nu$, and one has to take care about this in the complexity analysis for the choice of $\nu$.

To use the approximation in practice, one has to work with a probabilistic representation of the semigroup. On the discretization time grid $\Pi=\{t_0=0<t_1<\dots<t_m=T\}$, we define the corresponding Euler scheme
by $X_{0}^{\Pi }=x$ and 
\begin{equation}
X_{t_{i+1}}^{\Pi }=X_{t_{i}}^{\Pi }+\sum_{j=1}^d \sigma_j (X_{t_{i}}^{\Pi
})(W^j_{t_{i+1}}-W^j_{t_{i}})+b(X_{t_{i}}^{\Pi })(t_{i+1}-t_{i}).
\label{i2}
\end{equation}%
Since the grids are independent from~$W$, we then have for smooth functions~$f$
\begin{equation*}
\E[f(X_{T})]=\sum_{i=1}^{r}c_{i}\E[f(X_{T}^{\Pi^{i}_\nu})]+O(n^{-\nu }),
\end{equation*}%
and the right hand side gives an estimator that can be computed in~$O(n)$ operations. Then, an important issue is the variance of this estimator. In Section~\ref{Sec_Prob_Repr}, we present a specific organization of the algorithm which allows to get a finite variance. Theorem~\ref{MAIN} proposes a particular way to gather the terms, i.e. a partition $\cI_1,\dots,\cI_q$ of $\{1,\dots,r\}$, and Theorem~\ref{varience} shows that the variance of $\sum_{i\in \cI_{q'}} c_i f(X_{T}^{\Pi^{i}_\nu})$ is bounded for all $q'$, so that the variance of the estimator is bounded.

An important and nice feature of our approach is that it is generic and provides an algorithm that can be used in many contexts.  Indeed, in the previous approach
the only fact which is necessary in order to make the algorithm work is to
have at hand a short time approximation $P_{h}^{n}$ for $P_{h}$ such that (%
\ref{I2}) holds. The construction of the grids $\Pi_i$ and the coefficients $c_i$ only depends on this.  This leads us to consider the following abstract framework.
Let $F$ be a vector space endowed with a family of seminorms $\|\|_k$, $k\in \N$, such that $\|f\|_k\le \|f\|_{k+1}$. We consider a family $(P_t,t\ge 0)$ of linear operators on~$F$ that have the semigroup property, i.e $P_0f=f$ and $P_{t+s}f=P_tP_sf$ for all $f\in F$, $t,s\ge 0$. Our goal is to approximate the semigroup $P_Tf$ and build, for any $\nu \in \N^*$ a linear operator $\hat{P}^{\nu,n}_T$ such that
$$\exists C>0,k \in \N^*,  \forall f \in F, \  \|P_Tf- \hat{P}^{\nu,n}_Tf\|_0 \le C\|f\|_k n^{-\nu}.$$
To achieve this goal, we suppose that we have at our hands a family of linear operators $Q_l:F\to F$, $l\in \N$, such that we have for some $\alpha>0$ and $\beta \in \N$,
\begin{equation}\tag{$\overline{H_{1}}$}
  \forall l,k\in \N, \exists C>0,\forall f \in F,\  \left\Vert (P_{h_{l}}-Q_{l})f\right\Vert _{k}\leq
C \left\Vert f\right\Vert _{k+\beta }h_{l }^{1+\alpha },
\label{barH1}
\end{equation}
where $h_l=T/n^l$. We note $Q_{l}^{[0]}$ the identity operator on~$F$ and, for $k\in \N^*$, $Q_{l}^{[k]}=Q_{l}^{[k-1]}Q_{l}$ the operator obtained by applying $k$ times the operator~$Q_l$. We also assume that all these operators satisfy
\begin{equation}\tag{$\overline{H_{2}}$}
\forall l,m \in \N, \exists C>0, \ \max_{0\le k \leq n^l}\Vert Q_l^{[k]}f\Vert _{m}+\sup_{t\leq T}\Vert P_{t}f\Vert _{m }\leq C\Vert f\Vert _{m}.  \label{barH2}
\end{equation}
 The Euler scheme discussed before corresponds to $Q_l=P_{h_l}^{n^l}$ with $h_l=T/n^l$, and the approximation of order~$2$ only involves $Q_1$.
But, if we want to construct higher order schemes as in (\ref{I3}), we have
to mix operators $Q_{l}$ for $l\in \N^*$. This leads us to consider
grids $\Pi =\{0\leq t_{1}<...<t_{m}=T\}$ with the property that for every $
i=1,...,n,$ we have $t_{i}-t_{i-1}=h_{l_{i}}$ for some $l_{i}\in \N$. Then
we define $P_{T}^{\Pi }=Q_{l_{n}}Q_{l_{n-1}}...Q_{l_{1}}$. Notice that $
P_{T}^{\Pi }$ is built by using the "short time" approximation operators $
Q_{l},l\in \N$ only. Eventually, we show (see Theorem~\ref{MAIN}) that for any order $\nu\in \N^*$, there exists coefficients $c_1,\dots, c_r\in \R$ and random grids $\Pi^{i}_\nu(\omega )\subset \{jT/n^{l}:j\leq n^{l}\},i=1,...,r$ with  $\card(\Pi _{i})\leq C_{\nu }\times n$ for some $C_\nu>0$ and constants $C>0$ and $k\in \N$ such that
$$\forall f \in F, \ \left\|P_Tf-\sum_{i=1}^r c_i \E[P^{\Pi^i_\nu}_Tf]\right\|_0\le C\|f\|_{k}n^{-\nu}.$$
 We
stress that the coefficients $c_{i}$ and the grids $\Pi _{i}(\omega),i=1,..,r$ does not depend on $P_{t}$ nor on the specific form of $Q_{l}$: only the order of approximation $\nu$ and $\alpha$ in (\ref{barH1}) matter. Then, we give several examples of applications besides the Euler scheme: the Ninomiya Victoir scheme for diffusion processes (then $\alpha =2$ and $\beta =6$), or approximation schemes for
ordinary differential equations and piecewise deterministic Markov processes.\\

The approximations introduced in this paper are of any order $\nu$ with a computation time in $O(n)$. To calculate then $P_Tf$ with a precision $\varepsilon$, we naturally use a Monte-Carlo method with $n\sim \varepsilon^{-1/\nu}$ and $M\sim\varepsilon^{-2}$ samples, which has a computational cost of $O(\varepsilon^{-(2+1/\nu)})$. Since $\nu$ is arbitrary large, we will denote by $O(\varepsilon^{-2+})$ this complexity.  There is a large literature in numerical probability dedicated to construct either unbiased estimators of~$P_Tf$, leading then to a computational cost of $O(\varepsilon^{-2})$ (but this is only true in the case of finite variance), or approximated estimators leading to a computational cost of $O(\varepsilon^{-2+})$. Let us give an overview of the different methods to position our work.

When $Q_1^{[n]}f=P_Tf+c_1n^{-1}+\dots+c_\nu n^{1-\nu} +O(n^{-\nu})$, the Richardson-Romberg extrapolation provides an approximation of order~$\nu$, and Pagès~\cite{Pages} shows in the case of the Euler scheme for SDEs how to get with this method an estimator with bounded variance. In a different way, extending Fujiwara's method, Oshima et al.~\cite{OTV} propose approximations for SDEs of any order by considering linear combinations of Ninomiya and Victoir schemes with different time steps. These approximations have very similar properties to the ones presented in this paper, but they are obtained with a significantly different approach: they are constructed with linear combinations of schemes using uniform grids obtained with multiples of the same time-step, while our approximations uses non-uniform time grids that are refined at some random places. Also, the principle of our methodology is not to find a combination of schemes that cancels the terms of orders $n^{-i}$ for $i=1,\dots,\nu-1$, but instead to calculate the contribution of all these terms.

The Multi-Level Monte-Carlo (MLMC) method proposed by Giles~\cite{Giles} that generalizes the statistical Romberg method of Kebaier~\cite{Kebaier} gives another generic way to approximate $P_Tf$ in $O(\varepsilon^{-2+})$.  McLeish~\cite{McLeish} and Rhee and Glynn~\cite{RhGl} have then proposed an unbiased estimator constructed with similar ideas, see also the recent work of Vihola~\cite{Vihola}. Contrary to the previous approaches, the MLMC method does not rely on the development of high order approximations since it already works using the Euler scheme. It stems from a clever probabilistic representation and variance analysis. The MLMC method is in fact complementary to high order approximations. For instance, Lemaire and Pagès~\cite{LePa} have proposed estimators combining the MLMC method and the Richardson-Romberg extrapolation, improving the asymptotic complexity of the standard MLMC method with the Euler scheme.

Last, there is a stream of papers that develop unbiased estimators for $P_Tf$ in the case of SDEs. We have already mentioned the unbiased estimators \cite{McLeish,RhGl} that are obtained as telescopic series and that use a discretization scheme with more and more refined time grids. Another direction of research is to try to write $P_Tf=\E[W_Tf(\tilde{X}_T)]$, where $\tilde{X}_T$ is a simulatable process (e.g. a Euler scheme) and $W_T$ is some computable weight. By using a change of measure and a rejection algorithm, Beskos and Roberts~\cite{BeRo} have proposed such a method for one-dimensional diffusions. Recently, Bally and Kohatsu-Higa~\cite{BK-H} have given a probabilistic representation of the parametrix method that opens the road to construct unbiased estimators for a wide class of Markov processes, including stopped or reflected diffusions~\cite{FKHL,AHKH}. By using a different approach, Henry-Labordère et al.~\cite{HLTT} have lately proposed unbiased estimators for SDEs that present nonetheless a similar structure as the ones obtained with the parametrix method. A common important issue with all these unbiased estimators is to come up with a bounded variance estimator. This problem is tackled by Andersson and Kohatsu-Higa~\cite{AnKH} who provide a finite variance estimator for the parametrix method, see also Agarwal and Gobet~\cite{AgGo}. The approximation method that we develop in this paper can be seen somehow as a discrete version of the parametrix method. Instead of considering a continuous time approximating semigroup  $P^x_t$, and iterate indefinitely the formula $P_Tf-P^x_Tf=\int_0^TP^x_t(L-L^x)P_sf ds$ ($L$ and $L^x$ are the corresponding infinitesimal generators), we iterate the equality $P_Tf- Q_1^{[n]} =\sum_{k=0}^{n-1}P_{(n-(k+1))h_1}(P_{h_1}-Q_1) Q_1^{[k]}$ a finite number of times until to achieve an approximation of order $\nu$. The main advantage of our approximation schemes is that their construction is generic and only depends on the parameter~$\alpha$ in~\eqref{barH1}, while the weights involved in these  unbiased estimators really depends on the underlying SDE or Markov process. This makes our approach much easier to implement for an whole class of processes. Besides, the discrete structure enables us to gather the correcting terms in a way to get a finite variance estimator as already mentioned.

The paper is organized as follows. In Section~\ref{Sec_BD}, we introduce some notation and present the recipe to construct iteratively high-order approximation schemes. Section~\ref{Sec_TO} introduces trees,  random trees and random grids that we use to construct our approximation schemes. It also prepares the variance analysis by gathering the terms of the approximations in an appropriate way. Theorem~\ref{MAIN} states our first main result. Section~\ref{Sec_Prob_Repr} specify these approximations in some cases by using particular probabilistic representations of semigroups, for instance in the case of the Euler scheme for SDEs. In this case, we state in Theorem~\ref{varience} our second main result that ensures that our estimators have a finite variance. Last, we provide in Section~\ref{Sec_Num} numerical examples of our approximations that illustrates the broad application of our approach.

\bigskip

\section{Basic development}\label{Sec_BD}

We first introduce notation that will be used through the paper. We denote by $C_{b}^{\infty }(\R^{d})$ the space
of smooth functions from $\R^{d}$ to $\R^{d}$ which are bounded and have
bounded derivatives of any order. And we work with the norms 
\begin{equation}\label{def_multiindex}
\left\Vert f\right\Vert _{k,\infty }=\sum_{0\leq \left\vert \gamma
\right\vert \leq k}\sup_{x\in \R^{d}}\left\vert \partial ^{\gamma
}f(x)\right\vert
\end{equation}%
where for a multi-index $\gamma =(\gamma _{1},...,\gamma _{m})\in \cup_{m'\in \N}\{1,...,d\}^{m'}$ we denote $$\left\vert \gamma \right\vert =m \text{ and } \partial
^{\gamma }=\partial _{x_{\gamma _{1}}}...\partial _{x_{\gamma _{m}}}.$$
In many proofs of the paper, we will have to deal with derivatives of composed functions. Let $f:\R^d \rightarrow \R$ and $g:\R^{d}\rightarrow \R^{d}$ be smooth functions. We note $g^j$ with $j\in \{1,\dots,d\}$ the coordinates of $g$. Then, one may prove by recurrence that 
\begin{equation}
\partial ^{\alpha }[f\circ g]=\sum_{\left\vert \beta \right\vert \leq
\left\vert \alpha \right\vert }(\partial ^{\beta }f)(g)P_{\alpha ,\beta }(g)
\label{APP3}
\end{equation}%
with%
\begin{equation}
  P_{\alpha ,\beta }(g)=\sum c_{\alpha ,\beta }((\gamma_{1},j_1),\dots,( \gamma_{k},j_k)) \prod_{i=1}^{k}\partial ^{\gamma _{i}}g^{j_i} , \label{APP4}
\end{equation}%
where the sum is over all $k=1,\dots,|\alpha|$, $j_1,\dots,j_k\in \{1,\dots,d\}$ and (non void) multi-indices $\gamma_1,\dots,\gamma_k \in \cup_{m\ge 1} \{1,\dots,d\}^m$ such that $ \sum_{i=1}^k |\gamma_i|\le |\alpha|$.  For $f:\R^d\rightarrow \R^{d}$, we note $\partial ^{\alpha }[f\circ g]:=(\partial ^{\alpha }[f^1\circ g],\dots,\partial ^{\alpha }[f^d\circ g])$, and the same formula applies coordinate by coordinate. This result is known in the literature as the Faà di Bruno's formula, but we do not need in this work to use the explicit formula for the coefficients $c_{\alpha,\beta}$, see Constantine and Savits~\cite{CoSa}.

We consider a semigroup of linear operators $P_{t}:C_{b}^{\infty
}(\R^{d})\rightarrow C_{b}^{\infty }(\R^{d})$ which satisfies%
\begin{equation*}
P_{t+s}=P_{t}P_{s}.
\end{equation*}
Let $T>0$ be a time horizon which is fixed in the sequel. We are interested
in building approximation schemes for $P_{T}.$ For $n\in \mathbb{N}^*$
and $l\in \mathbb{N}$, we define 
\begin{equation*}
h_{l }=\frac{T}{n^{l }},\qquad h=h_{1}=\frac{T}{n},\quad h_{0}=T.
\end{equation*}
We suppose that we are given a sequence of linear operators $Q_{l}:C_{b}^{\infty }(\R^{d})\rightarrow C_{b}^{\infty }(\R^{d}),l\in \N$ which will be used
in order to construct our approximation schemes. The operator $Q_{l}$ is supposed to be an
approximation of $P_{h_{l}}$, more precisely we assume that for every $k\in
\N $ and $l\in \N$%
\begin{equation}\tag{$H_{1}$}
  \forall l,k\in \N, \exists C>0,\forall f \in C^\infty_b(\R^d),\  \left\Vert (P_{h_{l}}-Q_{l})f\right\Vert _{k,\infty }\leq
C \left\Vert f\right\Vert _{k+\beta ,\infty }h_{l }^{1+\alpha }
\label{b00}
\end{equation}%
for some $\alpha >0,\beta \in \mathbb{N}.$ In the case of the Euler scheme we
have $\alpha =1,$ and $\beta =4$ (see Example~\ref{Example_Euler} below). We denote%
\begin{equation*}
\Delta _{h_{l}}=P_{h_{l}}-Q_{l}
\end{equation*}%
and%
\begin{equation*}
P_{h_{l}}^{h_{l}}=Q_{l}\quad \text{ and }\quad P_{kh_{l}}^{h_{l}}=Q_{l}...Q_{l}\quad k%
\text{ times.}
\end{equation*}%
Thus, we produce a discrete semigroup $P_{t}^{h_{l}},t=kh_{l}.$
We will use the following regularity hypothesis: 
\begin{equation}\tag{$H_{2}$}
\forall l,m \in \N, \exists C>0, \ \max_{kh_{l}\leq T}\Vert P_{kh_{l}}^{h_{l}}f\Vert _{m,\infty
}+\sup_{t\leq T}\Vert P_{t}f\Vert _{m,\infty }\leq C\Vert f\Vert _{m,\infty
}.  \label{b000}
\end{equation}
\begin{remark} We could more generally assume that the left hand-side of~\eqref{b000} is upper bounded by $ C\Vert f\Vert_{m+\tilde{\beta},\infty}$ for some $\tilde\beta \in \N$: this would not modify the main results of Sections~\ref{Sec_BD} and~\ref{Sec_TO}. However, this generalization is not relevant for usual semigroups that already satisfy this bound for $\tilde{\beta}=0$. For simplicity, we only consider this case.
\end{remark}
\begin{example}\label{Example_Euler}(Euler scheme for diffusion processes) We work with the
diffusion process~(\ref{i1}). We assume that  $\sigma_j,b \in C^{\infty }(\R^{d})$ and the derivatives of any order of $\sigma_j$ and $b$ are bounded. In particular they have linear growth. By standard results on stochastic flows (see Proposition 2.1 and Theorem 2.3 of~\cite{IW}, Chapter 5), we have~\eqref{b000}. 

We denote by $P_{t}$ the semigroup a diffusion process~\eqref{i1} and $P_{h}^{h}f(x):=\E[f(x+b(x)h+\sigma(x)W_h)]$, $P_{kh}^{h}:=(P_{h}^{h})^{k}$ the (discrete) semigroup of the Euler scheme
of step $h$. Then
\begin{equation*}
\Delta_{h}f(x)=P_{h}f(x)-P_{h}^{h}f(x)=\int_{0}^{h}%
\int_{0}^{s}\E((L^{2}f)(X_{r}(x)) - \E((L_x^{2}f)(x+b(x)r+\sigma(x)W_r))drds,
\end{equation*}%
where $L$ is the infinitesimal operator of the semigroup $P_{t}$ and $L_x$ is the semigroup corresponding to the Euler scheme, with frozen coefficients $b(x)$ and $\sigma(x)$.  By Theorem~4.4~\cite{KunitaSF}, we can take a modification of the solution such that the flow $x\rightarrow X_{t}(x)$ is infinitely differentiable  with derivatives which have finite moments of any order. Thus
we have
\begin{equation}
\forall k \in \N, \exists C>0, \ \left\Vert \Delta _{h}f\right\Vert _{k,\infty }\leq C\left\Vert f\right\Vert
_{k+4,\infty }h^{2}  \label{b0}
\end{equation}
Thus, the property~\eqref{b00} is satisfied with $\alpha =1,\beta =4.$
\end{example}

We come back to the general case and we present the basic decomposition that
we will use. We use the linearity of the operators in order to get%
\begin{eqnarray*}
P_{T}-P_{T}^{h_{1}}
&=&P_{nh_{1}}-P_{nh_{1}}^{h_{1}}=%
\sum_{k=0}^{n-1}P_{(n-k)h_{1}}P_{kh_{1}}^{h_{1}}-P_{(n-(k+1))h_{1}}P_{(k+1)h_{1}}^{h_{1}}
\\
&=&\sum_{k=0}^{n-1}P_{(n-(k+1))h_{1}}\Delta _{h_{1}}P_{kh_{1}}^{h_{1}}.
\end{eqnarray*}%
Iterating this equality, we get for every $1\leq m\leq n$, 
\begin{equation}
P_{T}=P_{T}^{h_{1}}+\sum_{i=1}^{m-1}I_{i}^{h_{1}}(n)+R_{m}^{h_{1}}(n)
\label{b1}
\end{equation}%
with (convention $k_{0}=-1$ and $\prod_{j=0}^{i-1}A_{j}=A_{i-1}\dots A_{0}$
for non commutative operators $A_{j}$)%
\begin{eqnarray*}
I_{i}^{h}(n) &=&\sum_{0\leq
k_{1}<...<k_{i}<n}P_{(n-(k_{i}+1))h}^{h}\prod_{j=0}^{i-1}(\Delta
_{h}P_{(k_{j+1}-k_{j}-1)h}^{h}) \\
R_{m}^{h}(n) &=&\sum_{0\leq
k_{1}<...<k_{m}<n}P_{(n-(k_{m}+1))h}\prod_{j=0}^{m-1}(\Delta_{h}P_{(k_{j+1}-k_{j}-1)h}^{h})
\end{eqnarray*}%
Then, using (\ref{b00}) and~\eqref{b000} we get for $h\in\{h_l,l\in\N\}$,
\begin{eqnarray}
\left\Vert R_{m}^{h}(n)f\right\Vert _{\infty } &\leq &C^{m}\left\Vert
f\right\Vert _{\beta m,\infty }h^{(1+\alpha )m}\times \binom{n}{m}\leq \frac{%
C^{m}\left\Vert f\right\Vert _{\beta m,\infty }}{m!}h^{(1+\alpha )m}n^{m}.
\label{b2_1} 
\end{eqnarray}%
Thus, we get an error of order $O(h^{\alpha m})=O(n^{-\alpha m})$ for $%
h=h_{1}=T/n$.

Formula (\ref{b1}) represents a discretization with step $h_{1}>0$ on the
interval $[0,T]=[0,h_{0}]$. In the sequel we will use similar developments
on intervals $[0,h_{l}]$ with step $h_{l+1}$. So, using the above
formula with $T=h_{l}$ and with step $h_{l+1}$ instead of $h=h_{1}$,
we obtain 
\begin{equation}
P_{h_{l }}=P_{h_{l }}^{h_{l +1}}+\sum_{i=1}^{m-1}I_{i}^{h_{l
+1}}(n)+R_{m}^{h_{l +1}}(n),  \label{b1_bis}
\end{equation}%
with $l \in \mathbb{N}$. We then have from~\eqref{b2_1} with $h=h_{l
+1}=Tn^{-(l +1)}$, 
\begin{equation}
\left\Vert R_{m}^{h_{l +1}}(n)f\right\Vert _{\infty }\leq \frac{%
C^{m}\left\Vert f\right\Vert _{\beta m,\infty }T^{(1+\alpha )m}}{m!}\frac{1}{%
n^{((1+\alpha )l +\alpha )m}}.  \label{b2_2}
\end{equation}%
Similarly, we get%
\begin{equation}
\left\Vert I_{i}^{h_{l +1}}(n)f\right\Vert _{\infty }\leq \frac{%
C^{m}\left\Vert f\right\Vert _{\beta i,\infty }T^{(1+\alpha )i}}{i!}\frac{1}{%
n^{((1+\alpha )l +\alpha )i}}.  \label{b2_3}
\end{equation}
Formula~\eqref{b1} is appealing since it may lead to an approximation of order $O(n^{-\alpha m})$. The natural question is then how to simulate the terms $I_{i}^{h}f(n)$. This raises two problems that we explain now.

\textbf{Problem 1}. It seems cumbersome (time consuming of complexity $O(n^{i})$) to compute the sum defining~$I_i^h(n)$. To avoid this issue,  we will
use a randomization procedure (inspired from~\cite{BK-H} in the framework of the
parametrix method). We fix $i$ and we consider a random variable $\kappa(\omega )=(\kappa _{1}(\omega ),...,\kappa _{i}(\omega ))$ that follows a "discrete order statistics" on $\{0,1,...,n-1\}.$ Precisely, $\kappa$ follows the distribution
\begin{equation*}
\mu _{i}(dk_{1},...,dk_{i})=\frac{1}{\binom{n}{i}}\sum_{0\leq \kappa
_{1}<...<\kappa _{i}<n}\delta _{(\kappa _{1},...,\kappa
_{i})}(dk_{1},...,dk_{i}).
\end{equation*}%
We will use the notation 
\begin{equation*}
\kappa _{i}^{\prime }=\kappa _{i}+1\text{ for }i\geq 1\text{ and }\kappa
_{0}^{\prime }=0.
\end{equation*}%
Then 
\begin{eqnarray}
I_{i}^{h}(n) &=&\sum_{k_{1}<...<k_{i}<n}P_{(n-k_{i}^{\prime
})h}^{h}\prod_{j=1}^{i}(\Delta _{h}P_{(k_{j}-k_{j-1}^{\prime })h}^{h})
\label{b3} \\
&=&\binom{n}{i}\times \E_{\mu _{i}}\left( P_{(n-\kappa _{i}^{\prime
})h}^{h}\prod_{j=1}^{i}(\Delta _{h}P_{(\kappa _{j}-\kappa _{j-1}^{\prime
})h}^{h})\right)  \label{b3'}
\end{eqnarray}

\begin{remark}
  If we look to the equality between the terms in (\ref{b3}) and in~(\ref{b3'}), we see that (\ref{b3'}) gives a way to compute the sum which appears in (\ref{b3}) by the Monte-Carlo method.
  This Monte Carlo avoids the ``curse of dimensionality'' since the discrete simplex of dimension~$i$, $\{0\le k_{1}<...<k_{i}<n\}$, has $O(n^i)$ elements. Besides, the different terms in the sum~(\ref{b3}) have values that may be very close each other leading to a bounded variance. This will be analyzed later on in Subsection~\ref{subsec_variance} for SDEs and the Euler scheme.

  Let us note that this randomization makes the approximation~\eqref{b1} effective. Otherwise, it would have a computational cost of $O(n+\dots+n^{m-1})=O(n^m)$ for a precision in $O(n^{-\alpha m})$ (see~\eqref{b2_1}), exactly as $P^{h_m}_{n^mh_m}$.
  
\end{remark}

\textbf{Problem 2}. The basic element in the above formula is $\Delta_{h},$
and we are not able to simulate directly this quantity, due to $P_{h}$.
To overcome this problem, we will use the fact that the short-time estimate~(\ref{b00}) of the semigroup is more and more precise when $l$ increases since $\alpha>0$: 
\begin{equation*}
\Vert P_{h_{l }}^{h_{l }}-P_{h_{l }}\Vert _{\infty }\leq C\Vert
f\Vert _{\beta ,\infty }h_{l }^{1+\alpha }=O(n^{-(1+\alpha )l }).
\end{equation*}%
Thus, we will construct by backward induction on $l$, some
approximations only based the approximation kernels $Q_{l}$, each of them
involving at most $O(n)$ iterations of these approximations (so we keep a
complexity of order $n$). This is precised by the following lemma, which is
the core of our computations. We will use the following numbers: for $l\in 
\mathbb{N},i\in \mathbb{N}^{\ast }$ and $\nu \in \mathbb{N}^{\ast }$ we
define 
\begin{eqnarray}
q_{i}(l ,\nu ) &=&\nu+ \lceil i-(1+\alpha )(l +1)(i-1)\rceil ,
\label{n1} \\
m(l,\nu ) &=&\lceil \frac{\nu }{(1+\alpha )l+\alpha }\rceil  \label{n1'}
\end{eqnarray}
Here, $\lceil x\rceil =q$ if $x\in (q-1,q]$ is the ceiling function. We
observe that $i\in \mathbb{N}^{\ast }\mapsto q_{i}(l ,\nu )$ is
nonincreasing and therefore $q_{i}(l ,\nu )\leq q_{1}(l ,\nu )=\nu +1$.

\begin{lemma}
\label{lem_rec_approx} Let $l \in \mathbb{N}$ and $\nu _{0}\in \mathbb{N}^{\ast }$. Suppose that we have already a sequence of operators $\hat{P}_{h_{l +1}}^{\nu }$ for $1\leq \nu \leq \nu _{0}+1$ such that 
\begin{equation}
\Vert (\hat{P}_{h_{l +1}}^{\nu }-P_{h_{l +1}})f\Vert _{\infty }\leq
C_{l +1,\nu }\Vert f\Vert _{k(l +1,\nu ),\infty }n^{-\nu },
\label{hyp_rec_err}
\end{equation}%
for some $C_{l +1,\nu }>0$ and $k(l+1,\nu)\in \N$. For $\nu \leq \nu _{0}$, we define 
\begin{equation}
\hat{P}_{h_{l }}^{\nu }=P_{h_{l }}^{h_{l +1}}+\sum_{i=1}^{m(l,\nu
)-1}I_{i}^{\nu ,h_{l +1}}(n),  \label{induction_hatPl}
\end{equation}%
with 
\begin{equation}
I_{i}^{\nu ,h_{l +1}}(n)=\binom{n}{i}\times \E_{\mu _{i}}\left(
P_{(n-\kappa _{i}^{\prime })h_{l +1}}^{h_{l
+1}}\prod_{j=0}^{i-1}\left( (\hat{P}_{h_{l +1}}^{q_{i}(l ,\nu
)}-P_{h_{l +1}}^{h_{l +1}})P_{(\kappa _{j+1}-\kappa _{j}^{\prime
})h_{l +1}}^{h_{l +1}}\right) \right) .  \label{int}
\end{equation}%
Then, we have 
\begin{equation}
\Vert (\hat{P}_{h_{l }}^{\nu }-P_{h_{l }})f\Vert _{\infty }\leq
C_{l ,\nu }\Vert f\Vert _{k(l ,\nu ),\infty }n^{-\nu },
\label{approx_ordre_nu_lem}
\end{equation}%
for some $C_{l ,\nu }>0$ and with
\begin{equation}
k(l ,\nu ) =\max (\beta m(l,\nu ),\max_{i=1}^{m(l,\nu )-1}i\times
k(l +1,q_{i}(l ,\nu ))).  \label{n2}
\end{equation}%
\end{lemma}

\bigskip

\begin{remark}\begin{enumerate}
    \item Compare the definition of $I_{i}^{\nu ,h_{l+1}}(n)$ with
the one of $I_{i}^{h_{l+1}}(n)$ in~(\ref{b3'}): one just replaces $\Delta_{h_{l+1}}=(P_{h_{l +1}}-P_{h_{l +1}}^{h_{l +1}})$ by $(\hat{P}_{h_{l +1}}^{q_{i}(l ,\nu )}-P_{h_{l +1}}^{h_{l +1}})$. So $P_{h_{l+1}}$ is replaced by $\hat{P}_{h_{l +1}}^{q_{i}(l ,\nu )}$,
which is supposed to be "computable".

\item Recall that for $l=0$, we have $h_{0}=\frac{T}{n^{0}}=T$. Thus so $\hat{P}_{h_{0}}^{\nu }=\hat{P}_{T}^{\nu }$ is an approximation of order $n^{-\nu }$ of $P_{h_{0}}=P_{T}$. This is what we want to obtain.

\item The inductive construction suggested by Lemma~\ref{lem_rec_approx} to get a $\nu$-th order scheme for $P_{h_l}$ is finite. See the construction of the tree $\mathcal{T}_{l}^{\nu }$ in (\ref{Tree}).
  \end{enumerate}
\end{remark}
\bigskip

\begin{proof}[Proof of Lemma~\ref{lem_rec_approx}] For $\nu \leq \alpha +(1+\alpha )l$, we have $m(l,\nu )=1$ so
that $\hat{P}_{h_{l }}^{\nu }=P_{h_{l }}^{h_{l +1}}.$ Using~%
\eqref{b2_2} with $m=1$, we obtain (\ref{approx_ordre_nu_lem}).

Suppose now that $\nu >\alpha +(1+\alpha )l$ so that we have $m(l,\nu )-1>0.$
We write%
\begin{eqnarray*}
\hat{P}_{h_{l }}^{\nu } &=&P_{h_{l }}^{h_{l
+1}}+\sum_{i=1}^{m(l,\nu )-1}I_{i}^{h_{l +1}}(n)+\sum_{i=1}^{m(l,\nu
)-1}(I_{i}^{\nu ,h_{l +1}}(n)-I_{i}^{h_{l +1}}(n)) \\
&=&P_{h_{l }}-R_{m(l,\nu )}^{h_{l +1}}(n)+\sum_{i=1}^{m(l,\nu
)-1}(I_{i}^{\nu ,h_{l +1}}(n)-I_{i}^{h_{l +1}}(n)).
\end{eqnarray*}%
First we compare $I_{i}^{h_{l +1}}(n)$ and $I_{i}^{\nu ,h_{l +1}}(n)$.
From (\ref{int}) we have 
\begin{equation*}
I_{i}^{\nu ,h_{l +1}}(n)=\binom{n}{i}\times \E_{\mu _{i}}\left(
P_{(n-\kappa _{i}^{\prime })h_{l +1}}^{h_{l +1}}\prod_{j=0}^{i-1}(\hat{%
P}_{h_{l +1}}^{q_{i}(l ,\nu )}-P_{h_{l +1}}+\Delta _{h_{l
+1}})P_{(\kappa _{j+1}-\kappa _{j}^{\prime })h_{l +1}}^{h_{l
+1}}\right) .
\end{equation*}%
We get by expanding (choose $j$ times $\hat{P}_{h_{l +1}}^{q_{i}(l
,\nu )}-P_{h_{l +1}}$ and $i-j$ times $\Delta _{h_{l +1}}$) and using (%
\ref{b00}) and~\eqref{hyp_rec_err} 
\begin{align*}
\Vert I_{i}^{\nu ,h_{l +1}}(n)f-I_{i}^{h_{l +1}}(n)f\Vert _{\infty }&
\leq C \binom{n}{i}\sum_{j=1}^{i}\binom{i}{j}\Vert f\Vert _{j\times
k(l +1,q_{i}(l ,\nu ))+\beta (i-j),\infty }n^{-jq_{i}(l ,\nu
)}h_{l +1}^{(1+\alpha )(i-j)} \\
& \leq C \frac{n^{i}}{i!}\Vert f\Vert _{k(l ,\nu ),\infty }\sum_{j=1}^{i}%
\binom{i}{j} \frac{T^{(1+\alpha )(i-j)}}{n^{jq_{i}(l ,\nu
)+(1+\alpha )(l +1)(i-j)}}
\end{align*}%
with $C$ depending on the constants $C_{l+1,q_{i}(l,\nu )}.$ We have to check that,
for every $j=1,...,i$%
\begin{equation*}
jq_{i}(l ,\nu )+(1+\alpha )(l +1)(i-j)-i\geq q_{i}(l ,\nu
)+(1+\alpha )(l +1)(i-1)-i \ge \nu .
\end{equation*}%
The first inequality is true if $q_{i}(l ,\nu )\geq (1+\alpha )(l+1)$ and thus if
$\nu+i-(1+\alpha)(l+1)(i-1)\ge (1+\alpha )(l+1)$ by using~\eqref{n1}. The last inequality is equivalent to
\begin{equation*}
i\leq \frac{\nu }{(1+\alpha )(l+1)-1},
\end{equation*}%
which holds since
\begin{equation*}
i\leq m(l,\nu )-1\leq \frac{\nu }{\alpha +(1+\alpha )l}=\frac{\nu }{%
(1+\alpha )(l+1)-1}.
\end{equation*}%
We obtain%
\begin{equation*}
\Vert I_{i}^{\nu ,h_{l +1}}(n)f-I_{i}^{h_{l +1}}(n)f\Vert _{\infty
}\leq C_{l,\nu }\left\Vert f\right\Vert _{k(l,\nu ),\infty }n^{-\nu }.
\end{equation*}%
We deal now with the remainder. Using (\ref{b2_2}) with $m=m(l,\nu )$ we
obtain 
\begin{equation*}
\left\Vert R_{m}^{h_{l +1}}(n)f\right\Vert _{\infty }\leq C\left\Vert
f\right\Vert _{\beta m,\infty }\frac{1}{n^{((1+\alpha )l +\alpha )m}}.
\end{equation*}%
Since $((1+\alpha )l +\alpha )m(l,\nu )\geq \nu $ the proof is completed. 
\end{proof}

\begin{remark}
Lemma~\ref{lem_rec_approx} gives a recursive way to construct approximation of order $\nu \in \N$. When $\alpha$ is not an integer, another natural choice may be to consider approximations of order $\alpha \nu$, with $\nu \in \N$. Of course, it is possible then to get an analogous recursive construction. 
\end{remark}

\section{High order approximations of semigroups}\label{Sec_TO}

Lemma~\ref{lem_rec_approx} gives the recipe to construct high order approximations of semigroups by induction: from high order approximations on a time step $h_{l+1}$, we produce  high order approximations on a time step $h_{l}$, we go on this construction to get  high order approximations for $T=h_{0}$. To describe precisely this construction, we need to introduce basic mathematical objects. In Subsection~\ref{subsec_trees}, we introduce "trees", "random trees"
and "random grids" that will be used to define suitably our approximation schemes. Then, Subsection~\ref{subsec_oper} presents a sequence of
abstract operators and the composition operations associated to some given
random tree. All these definitions are motivated by the approximation  schemes that we
describe in Subsection~\ref{Sec_tree_approx}, but for the moment we keep an
abstract framework because this allows a precise and simple presentation.

\subsection{Trees, random trees and random grids}\label{subsec_trees}

The approximations that we construct in this paper involve a quite intricate combinatorics. This can be understood from Lemma~\ref{lem_rec_approx}: an approximation of order $\nu$ at a level~$l$ is constructed from approximations of different orders at level~$l+1$. To describe this recursion, we will use trees, see~\eqref{D4} and~\eqref{approx_tree} thereafter. Then, to make this approximation more explicit and non-recursive, we will then use what we call random trees, i.e. trees labeled with particular random variables. In the case of SDE and the Euler scheme approximations, these random trees can be seen as a way to represent the random grid on which the Euler scheme is constructed. In this paper, we will use as much as possible the letter $\cT$ for trees and $\cA$ for random trees.

\subsubsection{Trees} We will use the Neveu notation~\cite{Neveu}. Let $\mathcal{U}=\cup_{n\geq 0}(\mathbb{N}^{\ast })^{n}$ be the set of finite
sequences of non-negative integers. For $u=(u_{1},...,u_{m})\in \mathcal{U}$
and $i\in \N$ we denote $iu=(i,u_{1},...,u_{m})$ and $ui=(u_{1},...,u_{m},i).$
We also denote $\left\vert u\right\vert =m$ the length of $u.$

\begin{definition}\label{def_tree}
\textbf{(Trees)} A tree is a subset $\mathcal{T}\subset \mathcal{U}$ such
that:
\begin{enumerate}
\item $\emptyset \in \mathcal{T}$,

\item $uj\in \mathcal{T\quad \Rightarrow \quad }u\in \mathcal{T}$,

\item $uj\in \mathcal{T\quad \Rightarrow \quad }ui\in \mathcal{T}$ for every 
$i<j.$
\end{enumerate}
\end{definition}

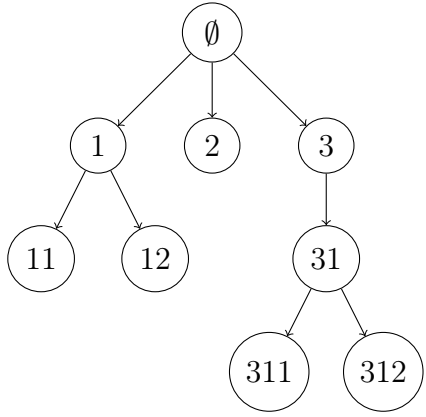
\begin{figure}[h]
  \centering
\begin{tikzpicture}[nodes={draw, circle}, ->] 
\node {$\emptyset$}
child {
  node {1}
  child {
    node {11}
  }
  child {
    node {12}
  }
}
child {
  node {2}
}
child {
  node {3}
   child {
    node {31}   
   child {
    node {311}
   } child {
    node {312}
   }
   }
};
\end{tikzpicture}
\caption{Example illustrating the Neveu notation for the tree $\{\emptyset ,1,2,3, 11,12,31,311,312\}$ (we note $312$ instead of $(3,1,2)$ when no confusion is possible).}
\end{figure}
\noindent {\bf Convention:} Throughout the paper, we use a different symbol for the ancestor (root) of a tree and for the void set: 
\begin{equation}\label{convention}
\emptyset =\text{Ancestor,\quad }\varnothing =\text{ void set.}
\end{equation}

We think to $\mathcal{T}$ as a genealogical tree: each $%
u=(u_{1},...,u_{m})\in \mathcal{T}$ represents an individual (we call it
also node or vertex) which is characterized by his genealogy: $u$ is the $
u_{m}$-th son of $(u_{1},...,u_{m-1}).$ So the first property
imposes that the root $\emptyset $ belong to the tree (it is the universal
ancestor), the second property says that any node of the tree (except the
root) has a father, and the third property imposes to number the sons of a
node increasingly, without jumping any number: put it otherwise, if a third
son exists, then a second one has to exist also. Last, let us mention that a tree $\cT$ can be infinite: throughout the paper, we will only consider finite trees. 

We introduce some more notation related to $\mathcal{T}$. For $u\in \cT$, we denote by $j_{u}(\mathcal{T})$ the number of sons of $u$ that is%
\begin{equation*}
j_{u}(\mathcal{T})=\max \{i \in \N^*:ui\in \mathcal{T}\} \text{ with  } \max\varnothing=0.
\end{equation*}%
We denote by $\mathcal{T}_{i}^{\prime }$ the sub tree of $\mathcal{T}$
rooted in $i$ that is $\mathcal{T}_{i}^{\prime }=\{u\in \cU :iu\in \mathcal{T}\}.$
We also denote $i\mathcal{T}=\{iu:u\in \mathcal{T}\}$.
This means that we root $\mathcal{T}$ at the point $i\in \N^*$. Notice that this is
not a tree because it does not contain $j\mathcal{T}$, for $j<i$, nor the ancestor. We note $
\left\vert \mathcal{T}\right\vert =\max \{\left\vert u\right\vert :u\in 
\mathcal{T}\}$, the depth of the tree $\mathcal{T}$. Finally we define the
extreme points (leaves) of $\mathcal{T}:$%
\begin{equation*}
\mathcal{E(T)}=\{u\in \mathcal{T}:j_{u}(\mathcal{T)}=0\}.
\end{equation*}

\subsubsection{Random trees}

Let $\cA$ be a finite tree and $n\in \N^*$ such that $n\ge \max_{u\in \cA}j_u(A)$. To every
vertex $u\in \mathcal{A} \setminus \mathcal{E}(\mathcal{A})$ (i.e. such  that $j_{u}(\mathcal{A})>0$) we associate a random variable 
\begin{equation*}
\kappa (u)=\{(\kappa _{1}(u),\dots,\kappa _{j_{u}(\mathcal{A})}(u)): 0\leq \kappa _{1}(u)<...<\kappa _{j_{u}(\mathcal{A})}(u)\leq
n-1\}\subset \{0,1,...,n-1\}^{j_{u}(\mathcal{A})}
\end{equation*}%
which we may considered as the (random) birthdays of the sons of $u$. We
denote%
\begin{equation*}
\kappa (\mathcal{A})=\{\kappa (u):u\in \mathcal{A}-\mathcal{E}(\mathcal{A}%
)\}.
\end{equation*}%
We assume :
\begin{itemize}
  \item The random variables $\kappa (u),u\in \mathcal{A}-\mathcal{E}(%
\mathcal{A})$ are independent each other.

\item $0\leq \kappa _{1}(u)<...<\kappa _{j_{u}(\mathcal{A})}(u)\leq n-1$
is an order statistics, i.e. they are uniformly distributed on $\{(k_1,\dots,k_{j_{u}(\mathcal{A})})\in\{0,1,...,n-1\}^{j_{u}(\mathcal{A})} : 0\le k_1<\dots<k_{j_{u}(\mathcal{A})} \}.$
\end{itemize}

\begin{definition}\label{def_random_tree}
\textbf{(Random trees)} If the above hypothesis are verified, we call $(\mathcal{A},\kappa (\mathcal{A}))$ a random tree.
\end{definition}

\begin{remark}\label{random_tree_heritage}
Let us precise the relation between the random trees $(\mathcal{A},\kappa (\mathcal{A}))$ and its subtrees $(\mathcal{A}_{j}^{\prime },\kappa^j (\mathcal{A}_{j}^{\prime }))$, $j=1,...,i$ with $i=j_\emptyset(\cA)$. For $u\in \mathcal{A}$ with $u=jv$
we will assume $\kappa (u)=\kappa^j (v)$ where $\kappa^j(v)$ is the random
variable associated to $v\in \mathcal{A}_{j}^{\prime }$. So, $\kappa (\mathcal{A})$ is the family of uniform random variables obtained from $\kappa^j(\mathcal{A}_{j}^{\prime }),j=1,...,i$
to which is added one independent random variable $\kappa(\emptyset)$ uniformly distributed on $\{(k_1,\dots,k_{j_{\emptyset}(\mathcal{A})})\in\{0,1,...,n-1\}^{j_{\emptyset}(\mathcal{A})} : 0\le k_1<\dots<k_{j_{\emptyset}(\mathcal{A})} \}$.
\end{remark}

\subsubsection{Random grids}

We associate to a random tree $(\mathcal{A},\kappa (\mathcal{A}))$ a random grid. We fix $l\in\N, n\in \N$ and $T\geq 0$ and we recall 
\begin{equation*}
h_{l}=\frac{T}{n^{l}},\quad h=h_{1}=\frac{T}{n},\quad h_{0}=T,
\end{equation*}%
and we use here (and in the sequel) the convention 
\begin{equation}
\kappa ^{\prime }=\kappa +1.  \label{conv}
\end{equation}%
Then, we construct by recurrence the random grid $G_{l}(\mathcal{A})$ on $[0,h_{l}]$ in the following way. For convenience, we drop in the notation  $G_{l}(\mathcal{A})$ the dependence in $\kappa(\cA)$ even if this grid is constructed by using $\kappa(\cA)$. 
If $\mathcal{A=\{\emptyset \}}$ (contains just the ancestor) then $j_{\emptyset }(\mathcal{A)}=0$ and we define
\begin{equation}
G_{l}(\mathcal{A})=G_{l}(\{ \emptyset \})=\{qh_{l+1},q=0,...,n\}.
\label{G1}
\end{equation}%
This is the usual uniform grid of step $h_{l+1}$ on $[0,h_{l}]$. Otherwise, we have
 $j_{\emptyset }(\mathcal{A})>0$ and we define by recurrence%
\begin{equation}
G_{l}(\mathcal{A})=\{qh_{l+1},q=0,...,n\}\cup \left( \cup
_{i=1}^{j_{\emptyset }(\mathcal{A)}}\{\kappa _{i}(\emptyset) h_{l+1}+G_{l+1}(\mathcal{A}_{i}^{\prime }) \}\right) ,  \label{G2}
\end{equation}%
where $t+G:=\{t+s, s\in G\}$. This means that we consider the uniform grid of step $h_{l+1}$ on $[0,h_{l}]$ and moreover we refine the intervals $[\kappa _{i}h_{l+1},\kappa_{i}^{\prime }h_{l+1}]$ according to the random grid $G_{l+1}(\mathcal{A}_{i}^{\prime })$ (see Remark~\ref{random_tree_heritage} to see how the random trees $\cA$ and $\cA'_i$ are related). Notice that, if $\left\vert \mathcal{A}\right\vert =r,$ then $G_{l}(\mathcal{A})\subset \{qh_{l+r+1},q=0,...,n^{r+1}\}.$ We denote $m=\card (G_{l}(\mathcal{A}))-1$ and we define
\begin{equation}
\Pi _{l}(\mathcal{A})=\{0=s_{0}<s_{1}<...<s_{m}=h_{l}\}\subset
\{qh_{l+r+1},q=0,...,n^{r+1}\}  \label{G3}
\end{equation}%
the reordering of $G_{l}(\mathcal{A})$. We notice that for every $k=1,....,m$
one has $s_{k}-s_{k-1}=h_{l+p_{k}}$ for some $p_{k}\in \{1,2,...,r\}.$
We finally give an alternative representation
of the random grid $G_{l}(\mathcal{A})$.

\begin{lemma}\label{lem_grids}
Let $(\cA,\kappa(\cA))$ be a random tree. Let us define $t_{l}(\emptyset )=0$ and for $u=(u_{1},...,u_{m})\in \cA$, we
define 
\begin{eqnarray*}
t_{l}(u) &=&\kappa _{u_{1}}(\emptyset )h_{l+1}+\kappa
_{u_{2}}(u_{1})h_{l+2}+...+\kappa _{u_{m}}(u_{1},...,u_{m-1})h_{l+m} \\
&=&t_{l}(u_{1},...,u_{m-1})+\kappa _{u_{m}}(u_{1},...,u_{m-1})h_{l+m}.
\end{eqnarray*}%
Then, we have 
\begin{equation*}
G_{l}(\mathcal{A})=\cup _{u\in \mathcal{A}}\{t_{l}(u)+kh_{l+\left\vert
u\right\vert +1},k=0,...,n\}.
\end{equation*}%
\end{lemma}
\begin{proof}
  We prove this result by recurrence on the depth~$|\cA|$. The result is clear for $\cA=\{\emptyset\}$. Suppose the result true for any random tree~$|\cA'|\le r-1$ and assume $|\cA|=r$. Let $i\in\{1,\dots,j_{\emptyset}(\cA)\}$ and $\kappa^i$ the random variables associated to $\cA'_i$ (see Remark~\ref{random_tree_heritage}). For $u'\in \cA'_i$, we have
  $$t_{l+1}(u')=\kappa^i_{u'_{1}}(\emptyset )h_{l+2}+...+\kappa^i_{u'_{m}}(u'_{1},...,u'_{m-1})h_{l+1+m}.$$
  We set $u=(i,u')$. Since $\kappa((i,u))=\kappa^i(u)$ for any $u\in \cA'_i$, we get
  $$t_l(u)=\kappa _{i}(\emptyset )h_{l+1}+t_{l+1}(u').$$
 From $\cA=\{\emptyset\}\cup\left(\cup_{i=1}^{j_\emptyset(\cA)} i\cA'_i \right)$, we deduce that
  \begin{align*}
&\cup _{u\in \mathcal{A}}\{t_{l}(u)+kh_{l+\left\vert
  u\right\vert +1},k=0,...,n\}\\
    &= \{qh_{l+1},q=0,...,n\} \cup\left(\cup_{i=1}^{j_\emptyset(\cA)} \cup _{u'\in \cA'_i} \{\kappa _{i}(\emptyset )h_{l+1}+t_{l+1}(u')+kh_{l+\left\vert u'\right\vert +2},k=0,...,n\}\right). 
  \end{align*}
  By using the recurrence hypothesis and~\eqref{G2}, this set is equal to $G_{l}(\mathcal{A})$. 
\end{proof}

\subsection{Operators}\label{subsec_oper}

We consider again  a sequence of operators $Q_{l}:C^\infty(\R^d) \rightarrow C^\infty(\R^d)$, $l\in \N$ (or more generally  $Q_{l}:F\rightarrow F$ where $F$ is an abstract vector space). For $k\in \N$ we denote 
\begin{equation*}
Q_{l}^{[k]}=Q_{l}\dots Q_{l}\quad k\text{ times.}
\end{equation*}
Given a random tree $(\mathcal{A},\kappa (\mathcal{A}))$ we construct by
recurrence $Q_{l}^{\mathcal{A}}$ in the following way (again we drop the dependence on $\kappa(\cA)$ in the notation). If $\mathcal{A}=\{\emptyset \}$ we define 
\begin{equation}
Q_{l}^{\mathcal{A}}=Q_{l}^{\{\emptyset \}}=Q_{l+1}^{[n]}.  \label{Op2}
\end{equation}%
Suppose now that $j_{\emptyset }(\mathcal{A})=i\geq 1$ and let $\{\kappa _{1}<...<\kappa _{i}\}=\kappa _{\emptyset }(\mathcal{A})$. We
define by recurrence%
\begin{equation}
Q_{l}^{\mathcal{A}}=Q_{l+1}^{[n-\kappa _{i}^{\prime
}]}\prod_{j=1}^{i}(Q_{l+1}^{\mathcal{A}_{j}^{\prime }}Q_{l+1}^{[\kappa
_{j}-\kappa _{j-1}^{\prime }]}),  \label{Op1}
\end{equation}
see Remark~\ref{random_tree_heritage} for the dependence between the random trees. Finally, we define $\Gamma _{l}^{\mathcal{A}}$ in the following way. If $%
\mathcal{A=}\{\emptyset \}$ we define 
\begin{equation}
\Gamma _{l}^{\mathcal{A}}=\Gamma _{l}^{\{\emptyset \}}=Q_{l+1}^{[n]}-Q_{l}
\label{Op4}
\end{equation}%
and if ${j}_{\emptyset }(\mathcal{A})=i\geq 1,$ we define by
recurrence%
\begin{equation}
\Gamma _{l}^{\mathcal{A}}=Q_{l+1}^{[n-\kappa _{i}^{\prime
}]}\prod_{j=1}^{i}(\Gamma _{l+1}^{\mathcal{A}_{j}^{\prime }}Q_{l+1}^{[\kappa
_{j}-\kappa _{j-1}^{\prime }]}).  \label{Op5}
\end{equation}%
Notice that the recurrence formula~(\ref{Op5}) is the same as (\ref{Op1}),
but the initial condition (\ref{Op4}) is different from (\ref{Op2}).

We consider now a deterministic tree $\cT$ (in contrast with $\cA$ which is a random tree) and define by recurrence $\Delta_{l}(\cT)$. If $\cT=\{\emptyset\}$, then 
\begin{equation}
\Delta _{l}(\cT)=\Delta _{l}(\emptyset )=Q_{l+1}^{[n]}-Q_{l}
\label{Op7}
\end{equation}%
and if ${j}_{\emptyset }(\cT)>0,$ 
\begin{equation}
\Delta _{l}(\cT)=\Delta _{l}(\{\emptyset \})+\sum_{i=1}^{{j}
_{\emptyset }(\cT)}\binom{n}{i}\times \E_{\mu _{i}}\left(
Q_{l+1}^{[n-\kappa _{i}^{\prime }]}\prod_{j=1}^{i}\Delta _{l+1}(\cT
_{i}^{\prime })Q_{l+1}^{[\kappa _{j}-\kappa _{j-1}^{\prime }]}\right)
\label{Op8}
\end{equation}%
where $\mu_{i}$ is the uniform law of the order statistics $0\leq \kappa
_{1}<...<\kappa _{i}\leq n-1.$

Our aim now is to give an explicit computational formula for $\Delta _{l}(\mathcal{T}).$ In order to do it, we need to introduce one more notation
concerning families of trees (forests). Let $\mathbb{F}=\{\cT_{j},j\in J_{\mathbb{F}}\}$, where $J_{\mathbb{F}}$ is a family of indices and $\cT_{j}$ is a tree for all $j\in J_{\mathbb{F}}$. Given $m\in \N^*$, we construct
\begin{equation}
\mathbb{F}^{\otimes m}=\{\cT(j_{1},\dots,j_{m}),j_{i}\in J_{\mathbb{F}},i=1,\dots,m\}  \label{Op9}
\end{equation}%
with 
\begin{equation}
\cT(j_{1},\dots,j_{m})=\{\emptyset \}\cup 1\cT_{j_{1}}\cup \dots \cup m\cT_{j_{m}}.  \label{Op10}
\end{equation}%
So, the tree $\mathcal{T}(j_{1},\dots,j_{m})$ is obtained by rooting to the ancestor the tree $\mathcal{T}_{j_{k}}$ at the node $k$, for each $k=1,...,m$. We note $\mathcal{A}(j_{1},\dots,j_{m})$ the random tree obtained by labeling the nodes of  $\mathcal{T}(j_{1},\dots,j_{m})$  with random variables, according to Definition~\ref{def_random_tree}. 

Using this notation, we are able to construct the family of trees associated
to a tree $\mathcal{T}$ by recurrence, in the following way. If $\mathcal{T}=\{\emptyset \}$ (this means the tree is composed just by the ancestor)
then we define $\mathbf{F(}\mathcal{T})=\mathbf{F}(\{\emptyset \})=\{\emptyset \}$ - so the finite family associated to the tree $\{\emptyset \}$ has just one element which is the tree  $\{ \emptyset \}$. Suppose now that $j_{\emptyset}(\mathcal{T)}\geq 1$. Then, we define%
\begin{equation}
\mathbf{F(}\mathcal{T})=\{\emptyset \}\cup \left( \cup _{i=1}^{j_{\emptyset }(%
\mathcal{T})}\mathbf{F}^{\otimes i}(\mathcal{T}_{i}^{\prime }) \right),
\label{Op11}
\end{equation}
where $\mathbf{F}^{\otimes i}(\mathcal{T}_{i}^{\prime })$ is the shorthand notation for $(\mathbf{F}(\mathcal{T}_{i}^{\prime }))^{\otimes i}$.
We are now able to give the first result from this section.
\begin{proposition}
\label{delta} Let $\mathcal{T}$ be a tree and let $\mathbf{F}(\mathcal{T})$
be the family of trees associated to $\mathcal{T}$ in~(\ref{Op11}). We label each tree $\cA\in \mathbf{F}(\mathcal{T})$ with random variables, so that $\cA$ is a random tree in the sense of Definition~\ref{def_random_tree}. Then,
with $\Gamma_{l}^{\mathcal{A}}$ defined in (\ref{Op4}) \ and (\ref{Op5}), we have
\begin{equation}
\Delta _{l}(\mathcal{T})=\Delta _{l}(\{\emptyset \})+\sum_{i=1}^{j%
_{\emptyset }(\mathcal{T})}\sum_{\mathcal{A}\in \mathbf{F}^{\otimes i}(%
\mathcal{T}_{i}^{\prime })}c(\mathcal{A})\E(\Gamma _{l}^{\mathcal{A}})=\sum_{%
\mathcal{A}\in \mathbf{F}(\mathcal{T})}c(\mathcal{A})\E(\Gamma _{l}^{\mathcal{%
A}})  \label{Op12}
\end{equation}%
with 
\begin{equation}
c(\mathcal{A})=\prod_{u\in \mathcal{A}}\binom{n}{j_{u}(\mathcal{A})}%
,\quad c(\emptyset )=1.  \label{Op13}
\end{equation}
\end{proposition}

\begin{proof}
  Take first $\mathcal{T}=\{\emptyset \}$. By definition, we
have $\Delta _{l}(\{\emptyset \})=\Gamma _{l}^{\{\emptyset\}}=Q_{l+1}^{[n]}-Q_{l}.$ On the other hand we have $\mathbf{F}(\mathcal{T})=\mathbf{F}(\{\emptyset \}\mathcal{)}=\{\emptyset \}$ and $c(\emptyset
)\Gamma _{l}^{\{\emptyset \}}=\Gamma _{l}^{\{\emptyset \}}=\Delta
_{l}(\emptyset ).$ So the equality (\ref{Op12}) holds true.

Suppose now that (\ref{Op12}) holds if $\left\vert \mathcal{T}\right\vert
\leq q-1$ and let us prove it for $\left\vert \mathcal{T}\right\vert =q.$
Using the recurrence formula (\ref{Op8}) first and the recurrence
hypothesis then we get 
\begin{eqnarray*}
\Delta _{l}(\mathcal{T}) &=&\Delta _{l}(\{\emptyset \})+\sum_{i=1}^{{j%
}_{\emptyset }(\mathcal{T})}\binom{n}{i}\times \E_{\mu _{i}}\left(
Q_{l+1}^{[n-\kappa _{i}^{\prime }]}\prod_{j=1}^{i}\Delta _{l+1}(\mathcal{T}%
_{i}^{\prime })Q_{l+1}^{[\kappa _{j}-\kappa _{j-1}^{\prime }]}\right) \\
&=&\Delta _{l}(\{\emptyset \})+\sum_{i=1}^{{j}_{\emptyset }(\mathcal{T%
})}\binom{n}{i}\times \E_{\mu _{i}}\left( Q_{l+1}^{[n-\kappa _{i}^{\prime
}]}\prod_{j=1}^{i}(\sum_{\mathcal{A}_{j}\in \mathbf{F}(\mathcal{T}%
_{i}^{\prime })}c(\mathcal{A}_{j})\E(\Gamma _{l+1}^{\mathcal{A}%
j}))Q_{l+1}^{[\kappa _{j}-\kappa _{j-1}^{\prime }]}\right) .
\end{eqnarray*}%
Let $i\in \{1,..,{j}_{\emptyset }(\mathcal{T})\}.$ We have 
\begin{eqnarray}
&&\binom{n}{i}\times \E_{\mu _{i}}\left( Q_{l+1}^{[n-\kappa _{i}^{\prime
}]}\prod_{j=1}^{i}(\sum_{\mathcal{A}_{j}\in \mathbf{F}(\mathcal{T}%
_{i}^{\prime })}c(\mathcal{A}_{j})\E(\Gamma _{l+1}^{\mathcal{A}%
_{j}}))Q_{l+1}^{[\kappa _{j}-\kappa _{j-1}^{\prime }]}\right)  \label{Op15}
\\
&=&\sum_{\mathcal{A}_{j_{1}},...,\mathcal{A}_{j_{i}}\in \mathcal{T}%
_{i}^{\prime }}\binom{n}{i}\prod_{k=1}^{i}c(\mathcal{A}_{j_{k}})\times
\E_{\mu _{i}}\left(Q_{l+1}^{[n-\kappa _{i}^{\prime }]}\prod_{k=1}^{i}\E(\Gamma
_{l+1}^{\mathcal{A}_{j_{k}}})Q_{l+1}^{[\kappa _{k}-\kappa _{k-1}^{\prime }]}\right)
\notag
\end{eqnarray}%
We recall that $\mathcal{A}(j_{1},...,j_{i})$ is defined in (\ref{Op10}), and
by construction we have $(\mathcal{A}(j_{1},...,j_{i}))_{k}^{\prime }=\mathcal{A}_{j_{k}}.$ Then, we use the recurrence formula (\ref{Op5}) in the definition
of $\Gamma _{l}^{\mathcal{A}_{i}(j_{1},...,j_{i})}$ and we obtain%
\begin{equation*}
\E_{\mu _{i}}\left(Q_{l+1}^{[n-\kappa _{i}^{\prime }]}\prod_{k=1}^{i} \E(\Gamma
_{l+1}^{\mathcal{A}_{j_{k}}})Q_{l+1}^{[\kappa _{k}-\kappa _{k-1}^{\prime
}]} \right)=\E(\Gamma _{l}^{\mathcal{A}(j_{1},...,j_{i})}).
\end{equation*}%
Moreover, we have
\begin{equation*}
\binom{n}{i}\prod_{k=1}^{i}c(\mathcal{A}_{j_{k}})=c(\mathcal{A}%
(j_{1},...,j_{i})),
\end{equation*}%
so the term in (\ref{Op15}) is equal to 
\begin{equation*}
\sum_{\mathcal{A}_{j_{1}},...,\mathcal{A}_{j_{i}}\in \mathcal{T}_{i}^{\prime
}}c(\mathcal{A}(j_{1},...,j_{i}))\E(\Gamma _{l}^{\mathcal{A}%
(j_{1},...,j_{i})})=\sum_{\mathcal{A}\in \mathbf{F}^{\otimes i}(\mathcal{T}%
_{i}^{\prime })}c(\mathcal{A})\E(\Gamma _{l}^{\mathcal{A}}).
\end{equation*}%
We conclude that%
\begin{equation*}
\Delta _{l}(\mathcal{T})=\Delta _{l}(\{\emptyset \})+\sum_{i=1}^{{j}%
_{\emptyset }(\mathcal{T})}\sum_{\mathcal{A}\in \mathbf{F}^{\otimes i}(%
\mathcal{T}_{i}^{\prime })}c(\mathcal{A})\E(\Gamma _{l}^{\mathcal{A}})=\sum_{%
\mathcal{A}\in \mathbf{F}(\mathcal{T})}c(\mathcal{A})\E(\Gamma _{l}^{\mathcal{%
A}}). \qedhere
\end{equation*}
\end{proof}

Our aim now is to compute in an explicit way $\Gamma _{l}^{\mathcal{A}}.$
For a tree $\mathcal{A}$ and for a subset of leaves $\Lambda \subset \mathcal{E(A)}$,
we define $\mathcal{A}_{\Lambda }=\mathcal{A}\setminus\Lambda$:  we cut the
extreme nodes which belong to~$\Lambda$. Notice that $\mathcal{A}_{\Lambda
}$ is no more a tree: for example, if $\mathcal{A}=\{\emptyset ,1,2,3\}$ and 
$\Lambda =\{2\}$ then $\mathcal{A}_{\Lambda }=\{\emptyset ,1,3\}$ is not
a tree: the first and second axioms of Definition~\ref{def_tree} are satisfied, not the third. We also stress that $\mathcal{A}_{\Lambda }$ may be the void set in the case $\mathcal{A}=\{\emptyset \}$ and $\Lambda =\{\emptyset \}$. Thus, we look to $\mathcal{A}_{\Lambda }$ as to a set (not a tree) which may be void as well (remember the convention~\eqref{convention}).

Suppose that $j_{\emptyset }(\mathcal{A})=r.$ Our first concern is to
precise how $\Lambda \subset \mathcal{E(A)}$ is decomposed on each of the subtrees $\mathcal{A}_{i}^{\prime },i=1,...,r$. We define 
\begin{equation}
\Lambda _{i}=\{u\in \mathcal{A}_{i}^{\prime }:iu\in \Lambda \}.  \label{V12}
\end{equation}%
We stress that, if no descendant of $i$ belongs to $\Lambda$, then we have $\{u\in \mathcal{A}_{i}^{\prime }:iu\in \Lambda \}=\varnothing $ (void set).
We also have%
\begin{equation}
\Lambda _{i}=\{\emptyset \}\text{ (ancestor)\quad } \text{if}\quad i\in \Lambda .
\label{V12'}
\end{equation}

We define now $Q_{l}^{\mathcal{A}_{\Lambda }}$ recursively. First, if $%
\mathcal{A}_{\Lambda }=\varnothing $ (void set) or if $\mathcal{A}_{\Lambda
}=\{\emptyset \}$ (ancestor) we define%
\begin{equation*}
Q_{l}^{\varnothing }=Q_{l},\quad Q_{l}^{\{\emptyset \}}=Q_{l+1}^{[n]}.
\end{equation*}%
Otherwise we have $j_{\emptyset }(\mathcal{A)}=r\geq 1$, and we define 
\begin{equation}
Q_{l}^{\mathcal{A}_{\Lambda }}=Q_{l+1}^{[n-\kappa _{r}^{\prime
}]}\prod_{i=1}^{r}Q_{l+1}^{(\mathcal{A}_{i}^{\prime })_{\Lambda
_{i}}}Q_{l+1}^{[\kappa _{i}-\kappa _{i-1}^{\prime }]}.  \label{V9}
\end{equation}%
with $(\kappa _{1},...,\kappa _{r})=\kappa _{\emptyset }(\mathcal{A})$ and $%
\Lambda _{i}$ defined in (\ref{V12}).

Before going further, we construct the grid $G_{l}(\mathcal{A}_{\Lambda })$
in a similar way with $G_{l}(\mathcal{A})$ defined in (\ref{G2}). We denote $j_{\emptyset }^{\Lambda }(\mathcal{A}):=j_{\emptyset }(\mathcal{A})-\card(\{1,...,j_{\emptyset }(\mathcal{A)\}}\cap \Lambda) .$ So $j_{\emptyset
}^{\Lambda }(\mathcal{A)}$ represents the number of sons of the ancestor $%
\emptyset $ which are not in $\Lambda $ (so, that are alive after killing
the individuals from $\Lambda$). We also denote $\{i_{1},...,i_{j_{\emptyset }^{\Lambda }(\mathcal{A)}}\}=\{1,...,j_{\emptyset }(\mathcal{A})\}\setminus \Lambda$, the indices of the surviving sons. Then, we define (with the convention $\cup
_{j=1}^{0}=\varnothing \}$\ 
\begin{equation}
G_{l}(\mathcal{A}_{\Lambda })=\{qh_{l+1},q=0,...,n\} \cup \left( \cup _{j=1}^{j_{\emptyset }^{\Lambda }(\mathcal{%
A)}}\{\kappa _{i_{j}}(\emptyset) h_{l+1}+G_{l+1}((\mathcal{A}_{i}^{\prime })_{\Lambda
_{i}}) \}\right)   \label{G8}
\end{equation}%
if $\cA_\Lambda \not = \varnothing$, and $G_{l}(\varnothing)=\{0,h_{l}\}$. 
Here $\Lambda _{i}$ is the set defined in (\ref{V12}). So, we use the
refinement procedure for $i_{j}$ only, and not for every $i=1,\dots,j_{\emptyset }(\mathcal{A})$. In the case $%
\Lambda =\varnothing $ (void set) $G_{l}(\mathcal{A}_{\Lambda })$ coincides with $G_{l}(\mathcal{A})$. As for Lemma~\ref{lem_grids}, we can show that 
\begin{equation*}
G_{l}(\mathcal{A}_{\Lambda })=\{0,h_l\} \cup\left( \cup_{u\in \mathcal{A}_\Lambda
}\{t_{l}(u)+kh_{l+\left\vert u\right\vert +1},k=0,...,n\}\right).
\end{equation*}
Note that we need to add the union with $\{0,h_l\} $ for the case $\mathcal{A}_\Lambda=\varnothing$, i.e. when $\cA=\Lambda=\{\emptyset\}$.
We denote%
\begin{equation}
\Pi _{l}(\mathcal{A}_{\Lambda })=\{0=s_{0}<s_{1}<...<s_{m}=h_{l}\}\subset
\{qh_{l+r+1},q=0,...,n^{r+1}\}  \label{G9}
\end{equation}%
the reordering of $G_{l}(\mathcal{A}_{\Lambda })$. We notice that for every $k=1,....,m$ one has $s_{k}-s_{k-1}=h_{l+p_{k}}$ for some $p_{k}=1,2,...,|\mathcal{A}_{\Lambda }|$. Thus, we produce
a sequence $p(\mathcal{A}_{\Lambda },\kappa (\mathcal{A}_{\Lambda }))$
associated to~$\mathcal{A}_{\Lambda }$, and we have 
\begin{equation}
Q_{l}^{\mathcal{A}_{\Lambda }}=Q_{l+p_{1}(\mathcal{A}_{\Lambda },\kappa (%
\mathcal{A}_{\Lambda }))}\dots Q_{l+p_{m_{\Lambda }}(\mathcal{A}%
_{\Lambda },\kappa (\mathcal{A}_{\Lambda }))}.  \label{G10}
\end{equation}

\begin{proposition}\label{landa} Let $\Gamma _{l}^{\mathcal{A}}$\ defined in (\ref{Op4}) \ and
(\ref{Op5}) and $Q_{l}^{\mathcal{A}_{\Lambda }}$ defined in~(\ref{V9}). Then 
\begin{equation}
\Gamma _{l}^{\mathcal{A}}=\sum_{\Lambda \subset \mathcal{E(A)}%
}(-1)^{\card(\Lambda) }Q_{l}^{\mathcal{A}_{\Lambda }}.
\label{V11}
\end{equation}%
The above sum includes $\Lambda =\varnothing $ (void set) and $\Lambda =%
\mathcal{E(A)}.$
\end{proposition}

Before giving the proof of the above proposition, we need to get a more
detailed description of the set $\Lambda $ and of the decomposition given in
(\ref{V12}). Let $\cA$ be such that $j_\emptyset(\cA)>0$, so that $\emptyset \not \in \cE(\cA)$. Then, for any $\Lambda \subset  \cE(\cA)$, we denote%
\begin{equation}
D(\Lambda )=(\Lambda _{1},...,\Lambda _{r}),  \label{V15}
\end{equation}%
where $\Lambda_i$ is defined by \eqref{V12}.
We define now the converse operation: given a sequence of sets $\Lambda
_{i}^{\prime }\subset \mathcal{E(}\mathcal{A}_{i}^{\prime }),i=1,...,r$ we
 define%
\begin{equation}
\Lambda' =\{iu:i=1,...,r,u\in \Lambda _{i}^{\prime }\}.  \label{V16'}
\end{equation}%
In order to precise the structure of $\Lambda'$ we consider the sets of
indices $J_{i}\subset \{1,...,r\}$ defined by%
\begin{eqnarray*}
J_{1} &=&\{i:\Lambda _{i}^{\prime }=\varnothing \text{ (void)}\} \\
J_{2} &=&\{i:\Lambda _{i}^{\prime }=\{\emptyset \}\text{ (ancestor)}\} \\
J_{3} &=&\{1,...,r\}-J_{1}-J_{2}.
\end{eqnarray*}%
 We stress that for $i\in J_{3},$ the set $%
\Lambda _{i}^{\prime }$ is not void and does not contain the ancestor $%
\emptyset .$ Then the set $\Lambda'$ defined in (\ref{V16'}) is given by%
\begin{equation}
\Lambda'=\{i:i\in J_{2}\}\cup _{i\in J_{3}}\{iu:u\in \Lambda _{i}^{\prime }\}
\label{V16}
\end{equation}%
and we define 
\begin{equation}
D^{-1}(\Lambda _{1}^{\prime },...,\Lambda _{r}^{\prime })= \Lambda'.
\label{V17}
\end{equation}%
\begin{lemma}\label{lem_decompo_lambda}Let $\cA$ be a tree such that $j_\emptyset(\cA)=r>0$. Then, we have for any $\Lambda\in \cE(\cA)$ and any $\Lambda _{i}^{\prime }\in \cE(\cA'_i)$, $i=1,\dots,r$
\begin{equation}
DD^{-1}(\Lambda _{1}^{\prime },...,\Lambda _{r}^{\prime })=(\Lambda
_{1}^{\prime },...,\Lambda _{r}^{\prime })\quad and\quad D^{-1}D\Lambda
=\Lambda .  \label{D18}
\end{equation}
\end{lemma}

\begin{proof}
  We just check the first equality. Let $\Lambda'
=D^{-1}(\Lambda _{1}^{\prime },...,\Lambda _{r}^{\prime }).$ We have to prove
that for each $i=1,...,r$ we have $\Lambda _{i}^{\prime }=D_{i}(\Lambda')$, where $D_i$ is the $i$th coordinate of the application~$D$ defined by~\eqref{V15}. From~\eqref{V16}, we have $\Lambda'=\cup_{1\le i\le r: \Lambda'_i \not = \varnothing }\{iu: u\in \Lambda'_i\}$. Thus, $D_i(\Lambda)=\varnothing$ if $\Lambda'_i  = \varnothing $ and $D_i(\Lambda)=\{u:  u\in \Lambda'_i\}=\Lambda'_i$ otherwise. 
The second equality is verified in a similar way. 
\end{proof}

\begin{proof}[Proof of Proposition \ref{landa}.] If $j_{\emptyset }(\mathcal{A})=0$
then $\mathcal{A}=\{\emptyset \}$ and $\Gamma _{l}^{\mathcal{A}%
}=Q_{l+1}^{n}-Q_{l}=Q_{l}^{\{\emptyset \}}-Q_{l}^{\varnothing }=Q_{l}^{%
\mathcal{A}_{\Lambda _{1}}}-Q_{l}^{\mathcal{A}_{\Lambda _{2}}}$ with $%
\Lambda _{1}$ is the void set and $\Lambda _{2}=\{\emptyset \}.$ So (\ref%
{V11}) holds.

If $j_{\emptyset }(\mathcal{A})=r>0$ then, using the recurrence hypothesis%
\begin{eqnarray*}
\Gamma _{l}^{\mathcal{A}} &=&Q_{l+1}^{[n-\kappa _{r}^{\prime
}]}\prod_{j=1}^{r} \left(\Gamma _{l+1}^{\mathcal{A}_{j}^{\prime }}Q_{l+1}^{[\kappa
_{j}-\kappa _{j-1}^{\prime }]} \right) \\
&=&Q_{l+1}^{[n-\kappa _{r}^{\prime }]}\prod_{j=1}^{r}\left(\sum_{\Lambda
_{j}\subset \mathcal{E(A}_{j}^{\prime }\mathcal{)}}(-1)^{\card(\Lambda
_{j})}Q_{l}^{(\mathcal{A}_{j}^{\prime }\mathcal{)}_{\Lambda
_{j}}}Q_{l+1}^{[\kappa _{j}-\kappa _{j-1}^{\prime }]}\right) \\
&=&\sum_{\Lambda _{j_{1}}\subset \mathcal{E(A}_{1}^{\prime }\mathcal{)}%
}...\sum_{\Lambda _{j_{r}}\subset \mathcal{E(A}_{r}^{\prime }\mathcal{)}%
}(-1)^{\card(\Lambda _{j_{1}}) +\dots+\card(\Lambda
_{j_{r}})}Q_{l+1}^{[n-\kappa _{r}^{\prime
}]}\prod_{k=1}^{r} \left(Q_{l}^{(\mathcal{A}_{k}^{\prime }\mathcal{)}_{\Lambda
_{j_{k}}}} Q_{l+1}^{[\kappa _{k}-\kappa _{k-1}^{\prime }]} \right).
\end{eqnarray*}%
Let $\Lambda =D^{-1}(\Lambda _{j_{1}},...,\Lambda _{j_{r}}).$ We have $\card(\Lambda _{j_{1}}) +\dots+\card(\Lambda
_{j_{r}})=\card(\Lambda)$, and according to~(\ref{V9})
\begin{equation*}
Q_{l}^{\mathcal{A}_{\Lambda }}=Q_{l+1}^{[n-\kappa _{r}^{\prime
}]}\prod_{k=1}^{r} \left( Q_{l}^{(\mathcal{A}_{k}^{\prime }\mathcal{)}_{\Lambda
_{j_{k}}}}Q_{l+1}^{[\kappa _{k}-\kappa _{k-1}^{\prime }]} \right).
\end{equation*}%
Since every $\Lambda \subset \mathcal{E(A)}$ may be decomposed in this way
by Lemma~\ref{lem_decompo_lambda}, we get 
\begin{equation*}
\Gamma _{l}^{\mathcal{A}}=\sum_{\Lambda \subset \mathcal{E(A)}%
}(-1)^{\left\vert \Lambda \right\vert }Q_{l}^{\mathcal{A}_{\Lambda }}.\qedhere
\end{equation*}%
\end{proof}

\subsection{Tree representation of the approximation schemes}\label{Sec_tree_approx}

We define now a family of trees which describes our approximation schemes. For $\nu \geq 1$ and $l \geq 0$, let us define the tree $\mathcal{T}_{l }^{\nu }$ as
follows: 
\begin{equation}
\mathcal{T}_{l }^{\nu }=\{\emptyset \}\cup \left( \bigcup_{i=1}^{m(l,\nu
)-1}i\mathcal{T}_{l +1}^{q_{i}(l ,\nu )}\right) ,  \label{Tree}
\end{equation}%
with $q_{i}(l ,\nu ),m(l,\nu )$ given in (\ref{n1}), (\ref{n1'}) and the
convention $\cup _{i=1}^{0}\{...\}=\varnothing $ (void set). These trees are defined by recurrence, and it is not clear at a first glance that the induction ends. This true by the next lemma.

\begin{lemma}\label{lem_Hk} Let $\alpha>0$. Let us denote for $k\in \N$, 
\begin{equation}\label{defHk}
\cH_{k}=\{(\nu ,l) \in  \N^{2}:\nu \leq \alpha +(1+\alpha )l+\alpha k\}.
\end{equation}%
We have  $\cup _{k\in \N}\cH _{k}=\N^{2}$ and
$$\forall k\in \N, \ (\nu ,l)\in \cH_{k+1} \implies \forall i \in \{1,\dots,m(l,\nu)-1\}, (q_i(l,\nu),l+1)\in \cH_{k}.$$
In particular, the recursion defining~$\cT^\nu_l$ in formula~\eqref{Tree} ends  for every $(\nu ,l)\in \N^2$.
\end{lemma}
\begin{proof} Since $\alpha>0$, we have $(\nu,l)\in \cH_{\lceil\max((\nu-l)/\alpha-(l+1),0) \rceil}$  for any $(\nu,l)\in \N^2$, which gives $\cup _{k\in \N}\cH _{k}=\N^{2}$. 
  Let us first observe that for $(\nu,l) \in \cH_0$, we have  $m(l,\nu )=1$ and thus $\mathcal{T}_{l}^{\nu }=\{\emptyset \}$ by~\eqref{Tree}. Therefore, the implication will prove that the recursion ends. Let us take then $(\nu ,l)\in \cH _{k+1}$ and  $i$ such that $1\le i \le m(l,\nu)-1$.  The last inequality implies $\nu\ge (1+\alpha)li+\alpha i$ and then $q_i(l,\nu)\ge 0$. Thus, we have to check  that 
\begin{equation*}
\nu + \lceil i-(1+\alpha )(l +1)(i-1)\rceil=q_{i}(l,\nu )\leq \alpha +(1+\alpha
)(l+1)+\alpha k.
\end{equation*}%
Since $\nu \leq \alpha +(1+\alpha )l+\alpha (k+1)$, it is sufficient to prove
\begin{equation*}
\alpha +(1+\alpha )l+\alpha (k+1)+\lceil i-(1+\alpha )(l +1)(i-1)\rceil\leq \alpha
+(1+\alpha )(l+1)+\alpha k.
\end{equation*}%
After simplifications, this inequality is equivalent to%
\begin{equation*}
\lceil i-(1+\alpha )(l +1)(i-1) \rceil \leq 1,
\end{equation*}%
which clearly holds true for  every $l\in \mathbb{N},i\in \mathbb{N}^{\ast }$ since $\alpha \ge 0$.  
\end{proof}

Now, we explain how we associate an approximation scheme to a finite tree.
In the following we will work with the specific trees $\mathcal{T}_{l}^{\nu}$ constructed above, but for the moment we consider a general finite tree $
\mathcal{T}$. We recall that ${j}_{\emptyset }(\mathcal{T})$ is the
number of sons of the root~$\emptyset $, and for $1\leq i\leq j_{\emptyset }(\mathcal{T})$, $\mathcal{T}_{i}^{\prime }=\{u\in \mathcal{U}
,iu\in \mathcal{T}\}$ is the subtree that is rooted at the node~$i$. For a
finite tree $\mathcal{T}$, we define the approximation scheme $\hat{Q}%
_{h_{l }}(\mathcal{T})$ as follows by induction.

If ${j}_{\emptyset }(\mathcal{T})=0$ then $\mathcal{T}=\{\emptyset \}$ and we put 
\begin{equation*}
\hat{Q}_{h_{l }}(\{\emptyset \})=(P_{h_{l +1}}^{h_{l+1}})^{n}=P_{h_{l}}^{h_{l+1}}.
\end{equation*}%
If ${j}_{\emptyset }(\mathcal{T})\geq 1$ we define by recurrence%
\begin{equation}
\hat{Q}_{h_{l }}(\mathcal{T})=(P_{h_{l +1}}^{h_{l
+1}})^{n}+\sum_{i=1}^{{j}_{\emptyset }(\mathcal{T})}\binom{n}{i}%
\times \E_{\mu _{i}}\left( P_{(n-\kappa _{i}^{\prime })h_{l +1}}^{h_{l
+1}}\prod_{j=1}^{i}(\hat{Q}_{h_{l +1}}(\mathcal{T}_{i}^{\prime
})-P_{h_{l +1}}^{h_{l +1}})P_{(\kappa _{j}-\kappa _{j-1}^{\prime
})h_{l +1}}^{h_{l +1}}\right) .  \label{approx_tree}
\end{equation}%
Since $\left\vert \mathcal{T}_{i}^{\prime }\right\vert =\left\vert \mathcal{T%
}\right\vert -1$ and the tree $\mathcal{T}$ is finite, this induction
clearly ends.

\begin{proposition}
  \label{prop_anyorder} (Tree representation of the approximations of order $\nu$)

  For every $\nu \geq 1,l \geq
0 $, we have 
\begin{equation*}
\hat{Q}_{h_{l }}(\mathcal{T}_{l}^{\nu })=\hat{P}_{h_{l }}^{\nu }
\end{equation*}%
where $\hat{P}_{h_{l }}^{\nu }$ is the approximation defined in~%
\eqref{induction_hatPl}. Consequently, we have
\begin{equation}
\exists C>0, \ \left\Vert (\hat{Q}_{h_{l }}(\mathcal{T}_{l}^{\nu
})-P_{h_{l}})f\right\Vert _{\infty }\leq C \left\Vert f\right\Vert
_{k(l,\nu ),\infty }n^{-\nu },  \label{D3}
\end{equation}%
with $k(l,\nu )$ defined in (\ref{n2}). In particular, taking $l=0$ (recall
that $h_{0}=T)$ we obtain%
\begin{equation}
\exists C>0, \ \left\Vert (\hat{Q}_{T}(\mathcal{T}_{0}^{\nu })-P_{T})f\right\Vert _{\infty
}\leq C\left\Vert f\right\Vert _{k(0,\nu ),\infty }n^{-\nu }.  \label{D4}
\end{equation}
\end{proposition}

\begin{proof}
We consider the sets $\cH_k$ defined in~\eqref{defHk} and prove the result by induction on~$k$.  
For $(\nu ,l)\in \cH _{0}$ we have $m(l,\nu )=0$ and $\mathcal{T}_{l}^{\nu }=\{\emptyset \}$ so that $\hat{Q}_{h_{l }}(\mathcal{T}_{l}^{\nu })=P_{h_{l }}^{h_{l +1}}=\hat{P}_{h_{l }}^{\nu }$. Let $(\nu,l)\in \cH_{k+1}$. From~(\ref{Tree}), we have $(\mathcal{T}_{l}^{\nu})_{i}^{\prime }=\mathcal{T}_{l+1}^{q_{i}(l,\nu )}$. Using Lemma~\ref{lem_Hk} and the induction hypothesis, we get $\hat{Q}_{l+1}(\mathcal{T}_{l+1}^{q_{i}(l,\nu )})=\hat{P}_{l+1}^{q_{i}(l,\nu )}.$ We also have ${j}_{\emptyset }(\mathcal{T}_{l}^{\nu })=m(l,\nu )-1,$ so the recurrence formulas~\eqref{approx_tree} for $\hat{Q}_{h_{l}}(\mathcal{T}_{l}^{\nu })$ and~\eqref{induction_hatPl} for $\hat{P}_{h_{l }}^{\nu }$ coincide, proving the claim.\end{proof}

We put the above formula in an alternative form which is more enlightening
and easier to handle. We define%
\begin{equation}
\Delta_{l}(\mathcal{T})=\hat{Q}_{h_{l }}(\mathcal{T})-P_{h_{l}}^{h_{l}}.
\label{D1}
\end{equation}%
Then, $\Delta_{l}(\{\emptyset\})=P_{h_{l}}^{h_{l+1}}-P_{h_{l}}^{h_{l}}$ formula~\eqref{approx_tree} is equivalent to
\begin{equation*}
\Delta_{l}(\mathcal{T})=\Delta _{l}(\{\emptyset \})+\sum_{i=1}^{{j}_{\emptyset }(\mathcal{T})}\binom{n}{i}\times \E_{\mu _{i}}\left(
P_{(n-\kappa _{i}^{\prime })h_{l +1}}^{h_{l +1}}\prod_{j=1}^{i}\Delta
_{l+1}(\mathcal{T}_{i}^{\prime })P_{(\kappa _{j}-\kappa _{j-1}^{\prime})h_{l +1}}^{h_{l +1}}\right) .  
\end{equation*}
This is precisely the operator defined in~\eqref{Op8}. We are now able to give the main result in this section, which is a consequence of Propositions~\ref{delta} and~\ref{prop_anyorder}.

\begin{theorem}
\label{MAIN} Suppose that Hypotheses~\eqref{b00} and~\eqref{b000} hold true. Let $\nu \in \N$ be given and let $\mathcal{T}_{0}^{\nu }$ be the tree constructed in (\ref{Tree}) for $l=0$. Let $\mathbf{F}(\mathcal{T}_{0}^{\nu })$ be the family of trees associated to $\mathcal{T}_{0}^{\nu }$ in (\ref{Op11}) and let $c(\mathcal{A})$ be given in (\ref{Op13}). Then, we define 
\begin{equation}
\hat{Q}_{T}(\mathcal{T}_{0}^{\nu })=Q_{0}+\sum_{\mathcal{A}\in \mathbf{F}(%
  \mathcal{T}_{0}^{\nu })}c(\mathcal{A})\E[\Gamma^{\cA}_0]
\label{D5}
\end{equation}%
and we have 
\begin{equation}
\left\Vert (\hat{Q}_{T}(\mathcal{T}_{0}^{\nu })-P_{T})f\right\Vert _{\infty
}\leq C \left\Vert f\right\Vert _{k(0,\nu ),\infty }n^{-\nu }.  \label{D6}
\end{equation}
\end{theorem}

\bigskip

\begin{remark}
  The result presented in this section also holds in the more abstract framework described in the introduction, with semigroups defined on a vector space $F$ with seminorms~$\|\|_k$. Under~\eqref{barH1} and~\eqref{barH2}, we have similarly $\left\Vert (\hat{Q}_{T}(\mathcal{T}_{0}^{\nu })-P_{T})f\right\Vert _{0
}\leq C\left\Vert f\right\Vert _{k(0,\nu )}n^{-\nu }$. 
\end{remark}

\section{Probabilistic representation of the approximation semigroup for some Markov processes}\label{Sec_Prob_Repr}

All the results presented in the previous sections apply for an abstract
semigroup $P_{t}$ with a family of approximation schemes corresponding to the
abstract operators $Q_{l}$. If one wants to use a Monte Carlo algorithm,
one needs to use some probabilistic representation for $Q_{l}$ in order to compute the approximation schemes.
This probabilistic representation  may be very different according to the problem at hand. Nonetheless, a crucial common issue is the variance of the estimator. More precisely, the approximation proposed in~\eqref{D5} has to be seen as an addition of correction terms that can be calculated independently. Instead, it is very important to try to calculate jointly the terms appearing in $\E[\Gamma^{\cA}_0f]$ for $\cA \in \mathbf{F}(\cT^\nu_0)$. This is the sum of $2^{\card(\cE(\cA))}$ terms with rather close values since $\E[Q^{\cA_\Lambda}_0f]=P_Tf+O(h_1)$. If one would use independent samples to compute each term, the correcting term $\E[c(\cA)\Gamma^{\cA}_0f] $ would have roughly $c(\cA)^2\times 2^{\card(\cE(\cA))}$ times the variance of the initial basic Monte-Carlo estimator, which would make the approximation~\eqref{D5} poorly efficient. Fortunately, it is in general possible to do much better. 

The goal of this section is to precise the probabilistic representation of the approximation schemes, when considering diffusion processes or Piecewise Deterministic Markov Processes (PDMP). In these cases, it is possible to specify a probabilistic representation of all the schemes $(Q^{\cA_\Lambda}_0, \Lambda\subset \cE(\cA))$ on the same probability space, and such that the variance of $c(\cA)\Gamma^{\cA}_0f$ is bounded. 

\subsection{Probabilistic representation of $Q^{\cA_\Lambda}_0$}\label{subsec_probrepr}
We start by presenting a general framework. We consider a Polish space $\cZ$ and consider a kernel $\Theta: \R_{+}\times \cZ \times \R^{d}\rightarrow \R^{d}$ such that the application $(t,z,x) \rightarrow \Theta (t,z,x)$ is measurable. Moreover, we
consider an independent random variable  $Z:\Omega \rightarrow \cZ$ and define 
\begin{equation*}
Q_{l}f(x)=\E[f(\Theta (h_{l},Z,x))].
\end{equation*}%
We will assume that for some $\beta \in \N$ and $\alpha >0,$ the estimates (\ref{b00}) and (\ref{b000}) hold true. We consider now a grid $\Pi =\{0=s_{0}<s_{1}<....<s_{m}=T\}$ and denote $\delta_{k}=s_{k+1}-s_{k}$. We also consider a sequence of independent copies $Z_{k}$
of $Z$ and  define the random vector fields 
\begin{equation}
\theta _{k}(x)=\Theta (\delta _{k},Z_{k},x)
\label{e1a}
\end{equation}%
and the approximating flow defined by $X_{0}^{\Pi }(x)=x$ and 
\begin{equation}
X_{s_{k+1}}^{\Pi }(x)=\theta _{k}(X_{s_{k}}^{\Pi }(x)).  \label{e1}
\end{equation}%
To a tree $\mathcal{A}$ and to a subset $\Lambda \subset \mathcal{E(A)}$ we
associate $Q_{0}^{\mathcal{A}_{\Lambda }}$ defined in (\ref{V9}) and the
grid $\Pi _{0}(\mathcal{A}_{\Lambda })$ defined in (\ref{G8}). It is easy
to check that we have the probabilistic representation 
\begin{equation*}
\E[Q_{0}^{\mathcal{A}_{\Lambda }}f(x)]=\E[f(X_{T}^{\Pi _{0}(\mathcal{A}_{\Lambda
})}(x))].
\end{equation*}%
Therefore, (\ref{D5})  (with $c(\mathcal{A})$ given in (\ref{Op13})) can be rewritten as
\begin{equation}
\widehat{Q}_{T}(\mathcal{T}_{0}^{\nu })f(x)=\E[f(X_{T}^{\Pi _{0}(\{ \emptyset
\} )}(x))]+\sum_{\mathcal{A\subset }\mathbf{F}(\mathcal{T}_{0}^{\nu })}c(
\mathcal{A})\E\left[ \sum_{\Lambda \subset \mathcal{E(A)}}(-1)^{\left\vert \Lambda
\right\vert } f(X_{T}^{\Pi _{0}(\mathcal{A}_{\Lambda })}(x)) \right].  \label{e8}
\end{equation}%
It is clear that $\E[f(X_{T}^{\Pi _{0}(\mathcal{A}_{\Lambda })}(x))]$ (and
consequently $\widehat{Q}_{T}(\mathcal{T}_{0}^{\nu })f(x))$ may be computed
using Monte-Carlo simulation, and Theorem~\ref{MAIN} gives 
\begin{equation}
  \sup_{x\in \R^{d}}\left\vert P_{T}f(x)-\widehat{Q}_{T}(\mathcal{T}_{0}^{\nu
})f(x)\right\vert \leq C\left\Vert f\right\Vert _{k(0,\nu ),\infty }\times 
\frac{1}{n^{\nu }}.  \label{e3}
\end{equation}

\begin{remark} The above estimate involves $\left\Vert f\right\Vert_{k(0,\nu ),\infty }$ which requires
  much regularity for the test function $f$. However, under some supplementary regularity and non degeneracy assumptions
one may prove convergence in total variation distance. Precisely, one may consider
measurable and bounded test functions and replace $\left\Vert f\right\Vert
_{k(0,\nu ),\infty }$ by $\left\Vert f\right\Vert _{\infty }$ in the estimation of the error.
This has been done in Bally and Rey~\cite{BaRe} for usual
approximation schemes and the uniform grid (which corresponds in our framework to $\nu =1$, i.e. to $\mathcal{T}_{0}^{\nu }=\{\emptyset \}$). The supplementary hypothesis are the following. First, one has to assume that $Z:\Omega\rightarrow \R^q$
satisfies the so called Doeblin condition: there exists $\varepsilon >0, r>0$
and $z\in \R^{q}$ such that for every measurable set $A\subset \{z': |z'-z|<r \}$ one
has $\Px(Z\in A)\geq \varepsilon \lambda (A)$ where $\lambda $ is the Lebesgue
measure. Moreover one has to assume some non degeneracy condition on the gradient of $\Theta$
with respect to~$z$. However the proof is technical and non trivial, so we do not consider this possible extension in
the present paper.
\end{remark}

\subsection{Probabilistic representation of $\Gamma^{\cA_\Lambda}_0$}

We now specify, on some examples, how to sample jointly the random variables $Z$ for all the schemes $X^{\Pi_0(\cA_\Lambda)}(x)$ for $\Lambda\subset \cE(\cA)$.

\subsubsection{Approximation schemes for SDEs}

We deal with approximation schemes for the $d$ dimensional diffusion process 
$X_{t}$ which solves the SDE (\ref{i1}):%
\begin{eqnarray*}
X_{t}(x) &=&x+\sum_{j=1}^{d}\int_{0}^{t}\sigma
_{j}(X_{s}(x))dW_{s}^{j}+\int_{0}^{t}b(X_{s}(x))ds \\
&=&x+\sum_{j=1}^{d}\int_{0}^{t}\sigma _{j}(X_{s}(x))\circ
dW_{s}^{j}+\int_{0}^{t}\overline{b}(X_{s}(x))ds
\end{eqnarray*}%
Here $\circ dW^{j}$ denotes the Stratonovich integral and $\overline{b}$
designates the drift coefficient that one obtains when passing from the It%
\^{o} integral to the Stratonovich integral. We assume that the
coefficients $\sigma_{j}:\R^d\rightarrow \R^d$ and $b:\R^d\rightarrow \R^d$ are $C^\infty$, bounded with bounded derivatives of any order.

We start with the Euler scheme. It corresponds to  
\begin{equation*}
\Theta (\delta ,z,x)=x+\sum_{j=1}^{d}\sigma _{j}(x)\sqrt{\delta} z_{j}+b(x)\delta
\end{equation*}
and $Z$ being distributed as standard normal random variable on~$\R^d$. The finest discretization is $\Pi_0(\cA)=\{0=s_0<s_1<\dots<s_m=h_0=T\}$, and one therefore needs $m$ independent random variables $Z_0,\dots,Z_{m-1}$. The grids $\Pi_0(\cA_\Lambda)$ with $\Lambda \subset \cE(\cA)$ are sub-grids of $\Pi_0(\cA_\Lambda)$. For some indices $i_k$, it goes directly from $s_{i_k}$ to $s_{i_k+n}$ while  the uniform discretization of $[s_{i_k},s_{i_k+n}]$ is contained in $\Pi_0(\cA)$. One takes then $\frac{Z_{i_k}+\dots+Z_{i_k+n-1}}{\sqrt{n}}$ for the corresponding normal variable for $\Pi_0(\cA_\Lambda)$.

We now present the Ninomiya and Victoir scheme~\cite{NiVi}. We use the following notation: for a vector field $V:\R^d\rightarrow \R^d$, we define $\Phi _{V}$ to be the solution of the ODE
\begin{equation}\label{Flow_EDO}
\Phi _{V}(x,t)=x+\int_{0}^{t}V(\Phi _{V}(x,s))ds, \ t\in \R, 
\end{equation}%
and denote $\exp (tV)(x)=\Phi_{V}(x,t)$. Then, we set%
\begin{align*}
\Theta(\delta ,(z,\rho),x)=&1_{\rho=1}\exp (\frac{\delta }{2}\overline{b}%
)\circ \exp (\sqrt{\delta} z_{1}\sigma _{1})\circ ...\circ \exp (\sqrt{\delta} z_{d}\sigma _{d})\circ
\exp (\frac{\delta }{2}\overline{b})(x)\\
&+1_{\rho=0}\exp (\frac{\delta }{2}\overline{b}%
)\circ \exp (\sqrt{\delta} z_{d}\sigma _{d})\circ ...\circ \exp (\sqrt{\delta} z_{1}\sigma _{1})\circ
\exp (\frac{\delta }{2}\overline{b})(x).
\end{align*}

Let $\rho $ be a Bernoulli random variable such that $P(\rho =1)=P(\rho =0)=\frac{1}{2}$ and let $Z\sim \mathcal{N}_d(0,I_d)$ be an independent standard normal random variable on $\R^d$. Then $\Theta(\delta ,(Z,\rho),x)$ represents the
Ninomiya and Victoir scheme. We denote%
\begin{equation*}
Q_{l}f(x)=\E[f(\Theta (h_{l},(Z,\rho),x))].
\end{equation*}%
One may show, adapting for example the proof of Theorem 1.18 in~\cite{Alfonsi}, that 
\begin{equation*}
\left\Vert (P_{h_{l}}-Q_{l})f\right\Vert _{\infty }\leq C\left\Vert
f\right\Vert _{6,\infty }h_{l}^{-3}
\end{equation*}%
and more generally that (\ref{b00}) holds with $\alpha =2$ and $\beta =6$. Therefore, the tree $\mathcal{T}_{0}^{\nu }$ will be different from the one of the Euler scheme and shorter. To sample the Ninomiya and Victoir scheme on the grids $\Pi_0(\cA)$, we take a sequence 
$(Z_k,\rho _{k})_{k\in\{0,\dots, m-1\}}$ of independent copies of $(Z,\rho)$. We define the corresponding flow by 
\begin{equation*}
\overline{X}_{s_{k+1}}(x)=\Theta(\delta _{k},(Z_{k},\rho_k),\overline{X}_{s_{k}}(x)).
\end{equation*}%
For each grid $\Pi_0(\cA_\Lambda)$ with $\Lambda \subset \cE(\cA)$, we again take $\frac{Z_{i_k}+\dots+Z_{i_k+n-1}}{\sqrt{n}}$ each time that the discretization goes from $s_{i_k}$ to $s_{i_k+n}$, and take $\rho_{i_k}$ for the associated Bernoulli variable. 

\subsubsection{Approximation schemes for PDMPs}\label{subsubsec_approx_PDMP}
We consider the infinitesimal operator%
\begin{equation*}
Lf(x)=b(x)\nabla f(x)+\int_{E}(f(x+c(z,x))-f(x))\lambda (x)\nu (dz).
\end{equation*}%
Here $(E,\mathcal{E})$ is a measurable space, $\nu $ is a finite measure on $E$, $b,\lambda :\R^{d}\rightarrow \R^{d}$ are  globally Lipschitz continuous functions, $c:E\times \R^{d}\rightarrow \R^{d}$ is measurable, bounded and Lipschitz continuous with respect to~$x$, uniformly with respect to $z\in E$. We set $\bar{\lambda}(x)=\lambda(x)/\|\lambda\|_\infty$. We denote by $P_{t}$ the semigroup associated to  the infinitesimal operator~$L$. The probabilistic representation of $P_{t}$ is given in the following way.

Let $J$ be a Poisson process with intensity $\nu(E) \|\lambda\|_\infty$, and a sequence of independent random variables $(Z_{k},U_k)_{k\in \N}$ such that 
\begin{equation*}
\Px(Z_{k}\in dz)=\frac{1}{\nu (E)}\nu (dz),\quad \Px(U_{k}\in du)=1_{[0,1]}(u)du,
\end{equation*}%
and $Z_k$ being independent of $U_k$. Then, we define $X_t(x)$ as the solution of
\begin{equation}
X_{t}(x)=x+\int_{0}^{t}b(X_{s})ds+\sum_{k\leq
J_{t}}c(Z_{k},X_{T_{k}-})1_{\{U_{k}\leq \bar \lambda (X_{T_{k}-}) \} },
\label{PDMP2}
\end{equation}%
where $T_k$ is the time of the $k$-th jump of~$J$.  It is well known that under our
hypothesis the above equation has a unique solution: between two jump times $t\in \lbrack T_{k-1},T_{k})$ it follows the deterministic curve given by $dX_{t}=b(X_{t})dt$, and at time~$T_{k}$ it makes the jump $c(Z_{k},X_{T_{k}-})$ if $U_{k}\leq \bar \lambda (X_{s-})$.  This process satisfies  $P_{t}f(x)=\E(f(X_{t}(x)))$. This is one particular possible description of PDMPs. There is a huge literature concerning this type of process and their applications, see e.g.~\cite{Jacobsen}.

We now define the approximation scheme. Let $\tilde{X}_{t}(x)$ be the solution of
\begin{equation}\label{def_Xtilde_PDMP}
\tilde{X}_{t}(x)=x+\sum_{k\leq
J_{t}}c(Z_{k},\tilde{X}_{T_{k}-})1_{\{U_{k}\leq \bar \lambda (\tilde{X}_{T_{k}-}) \} },
\end{equation}
which is the solution of~\eqref{PDMP2} for $b\equiv 0$. Then, we define
$$\widehat{X}_{h}(x)=\tilde{X}_{h}(x+b(x)h) \text{ and } P_{h}^hf(x)=\E(f(\widehat{X}_{h}(x)))=\tilde{P}_hf(x+b(x)h),  $$
where $\tilde{P}_h$ is the semigroup associated to $\tilde{X}$. On the discretization  time-grid $\Pi=\{0=s_0<s_1<\dots<s_m=h_0=T\}$, this amounts to consider
\begin{equation}\label{scheme_PDMP}X^{\Pi}_{s_{i+1}}=\Theta(s_{i+1}-s_i,(J_{s_{i+1}}-J_{s_i},(Z_k,U_k)_{J_{s_i}+1\le k\le J_{s_{i+1}}}),X^{\Pi}_{s_{i}}) ,
\end{equation}
where $\Theta$ is defined recursively by $\Theta(\delta,(0,()),x)=x+b(x)\delta$ and
\begin{align*}
  &\Theta(\delta,(n+1,(z_k,u_k)_{1\le k\le n+1}),x)= 1_{u_{n+1}>\bar \lambda(\Theta(\delta,(n,(z_k,u_k)_{1\le k\le n},x)))} \Theta(\delta,(n,(z_k,u_k)_{1\le k\le n}),x)\\&+1_{u_{n+1}\le \bar \lambda(\Theta(\delta,(n,(z_k,u_k)_{1\le k\le n}),x))}[\Theta(\delta,(n,(z_k,u_k)_{1\le k\le n}),x)+c(z_{n+1},\Theta(\delta,(n,(z_k,u_k)_{1\le k\le n}),x ))].
\end{align*}
Here, the generic Polish space $\cZ$ introduced in Subsection~\ref{subsec_probrepr} is $\cup_{n\in \N}\{(n,(z_k,u_k)_{1\le k\le n}): z_k\in E, u_k\in [0,1]\}$.
Let us note that other approximation schemes are possible. The interest of~\eqref{scheme_PDMP} is that all the schemes $X^{\Pi_0(\cA_\Lambda)}$  with $\Lambda\subset \cE(\cA)$ are sampled from the same random variables $(J_t)_{t\in [0,T]}$ and $(Z_k,U_k)_{1\le k \le J_T}$ and are likely to have very similar jumps, which is interesting to reduce the variance of~$\Gamma^{\cA}_0f$.

Let us assume now that $b$, $\lambda$ and $x\mapsto c(z,x)$  are $C^\infty$, bounded with bounded derivatives (uniformly in $z$). We check that (\ref{b000}) holds in this case. To do so, we introduce $\Phi_b(x,t)$ the flow associated to~$b$, see equation~\eqref{Flow_EDO}. We have (see e.g.~\cite{Jacobsen}, Lemma 7.3.3)
\begin{align}
  P_tf(x)&=e^{-\int_0^t\lambda(\Phi_b(x,s))ds} f(\Phi_b(x,t))\label{Ptf_PDMP}\\& + \int_0^t \lambda(\Phi_b(x,s))e^{-\int_0^s\lambda(\Phi_b(x,u))du} \int_EP_{t-s}f(\Phi_b(x,s)+c(z,\Phi_b(x,s)))\nu(dz)ds. \nonumber
\end{align}
By differentiating this equation, we get by induction on $k$ (we clearly have $\|P_tf\|_\infty\le \|f\|_\infty$ for $k=0$) that
$$\exists C>0, \forall t \in [0,T], \|P_tf\|_{k,\infty}\le C\|f\|_{k,\infty},$$
using the Faà di Bruno formula and Gronwall's lemma.
This property~\eqref{b000} has been studied for more general jump SDEs very recently by Bally, Goreac and Rabiet~\cite{BGR}. The next lemma proves that~\eqref{b00} also holds with with $\alpha=1$ and $\beta=2$.
\begin{lemma}We assume that $b$, $\lambda$ and $x\mapsto c(z,x)$  are $C^\infty$, bounded with bounded derivatives, uniformly in $z$. Then, we have
  $$\forall k\in \N, \exists C \in \R_+^*, \forall f \in C^\infty_b(\R^d), \ \|P_hf-P^h_hf\|_{k,\infty}\le C \| f\|_{k+2,\infty}h^2.$$
\end{lemma}
\begin{proof}
From~\eqref{Ptf_PDMP} and $f(\Phi_b(x,t))=f(x)+\int_0^t b(\Phi_b(x,s))\nabla f(\Phi_b(x,s))ds$, we get
\begin{align*}
  P_tf(x)-f(x)&=(e^{-\int_0^t\lambda(\Phi_b(x,s))ds}-1) f(\Phi_b(x,t))+\int_0^t b(\Phi_b(x,s))\nabla f(\Phi_b(x,s))ds  \\& + \int_0^t \lambda(\Phi_b(x,s))e^{-\int_0^s\lambda(\Phi_b(x,u))du} \int_EP_{t-s}f(\Phi_b(x,s)+c(z,\Phi_b(x,s)))\nu(dz)ds,
\end{align*}
which leads to
\begin{equation}\label{controle_semigrp}\exists C>0, \forall t\in[0,T], \ \|P_tf-f\|_{k,\infty}\le C\|f\|_{k+1,\infty} t.
\end{equation}
 Since $P_tf=f+\int_0^tP_{t-s}Lf ds$, we get from~\eqref{controle_semigrp}
$\|P_tf-f-tLf\|_{k,\infty}\le Ct^2\|Lf\|_{k+1,\infty} \le Ct^2\|f\|_{k+2,\infty}$. We define $\tilde{\Phi}_b(x,t)=x+b(x)t$ which has bounded derivatives, uniformly in $t\in[0,T]$. Let $\tilde{L}f(x)=\int_{E}(f(x+c(z,x))-f(x))\lambda (x)\nu (dz)$ be the infinitesimal generator of~\eqref{def_Xtilde_PDMP}. We have similarly 
 $$\|\tilde{P}_t(f\circ\tilde{\Phi}_b(\cdot,t))  - f\circ\tilde{\Phi}_b(\cdot,t) -t\tilde{L}(f\circ\tilde{\Phi}_b(\cdot,t))\|_{k,\infty}\le Ct^2\|f\circ\tilde{\Phi}_b(\cdot,t)\|_{k+2,\infty}\le Ct^2\|f\|_{k+2,\infty}.$$
 We obviously have $\|f\circ\tilde{\Phi}_b(\cdot,t)-f -tb\nabla f\|_{k,\infty}\le  Ct^2\|f\|_{k+2,\infty}.$ Since $\|f\circ\tilde{\Phi}_b(\cdot,t)-f\|_{k,\infty}\le  Ct\|f\|_{k+1,\infty}$, we also have $\|t\tilde{L}f\circ\tilde{\Phi}_b(\cdot,t)-t(Lf-b\nabla f)\|_{k,\infty}\le  Ct^2\|f\|_{k+1,\infty}$, which yields to $\|P_hf-P^h_hf\|_{k,\infty}\le C \| f\|_{k+2,\infty}h^2$.
\end{proof}

\subsection{Estimates of the variance on the Euler scheme for SDEs}\label{subsec_variance}

The aim of this section is to estimate the variance of the algorithm given
in (\ref{e1}), (\ref{e8}) in the case of the Euler scheme for SDEs. We consider the $\R^{d}$ valued diffusion process solution of the SDE~\eqref{i1}. Given a random grid 
\begin{equation*}
\Pi (\omega )=\{0=s_{0}(\omega )<s_{1}(\omega )<....<s_{n(\omega )}(\omega
)<T\},
\end{equation*}%
we construct the corresponding Euler scheme by $X_{0}^{\Pi }=x$ and 
\begin{equation*}
X_{s_{i+1}}^{\Pi }=X_{s_{i}}^{\Pi }+\sum_{j=1}^{d}\sigma _{j}(X_{s_{i}}^{\Pi
})(W_{s_{i+1}}^{j}-W_{s_{i}}^{j})+b(X_{s_{i}}^{\Pi })(s_{i+1}-s_{i}).
\end{equation*}%
In (\ref{e8}), we have constructed an approximation scheme based on a linear
combination of $X_{T}^{\Pi _{0}(\mathcal{A}_{\Lambda })}(x).$ We use here
all the notation introduced there. We denote

\begin{equation}\label{def_XiA}
\Upsilon_{\mathcal{A}}f(x)=\sum_{\Lambda \subset \mathcal{E(A)}}(-1)^{\left\vert
\Lambda \right\vert }f(X_{T}^{\Pi _{0}(\mathcal{A}_{\Lambda })}(x)).
\end{equation}

\begin{remark}
The important point here is that all the Euler schemes $X_{T}^{\Pi _{0}(\mathcal{A}_{\Lambda })}(x)$ for $\Lambda \subset \mathcal{E(A)}$ are defined on
the same probability space and constructed with the same Brownian motion~$W$. Thus, all the values of $X_{T}^{\Pi _{0}(\mathcal{A}_{\Lambda })}(x)$ are close. When summing according to~\eqref{def_XiA}, we may then expect that $\Upsilon_{\mathcal{A}}f(x)$ is small with a small variance. This is precised in the next proposition. 
\end{remark}

\begin{theorem}
\label{varience} Suppose that $\sigma_{j},b\in C_{b}^{\infty }(\R^{d})$. Then, we have for any $f \in C_{b}^{\infty }(\R^{d})$,
\begin{equation}
\E(\Upsilon _{\mathcal{A}}^{2}f(x))\leq \frac{C}{n^{2 \sum_{u\in \mathcal{E(A)}}\left\vert u\right\vert }}.  \label{v1}
\end{equation}%
In particular, we have $\E[(c(\mathcal{A)}\Upsilon _{\mathcal{A}}f(x))^2]\le C$ and thus $Var[c(\mathcal{A)}\Upsilon _{\mathcal{A}}f(x)]\le C$.
\end{theorem}

\bigskip

The proof needs some preparation. We use the alternative representation
of the random grids $G_{0}(\mathcal{A})$ and $G_{0}(\mathcal{A}_\Lambda)$ with $\Lambda\subset \mathcal{E(A)}$ given by Lemma~\ref{lem_grids} and~\eqref{G8}. 
We recall that $\Pi_{0}(\mathcal{A})$ is the ordered grid $G_{0}(\mathcal{A})$:
\begin{equation*}
\Pi_{0}(\mathcal{A})=\{0=s_{0}<s_{1}<\dots<s_{m}=h_0=T\}.
\end{equation*}%
We write $\mathcal{E(A)}=\{u^{(1)},\dots,u^{(r)}\}$ with $r=\card(\mathcal{E(A)})$. For $k\in \{1,\dots, r\}$, there exists $i_k\in \{0,\dots,m\}$ such that $s_{i_{k}}=t_{0}(u^{(k)})$. We check that these indices are distinct, and we assume without loss of generality that $i_{1}<\dots<i_{r}$.   These are the "extreme times". We also denote
$$j_{k}=i_{k}+n, \ k \in \{1,\dots, r \}, \  j_0=0.$$

\begin{example} We consider $n=3$, $\cA=\{\emptyset,1,2,21\}$ with $\kappa(\emptyset)=(0,2)$ and $\kappa(2)=1$. The grid $\Pi_0(\cA)$ is drawn below to scale, with $s_{i}-s_{i-1}=T/n^{l_i}$, $l_i\in\{1,2,3\}$. 
  \begin{center}
 \begin{tikzpicture}
        \draw[|-] (0,0) -- (3,0);
        \draw[|-] (3,0) -- (6,0);
        \draw[|-|] (6,0) -- (9,0);
        
        \draw[-|] (0,0) -- (1,0);
        \draw[-|] (1,0) -- (2,0);

        \draw[-|] (6,0) -- (7,0);
        \draw[-|] (7,0) -- (8,0);

        \draw[-|] (7,0)   -- (7.33333333334,0);
        \draw[-|] (7.4,0) -- (7.66666666667,0);

        \draw (9,0) node[above] {$T$};
        \draw (0,0) node[above] {$0$};
        
        \draw (0,0) node[below] {$s_0$};
        \draw (1,0) node[below] {$s_1$};
        \draw (2,0) node[below] {$s_2$};
        \draw (3,0) node[below] {$s_3$};
        \draw (6,0) node[below] {$s_4$};
        \draw (7,0) node[below] {$s_5$};
        \draw (7.33333333334,0) node[below] {$s_6$};
        \draw (7.66666666667,0) node[below] {$s_7$};
        \draw (8,0) node[below] {$s_8$};
        \draw (9,0) node[below] {$s_9$};
\end{tikzpicture}
  \end{center}
  On this example, we have $\cE(\cA)=\{1,21 \}$, $r=2$, $i_1=0$, $j_1=3$, $i_2=5$ and $j_2=8$. The three other grids $\Pi_{0}(\mathcal{A}_\Lambda)$ needed in the computation of~\eqref{def_XiA} are
  \begin{center}
    \begin{tabular}{l}
 \begin{tikzpicture}
        \draw[|-] (0,0) -- (3,0);
        \draw[|-] (3,0) -- (6,0);
        \draw[|-|] (6,0) -- (9,0);

        \draw[-|] (6,0) -- (7,0);
        \draw[-|] (7,0) -- (8,0);

        \draw[-|] (7,0)   -- (7.33333333334,0);
        \draw[-|] (7.4,0) -- (7.66666666667,0);

        \draw (9,0) node[above] {$T$};
        \draw (0,0) node[above] {$0$};
        
        \draw (0,0) node[below] {$s_0$};       
        \draw (3,0) node[below] {$s_3$};
        \draw (6,0) node[below] {$s_4$};
        \draw (7,0) node[below] {$s_5$};
        \draw (7.33333333334,0) node[below] {$s_6$};
        \draw (7.66666666667,0) node[below] {$s_7$};
        \draw (8,0) node[below] {$s_8$};
        \draw (9,0) node[below] {$s_9$};
        \draw (9,0) node[right] {$\quad \Lambda=\{1\}$,};
 \end{tikzpicture}
\\
 \begin{tikzpicture}
        \draw[|-] (0,0) -- (3,0);
        \draw[|-] (3,0) -- (6,0);
        \draw[|-|] (6,0) -- (9,0);
        
        \draw[-|] (0,0) -- (1,0);
        \draw[-|] (1,0) -- (2,0);

        \draw[-|] (6,0) -- (7,0);
        \draw[-|] (7,0) -- (8,0);

        \draw (9,0) node[above] {$T$};
        \draw (0,0) node[above] {$0$};
        
        \draw (0,0) node[below] {$s_0$};
        \draw (1,0) node[below] {$s_1$};
        \draw (2,0) node[below] {$s_2$};
        \draw (3,0) node[below] {$s_3$};
        \draw (6,0) node[below] {$s_4$};
        \draw (7,0) node[below] {$s_5$};
       
        \draw (8,0) node[below] {$s_8$};
        \draw (9,0) node[below] {$s_9$};
         \draw (9,0) node[right] {$\quad \Lambda=\{21\}$,};
 \end{tikzpicture}
 \\
 \begin{tikzpicture}
        \draw[|-] (0,0) -- (3,0);
        \draw[|-] (3,0) -- (6,0);
        \draw[|-|] (6,0) -- (9,0);

        \draw[-|] (6,0) -- (7,0);
        \draw[-|] (7,0) -- (8,0);

        \draw (9,0) node[above] {$T$};
        \draw (0,0) node[above] {$0$};
        
        \draw (0,0) node[below] {$s_0$};
        \draw (3,0) node[below] {$s_3$};
        \draw (6,0) node[below] {$s_4$};
        \draw (7,0) node[below] {$s_5$};
       
        \draw (8,0) node[below] {$s_8$};
        \draw (9,0) node[below] {$s_9$};
         \draw (9,0) node[right] {$\quad \Lambda=\{1,21\}$.};
\end{tikzpicture}
    \end{tabular}
    \end{center}

\end{example}

Notice that the grid $\Pi_{0}(\mathcal{A})$ contains, by
construction, all the points in the uniform grid on $[s_{i_{k}},s_{j_{k}}]$,
that is $s_{i_{k}+j},$ with $j=1,\dots,n$. But, if $u^{(k)}\in \Lambda$, the grid
$\Pi_{0}(\mathcal{A}_{\Lambda })$ is not refined between $s_{i_k}$ and $s_{j_k}$ and does not contain $s_{i_{k}+j}$ with $j=1,\dots,n-1$. We deduce the next lemma.

\begin{lemma}\label{lem_pi0}
  Let $\Lambda\subset \mathcal{E(A)}$. We note $\Lambda=\{u^{(k_1)},\dots, u^{(k_\ell)}\}$ with $\ell=\card(\Lambda)$ and set
  $\mathcal{I}^\Lambda= \{0,\dots,m \} \setminus \left(\cup_{\ell'=1}^\ell \{i_{k_{\ell'}}+1,\dots, j_{k_{\ell'}}-1\} \right)$. Then, $\Pi_0(A_\Lambda)=\{s_i,i\in \mathcal{I}\}$.
\end{lemma}

We also recall that, in order to construct our scheme (see (\ref{e1})), we
have considered a kernel $\Theta _{\rho }$ and a sequence of independent
random variables $Z_{k}$ and $\rho _{k},$ and we have defined the vector
fields $\theta _{k}(x)=\Theta _{\rho _{k}}(\delta _{k},\sqrt{\delta _{k}}Z_{k},x)$ with $\delta _{k}=s_{k+1}-s_{k}$ (see (\ref{e1a})). In the case of the
Euler scheme, we have a special representation of these random variables
and of these operators. We define%
\begin{equation*}
\theta_{k}(x)=x+\sum_{j=1}^{d}\sigma
_{j}(x)(W_{s_{k+1}}^{j}-W_{s_{k}}^{j})+b(x)(s_{k+1}-s_{k}).
\end{equation*}%
This corresponds to the quantity defined in (\ref{e1a}) with $\delta_{k}=s_{k+1}-s_{k}$ and $Z_{k}=(s_{k+1}-s_{k})^{-1/2}(W_{s_{k+1}}^{j}-W_{s_{k}}^{j}).$ Moreover, for $k=1,...,r=\card(\mathcal{E(A)})$ we define 
\begin{equation*}
\Psi_{k}(x)=\theta_{j_{k-1}}\circ \theta _{j_{k-1}+1}\circ \dots \circ \theta
_{i_{k}-1}(x)=\left(\prod_{j=j_{k-1}}^{i_{k}-1}\right)\theta _{j}(x),
\end{equation*}%
with the convention $\left(\prod_{j=j_{k-1}}^{i_{k}-1}\right)\theta _{j}(x)=x$ if $i_k=j_{k-1}$. This represents the flow of the approximation scheme which runs from $s_{j_{k-1}}$ to $s_{i_{k}}$, and that is common to all the grids $\Pi_{0}(\mathcal{A}_\Lambda)$ for $\Lambda\subset \mathcal{E(A)}$. We also define the flow between $s_{j_{k}}$ and $s_m=T$:
 $$\Psi_{r+1}(x)=\left(\prod_{j=j_{k}}^{m-1}\theta _{j}\right) (x).$$
We now specify if we use or not the refined grid on the interval $[s_{i_{k}},s_{j_{k}}]$. For the case where the grid is refined (i.e. when $u^{(k)}\not \in \Lambda$),
  we define 
\begin{equation*}
\phi _{k}(x)=\left(\prod_{j=i_{k}}^{j_{k}-1}\theta _{j} \right)\circ \Psi _{k}(x).
\end{equation*}%
This is the Euler scheme which starts from $\Psi _{k}(x)$ and runs from $s_{i_{k}}$ to $s_{j_{k}}=s_{i_{k}+n}$ using the uniform step. Instead, for the coarse discretization which goes from $s_{i_{k}}$ to $s_{i_{k}+n}$ directly in one single step (i.e.  when $u^{(k)} \in \Lambda$), we set  %
\begin{equation*}
\Phi _{k}(x)=\Theta_k\circ\Psi _{k}(x),\text{ with } \Theta_k(x)=x+\sum_{j=1}^{d}\sigma _{j}(x)(W_{s_{j_{k}}}^{j}-W_{s_{i_{k}}}^{j})+b(x)(s_{j_{k}}-s_{i_{k}}).
\end{equation*}
Now, we are able to define the flow of the whole Euler scheme on the grid $\Pi_0(\cA_\Lambda)$. 
If $u^{(k)}\in \Lambda $, we use $\Phi_{k}$ in order to go from $s_{j_{k-1}}$ to $s_{j_k}$. Instead, if $u^{(k)}\in \Lambda$ we use $\phi _{k}$. Thus,  we
define, for $k=1,...,r$ 
\begin{eqnarray*}
\theta _{k}^{\Lambda } &=&1_{u^{(k)}\in \Lambda}\Phi_{k} + 1_{u^{(k)}\not \in \Lambda}\phi_{k}.
\end{eqnarray*}%
From Lemma~\ref{lem_pi0}, we get
\begin{equation*}
X_{T}^{\Pi _{0}(\mathcal{A}_{\Lambda })}(x)=\Psi _{r+1}\circ \theta_{(r)}^{\Lambda }, \text{ with } \theta _{(r)}^{\Lambda }=\theta _{r}^{\Lambda }\circ ...\circ \theta
_{1}^{\Lambda }.
\end{equation*}%
As a consequence, we have
\begin{equation*}
\Upsilon _{\mathcal{A}}f(x)=\sum_{\Lambda \subset \mathcal{E(A)}}(-1)^{\card(\Lambda) }f(X_{T}^{\Pi _{0}(\mathcal{A}_{\Lambda
})}(x))=\sum_{\Lambda \subset \mathcal{E(A)}}(-1)^{\card( \Lambda)
 }f(\Psi _{r}\circ \theta _{(r)}^{\Lambda })(x)=\Gamma
_{r}^{\varnothing }(f\circ \Psi _{r})(x)
\end{equation*}%
with $r=\card(\mathcal{E(A)}) $ and $\Gamma _{r}^{\varnothing
}(f\circ \Psi _{r})$ defined in (\ref{APP1d}). We are now in the framework
of Appendix A. The above formula has to be understood in the following
way: $\theta _{k}^{\Lambda }$ represents the approximating flow associated
to the grid $\Pi _{0}(\mathcal{A}_{\Lambda })$ which runs from $s_{j_{k-1}}$
to $s_{i_{k}}.$ So when $k=r,$ we arrive in $s_{i_{r}}.$ This is the last
"extreme time". After this, we run with $\Psi _{r+1}$ up to $s_{m}=T.$

\begin{lemma}\label{lem_simu}
  With the notation above, we  define families $\cX_k$ of $2^k$ elements of $\R^d\times \{-1,1\}$ as follows. We set $\cX_0=\{(x,1)\}$ and for $k\in \{1,\dots,r\}$, we define 
  $$\cX_k=\{(\phi_k(x^{k-1}_j),\epsilon^{k-1}_j), 1\le j\le 2^{k-1}\} \cup \{(\Phi_k(x^{k-1}_j),-\epsilon^{k-1}_j), 1\le j\le 2^{k-1}\},$$
  where $\cX_{k-1}=\{(x^{k-1}_j,\epsilon^{k-1}_j),1\le  j\le 2^{k-1}\}$ and $\cup$ has to be understood as the concatenation symbol. Then, we have
  $$\sum_{\Lambda \subset \mathcal{E(A)}}(-1)^{\card(\Lambda) }f(X_{T}^{\Pi _{0}(\mathcal{A}_{\Lambda
})}(x))=\sum_{j=1}^{2^r}\epsilon_j^r f(\Psi _{r+1}(x^{r}_j)). $$
\end{lemma}
Lemma~\ref{lem_simu} is obvious but important for simulation purposes: by branching, it is possible to simulate at the same time all the values of $(X_{T}^{\Pi_{0}(\mathcal{A}_{\Lambda})}(x),(-1)^{\card(\Lambda)})$ as explained in Subsection~\ref{Subsec_Impl}. More precisely, there is no need to store $\Lambda$: adding the sign~$\epsilon$ to the state space makes the branching dynamics Markovian.

\begin{proof}[Proof of Theorem~\ref{varience}]
Our aim now is to check that $\Phi _{k}$ and $\phi _{k}$ verify the
hypothesis of Proposition~\ref{appendix}. Standard estimates concerning Euler schemes
(see the short sketch below) give 
\begin{equation}
\left\Vert \Phi_{k}\right\Vert _{q,p,\infty }:=\sup_{x}\sum_{\left\vert
\alpha \right\vert \leq q}(\E(\left\vert \partial _{x}^{\alpha }\Phi
_{k}(x)\right\vert ^{p}))^{1/p}<\infty   \label{e7}
\end{equation}%
and the same estimate holds for $\phi _{k}.$ Then, as a consequence of
Proposition \ref{appendix} we obtain 
\begin{equation}
\sup_{x}\E\left[ \left(\sum_{\Lambda \subset \mathcal{E(A)}}(-1)^{\left\vert \Lambda
\right\vert }f(X_{T}^{\Pi _{l}(\mathcal{A}_{\Lambda })}(x)) \right)^{2} \right]\leq C\
\prod_{k=1}^{r}\left\Vert \Phi _{k}-\phi _{k}\right\Vert _{r,4,\infty }^{2}
\label{v2}
\end{equation}%
Notice that $s_{j_{k}}-s_{i_{k}}=Tn^{-\left\vert u^{(k)}\right\vert }$ where $s_{i_{k}}=t_{0}(u^{(k)})$. Thus, from Lemma~\ref{lemma_flow_der} we get easily that $\left\Vert \Theta_{k}-\prod_{j=i_{k}}^{j_{k}-1}\theta _{j}\right\Vert _{r,4,\infty }\leq \frac{C}{n^{\left\vert u^{(k)}\right\vert }}$ and then
\begin{equation}
\left\Vert \Phi _{k}-\phi _{k}\right\Vert _{r,4,\infty }\leq \frac{C}{%
n^{\left\vert u^{(k)}\right\vert }},  \label{v3}
\end{equation}%
by using the Faà di Bruno formula. Then (\ref{v1}) follows. Moreover, we have  
\begin{equation*}
c(\mathcal{A})=\prod_{u\in \mathcal{A}}\binom{n}{j_{u}(\mathcal{A})}%
\leq n^{\sum_{u\in \mathcal{A}}j_{u}(\mathcal{A})}.
\end{equation*}%
One checks easily by induction that $\sum_{u\in \mathcal{A}}j_{u}(\mathcal{A})\leq \sum_{u\in \mathcal{E}(\mathcal{A})}\left\vert u\right\vert $, which gives~(\ref{v1}).

We now give a sketch of the proof of (\ref{e7}). We consider the grid $\Pi _{0}(\mathcal{A})=\{0=s_{0}<s_{1}<....<s_{m}=T \}$ given at the beginning of this section.
The corresponding Euler scheme on $[0,T]$ is defined by%
\begin{equation*}
X_{t}(x)=x+\sum_{j=1}^{d}\int_{0}^{t}\sigma _{j}(X_{\tau
(s)}(x))dW_{s}^{j}+\int_{0}^{t}b(X_{\tau (s)}(x))ds,\quad 0\leq t\leq h_{l}
\end{equation*}%
where $\tau (s)=s_{i}$ for $s\in[s_{i}, s_{i+1})$. Then, we have $\Phi_{k}(x)=X_{h_{l}}(x).$ Using Burkholder-Davis-Gundy inequality and the fact that the
coefficients are bounded, we get $\E(\left\vert X_{r}(x)-x\right\vert
^{p})\leq C$. Moreover, the first derivatives satisfy%
\begin{equation*}
\nabla X_{t}(x)=I+\sum_{j=1}^{d}\int_{0}^{t}\nabla \sigma _{j}(X_{\tau
(s)}(x))\nabla X_{\tau (s)}(x)dW_{s}^{j}+\int_{0}^{t}\nabla b(X_{\tau
(s)}(x))\nabla  X_{\tau (s)}(x)ds.
\end{equation*}%
Since $\nabla \sigma _{j}$ and $\nabla b$ are bounded, using  Burkholder-Davis-Gundy
inequality and Gronwall's lemma we get $\E(\left\vert \nabla X_{r}(x)\right\vert ^{p})\leq C.$ For higher order derivatives, the proof is
similar.\end{proof}

\begin{remark}\label{Rk_variance_terms} We have a better estimate for~\eqref{v3} when $\sigma(x)$ is constant. In this case, we have from Lemma~\ref{lemma_flow_der}
  $$\left\Vert \Phi _{k}-\phi _{k}\right\Vert _{1,p,\infty }\leq \frac{C}{ n^{ \frac32   \left\vert u^{(k)}\right\vert }},  $$
  which leads to get $\E(\Upsilon _{\mathcal{A}}^{2}f(x))\leq \frac{C}{n^{3 \sum_{u\in \mathcal{E(A)}}\left\vert u\right\vert }}$ instead of~\eqref{v1}. When $\sigma(x)=0$, we even have $\left\Vert \Phi _{k}-\phi _{k}\right\Vert _{1,p,\infty }\leq \frac{C}{    n^{  2   \left\vert u^{(k)}\right\vert }}$ and thus  $\E(\Upsilon _{\mathcal{A}}^{2}f(x))\leq \frac{C}{n^{4 \sum_{u\in \mathcal{E(A)}}\left\vert u\right\vert }}$.
\end{remark}
\begin{remark}\label{Rk_avoid_calculations}
Since $c(\cA)=O(n^{\sum_{u\in \mathcal{E(A)}}\left\vert u\right\vert })$, we get $\E(|c(\cA)\Upsilon_{\mathcal{A}}f(x)|) =O(n^{(1-a)\sum_{u\in \mathcal{E(A)}}\left\vert u\right\vert })$ with $a=2$ if $\sigma=0$, $a=3/2$ when $\sigma$ is a constant function.  Thus, the computation of some terms in the sum~\eqref{D5} is useless:  we can drop the terms $\E[\Gamma^{\cA}_0]$ for any $\cA$ such that
  $(a-1) \sum_{u\in \mathcal{E(A)}}\left\vert u\right\vert \ge \nu$. More precisely,
$\hat{Q}'_{T}(\mathcal{T}_{0}^{\nu })=Q_{0}+\sum_{\mathcal{A}\in \mathbf{F}(  \mathcal{T}_{0}^{\nu }) : (a-1) \sum_{u\in \mathcal{E(A)}}\left\vert u\right\vert < \nu }c(\mathcal{A})\E[\Gamma^{\cA}_0]$
  also satisfies
  $\left\Vert (\hat{Q}'_{T}(\mathcal{T}_{0}^{\nu })-P_{T})f\right\Vert _{\infty}\leq C_{l}\left\Vert f\right\Vert _{k(0,\nu ),\infty }n^{-\nu }$.
  
  For example, the tree $\cA=\{\emptyset,1,11,2,21\} \in \mathbf{F}(\cT^4_0)$ is such that $\sum_{u\in \mathcal{E(A)}}\left\vert u\right\vert =4$ and its calculation is useless for an approximation of order~$4$ for ODEs ($\sigma=0$).     
\end{remark}

\section{Numerical results}\label{Sec_Num}

\subsection{ Implementation }\label{Subsec_Impl}

First, we have to calculate the tree $\mathcal{T}^\nu_0$ given by Equation~\eqref{Tree} in function of the desired order $\nu$ of convergence. To calculate this tree, we only have to know $\nu$ and the coefficient~$\alpha$ that characterizes the order of convergence of the elementary scheme (see $(H_1)$ hypothesis). For the Euler scheme, we have $\alpha=1$. For example, the tree corresponding to the approximations of order $\nu=4$ and $\nu=6$ constructed with the Euler scheme are given in Figure~\ref{Fig_arbres46}. To compute these trees, we use the induction formula~\eqref{Tree}. To help the reader, we have indicated in the node the convergence order (i.e. the value of $q_i(l,\nu)$ in~\eqref{Tree}) needed in the induction. For example, $q_1(0,4)=5$, $q_2(0,4)=4$ and $q_3(0,4)=3$ are the value indicated for the sons of the ancestor of the tree $\mathcal{T}^4_0$.

\begin{figure}[h]
  \begin{subfigure}{0.35\textwidth} 
  \centering
\begin{tikzpicture}[nodes={draw, circle}, ->] 
\node {4}
child {
node {5}
child {
node {6}
child {
node {7}
}
}
}
child {
node {4}
child {
node {5}
}
}
child {
node {3}
};
\end{tikzpicture}
  \end{subfigure}
    \begin{subfigure}{0.5\textwidth} 
 \centering
\begin{tikzpicture}[nodes={draw, circle}, ->] 
\node {6}
child {
node {7}
child {
node {8}
child {
node {9}
child {
node {10}
child {
node {11}
}
}
}
}
child {
node {5}
}
}
child [ missing ]
child {
node {6}
child {
node {7}
child {
node {8}
child {
node {9}
}
}
}
}
child {
node {5}
child {
node {6}
child {
node {7}
}
}
}
child {
node {4}
child {
node {5}
}
}
child {
node {3}
};
\end{tikzpicture}
  \end{subfigure}

\caption{The trees $\mathcal{T}^4_0$ (left) and  $\mathcal{T}^6_0$ for the Euler scheme (or any elementary scheme with $\alpha=1$). }\label{Fig_arbres46}
\end{figure}
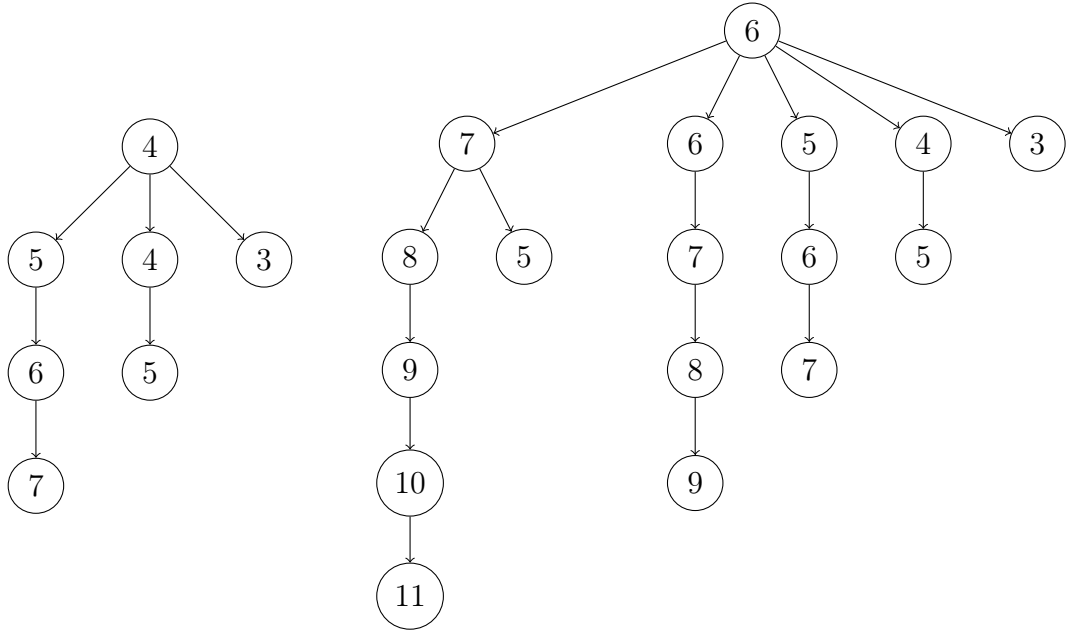

The second step consists in calculating the forest $\mathbf{F}(\mathcal{T}^\nu_0)$. According to Proposition~\ref{delta}, each tree of this forest represents a combination of elementary schemes. For example, using the Neveu notation, we have for the Euler scheme
\begin{align*}
  \mathbf{F}(\mathcal{T}^4_0)=&\{ \{\emptyset \},
\{\emptyset, 1\},
\{\emptyset, 1, 11 \},
\{\emptyset, 1, 11, 111 \},
\{\emptyset, 1, 2\}, 
\{\emptyset, 1, 11, 2\},
\{\emptyset, 1, 2, 21 \}, \\&
\{\emptyset, 1, 12, 2, 21 \},
\{\emptyset, 1, 2, 3\}
\}.
\end{align*}
Let us note that the number of trees in the forest $\mathbf{F}(\mathcal{T}^\nu_0)$ increases rapidly with $\nu$: for the Euler scheme ($\alpha=1$), we have $\card(\mathbf{F}(\mathcal{T}^4_0))=9$, $\card(\mathbf{F}(\mathcal{T}^6_0))=67$, $\card(\mathbf{F}(\mathcal{T}^{10}_0)=29135$. Nonetheless, these forests can be calculated once and for all.

The last step consist in calculating $\mathbb{E}[\Gamma^{\mathcal{A}}_0]$ for all the trees $\mathcal{A}\in \mathbf{F}(\mathcal{T}^\nu_0)$. Then, we get the approximation by using~\eqref{D5}. The key point here is to sample all the Euler schemes from the same Brownian path, as explained in Section~\ref{Sec_Prob_Repr}. Figure~\ref{Fig_grids} gives an illustration of the times grids that are involved in the calculation of $\Gamma^{\mathcal{A}}_0$, with  $\mathcal{A}=\{\emptyset, 1, 11, 2\}$.
\begin{figure}[h]
  \begin{subfigure}{0.35\textwidth} 
  \centering
\begin{tikzpicture}[nodes={draw, circle}, ->] 
\node {$\emptyset$}
child {
node {1}
child {
node {\bf 11}
}
}
child {
node {\bf 2}
};
\end{tikzpicture}
  \end{subfigure}
    \begin{subfigure}{0.5\textwidth} 
      \centering
      \begin{tikzpicture}
        \draw[|-] (0,0) -- (1,0);
        \draw[|-] (1,0) -- (2,0);
        \draw[|-] (2,0) -- (3,0);
        \draw[|-] (3,0) -- (4,0);
        \draw[|-|] (4,0) -- (5,0);    
        \draw[-|] (1.,0) -- (1.2,0);
        \draw[-|] (1.2,0) -- (1.4,0);
        \draw[-|] (1.4,0) -- (1.6,0);
        \draw[-|] (1.6,0) -- (1.8,0);
        \draw (5,0) node[right] {$G_0(\mathcal{A}_{\mathcal{E}(\mathcal{A})}) \ \  +1$};

        \draw[|-|] (0,1) -- (1,1);
        \draw[-|] (1,1) -- (2,1);
        \draw[-|] (2,1) -- (3,1);
        \draw[-|] (3,1) -- (4,1);
        \draw[-|] (4,1) -- (5,1);
        \draw[-|] (1.,1) -- (1.2,1);
        \draw[-|] (1.2,1) -- (1.4,1);
        \draw[-|] (1.4,1) -- (1.6,1);
        \draw[-|] (1.6,1) -- (1.8,1);
        \draw[-|] (4.,1) -- (4.2,1);
        \draw[-|] (4.2,1) -- (4.4,1);
        \draw[-|] (4.4,1) -- (4.6,1);
        \draw[-|] (4.6,1) -- (4.8,1);
        \draw (5,1) node[right] {$G_0(\mathcal{A}_{\{11\}} ) \ \ -1$};

        \draw[|-|] (0,2) -- (1,2);
        \draw[-|] (1,2) -- (2,2);
        \draw[-|] (2,2) -- (3,2);
        \draw[-|] (3,2) -- (4,2);
        \draw[-|] (4,2) -- (5,2);
        \draw[-|] (1.,2) -- (1.2,2);
        \draw[-|] (1.2,2) -- (1.4,2);
        \draw[-|] (1.4,2) -- (1.6,2);
        \draw[-|] (1.6,2) -- (1.8,2);
        \draw[-|] (1.0,2) -- (1.04,2);
        \draw[-|] (1.04,2) -- (1.08,2);
        \draw[-|] (1.08,2) -- (1.12,2);
        \draw[-|] (1.12,2) -- (1.16,2);
        \draw (5,2) node[right] {$G_0(\mathcal{A}_{\{2\}} ) \ \ \  -1$};

        \draw[|-|] (0,3) -- (1,3);
        \draw[-|] (1,3) -- (2,3);
        \draw[-|] (2,3) -- (3,3);
        \draw[-|] (3,3) -- (4,3);
        \draw[-|] (4,3) -- (5,3);
        \draw[-|] (1.,3) -- (1.2,3);
        \draw[-|] (1.2,3) -- (1.4,3);
        \draw[-|] (1.4,3) -- (1.6,3);
        \draw[-|] (1.6,3) -- (1.8,3);
        \draw[-|] (4.,3) -- (4.2,3);
        \draw[-|] (4.2,3) -- (4.4,3);
        \draw[-|] (4.4,3) -- (4.6,3);
        \draw[-|] (4.6,3) -- (4.8,3);
        \draw[-|] (1.0,3) -- (1.04,3);
        \draw[-|] (1.04,3) -- (1.08,3);
        \draw[-|] (1.08,3) -- (1.12,3);
        \draw[-|] (1.12,3) -- (1.16,3);
        \draw (5,3) node[right] {$G_0(\mathcal{A} ) \quad \ \ \ +1$};

\end{tikzpicture}
\end{subfigure} 
    \caption{On the left, we have represented the tree $\mathcal{A}=\{\emptyset, 1, 11, 2\}$ with its leafs $\mathcal{E}(\mathcal{A})=\{ 11, 2\}$ in bold font. On the right, we have indicated the four corresponding time-grids with their weights $\pm 1$ that are used in the calculation of $\Gamma^{\mathcal{A}}_0$ for $n=5$ on the event $\kappa(\emptyset)=(1,4)$ and $\kappa(1)=0$. (Recall that $\kappa(\emptyset)$ is a uniform r.v. on $\{(k,l): 0\le k <l<n\}$ and  $\kappa(1)$ is a uniform r.v. on $\{k: 0\le k <n\}$.)
}\label{Fig_grids}    
\end{figure}
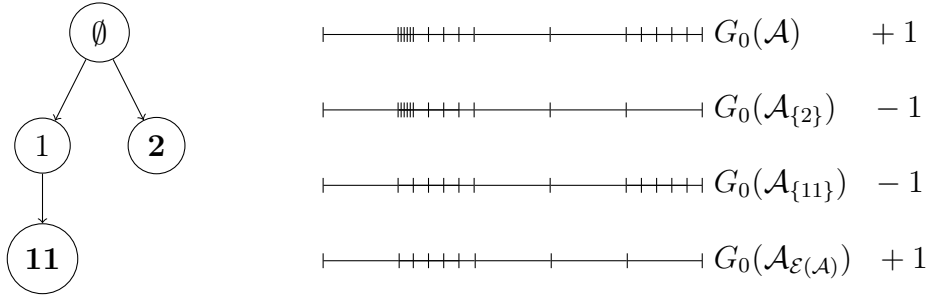
To implement the Euler schemes involved in~$\Gamma^{\mathcal{A}}$, it is possible to do it ``by hands'', i.e. to generate the random tree and then to simulate simultaneously the $2^{\card(\mathcal{E}(\mathcal{A}))}$ schemes. This is easy to do for rather small trees~$\mathcal{A}$, but the drawback is that it requires to write a routine for each $\mathcal{A} \in \mathbf{F}(\mathcal{T}^\nu_0)$. Thus, it is easy to do this direct implementation up to order three, but then it becomes rather cumbersome since the number of routines needed is rather large. Instead of this, it is possible to write a recursive routine that works for any $\mathcal{A}$. This routine starts from one initial value and calculates at the same time the $2^{\card(\mathcal{E}(\mathcal{A}))}$ schemes and branches each time it finds a leaf. It also calculates inductively the weight $\pm 1$ associated to each scheme. This routine works as follows. It takes in arguments a tree $\mathcal{A}$, a step $h_l$, and a set $\mathcal{X}=\{(x_i,\epsilon_i), 1\le i\le 2^M \}$ of initial values $x_i$ with weights $\epsilon_i \in \{-1,+1 \}$. If $\mathcal{A}=\{\emptyset\}$, it samples independent increments $(W_{kh_l/n}-W_{(k-1)h_l/n})_{1\le k\le n}$ and calculate for each $i$, the Euler scheme $\hat{X}^c_{i,h_l}$ on the coarse grid with time step $h_l$ starting from $x_i$ and the Euler scheme $\hat{X}^f_{i,h_l}$ on the fine grid with time step $h_l/n$ starting from $x_i$. It returns the set of~$2^{M+1}$ values
$$\{(\hat{X}^c_{i,h_l},\epsilon_i), 1\le i\le 2^M \} \cup\{(\hat{X}^f_{i,h_l},-\epsilon_i), 1\le i\le 2^M \}.$$
Otherwise, we have $\mathcal{A}=\{\emptyset,1\mathcal{A}'_1, \dots,r\mathcal{A}'_r\}$ with $r\le n$. We draw $\kappa(\emptyset)=(\kappa_1,\dots,\kappa_r)$ a uniform random variable on $\{(k_1,\dots,k_r):0\le k_1<\dots<k_r<n\}$ (see Remark~\ref{sampling_unifrv}). Then, we apply to all the initial values $k_1$~times the Euler scheme with time step~$h_{l+1}=h_l/n$, conserving their weights. They are used as argument to apply inductively the function with $\mathcal{A}'_1$ and $h_{l+1}$. This generates a set of values and weights to which we apply $k_2-k_1$~times the Euler scheme with time step~$h_{l+1}$, and then we apply again inductively the function  with $\mathcal{A}'_2$ and $h_{l+1}$. We repeat this~$r$ times, and finally apply $n-(k_r+1)$~times the Euler scheme with time step~$h_{l+1}$. This inductive algorithm consists precisely in implementing the formula given in Lemma~\ref{lem_simu}.

\begin{remark}\label{sampling_unifrv}To sample a uniform random variable on $\mathcal{S}_{r}:=\{(k_1,\dots,k_r):0\le k_1<\dots<k_r<n\}$ for $r\in \{1,\dots, n\}$, we can proceed as follows. If $r=1$, we simply draw a uniform r.v. on  $\{0,\dots, n-1\}$. For $r\ge 2$, we proceed by induction and draw a uniform random variable $(\kappa'_1,\dots,\kappa'_{r-1})$ on $\mathcal{S}_{r-1}$. Then, we draw  a uniform random variable~$\kappa'_r$ on $\{0,\dots, n-1\}\setminus \{\kappa'_1,\dots,\kappa'_{r-1} \}$. This can be done by sampling an independent random variable $\xi$ that is uniform on $\{0,\dots, n-r\}$ and then set
  $\kappa'_r=\xi+\sum_{i=1}^{r-1}  \mathbf{1}_{\xi+(i-1)\ge \kappa'_i}$. Last, we sort the $\kappa'$, which produces a vector $(\kappa_1,\dots,\kappa_r)$ that is uniformly distributed on  $\mathcal{S}_{r}$. 
\end{remark}

Now that we have an algorithm that is able to calculate $\mathbb{E}[\Gamma^{\mathcal{A}}_0]$ for any tree $\mathcal{A}$, we just have to approximate all these quantities for all the trees $\mathcal{A}\in \mathbf{F}(\mathcal{T}^\nu_0)$ and then to sum these contributions according to~\eqref{D5}. To decide how many samples~$N_\mathcal{A}$ we use to approximate  $\mathbb{E}[\Gamma^{\mathcal{A}}_0]$, we fix a desired precision $\varepsilon>0$, calculate the empirical variance~$\hat{V}_{\mathcal{A}}$ of $c(\mathcal{A})\Gamma^{\mathcal{A}}_0$ on a small sampling and then take $N_\mathcal{A}$ such that $1.96 \sqrt{\hat{V}_{\mathcal{A}}/N_\mathcal{A}} \approx \varepsilon$, so that all the terms have roughly the same statistical error with a 95\% confidence interval half-width equal to~$\varepsilon$.

\subsection{Numerical results for an ODE}

To visualize numerically the orders of convergence provided by~\eqref{D5} for the Euler scheme, it is more convenient to work with ODEs. In this case, the variance of the terms is very small and it is possible to observe the five first order of convergence. In the particular case of a linear ODE $dX_t=k(\theta-X_t)dt$, we can go further, but we can check also that the value of $\Gamma^{\mathcal{A}}_0$ is deterministic and does not depend on the uniform random variables $\kappa$'s. Thus, we have considered the following example
$$ dX_t=\alpha(1-X_t^2)dt,$$
with $X_0=0.4$ and $\alpha=0.1$. The exact value is given by $X_T=\tanh(\textup{arctanh}(X_0)+\alpha T)$. We have drawn on Figure~\ref{Convergence_EDO}, for $T=1$, the values of $\log(|X_T-\hat{\xi}^{n,\nu}_T|)$ in function of $\log(T/n)$ with $\nu=2$,  $\nu=3$, $\nu=4$  and $\nu=5$, where $\hat{\xi}^{N,\nu}_T$ is the estimator of $X_T$ given by equation~\eqref{D5} and $f(x)=x$. The corresponding values of the slopes are $2.003$, $3.025$, $4.056$ and $5.012$ which is in line with what is expected. All the values given on this example are with an half-width of the 95\% confidence interval that does not exceed $3\times 10^{-7}$. Our run for the approximation of order $\nu=6$ already gives with $n=6$ a value that is accurate up to $8\times 10^{-8}$: the exact value $-0.31280256721$ is already in the 95\% confidence interval.
\begin{figure}[h]
  \centering
  \includegraphics[width=\linewidth]{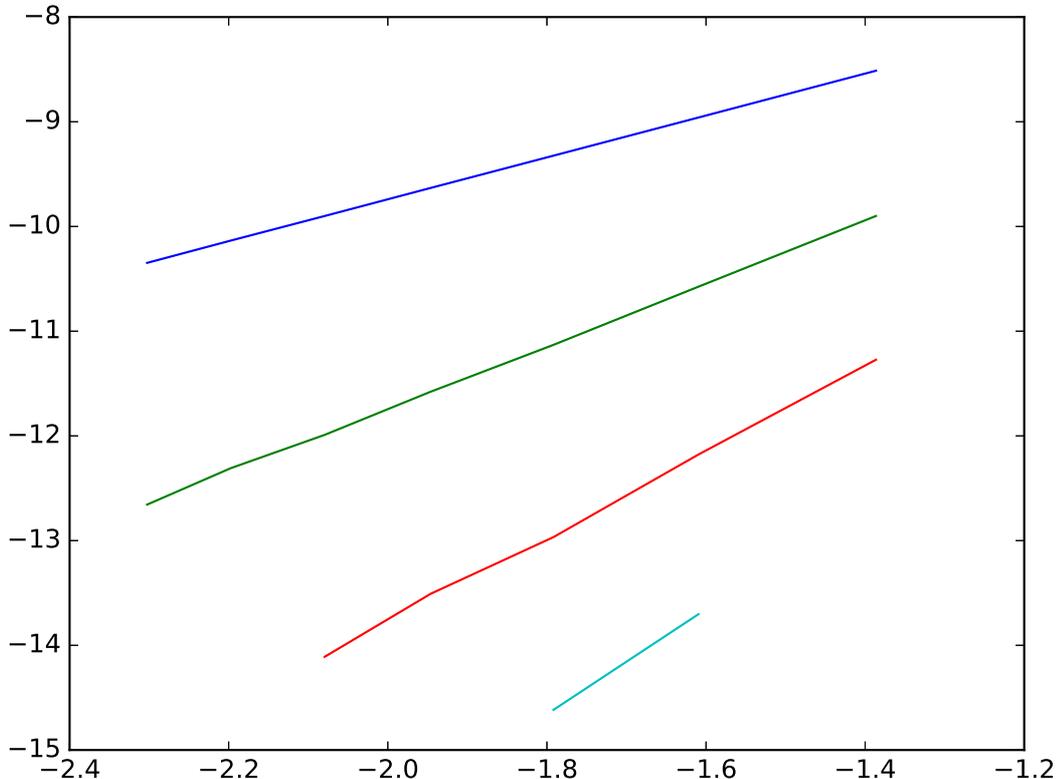}
  \caption{Plot of $\log(|X_1-\hat{\xi}^{n,\nu}_1|)$ in function of~$\log(1/n)$ for $\nu=2$ (blue), $\nu=3$ (green), $\nu=4$ (red) and $\nu=5$ (cyan). } \label{Convergence_EDO}
\end{figure}

Last, let us mention that for this ODE, we have used the same approximation rule as for the SDE and calculated all the terms of~\eqref{D5}. However, as noticed in Remark~\ref{Rk_avoid_calculations}, it is possible to avoid the calculation of many terms. 

\subsection{Numerical results for an SDE}\label{Subsec_Num_SDE}
We now want to illustrate the orders of convergence for the approximation given by~\eqref{D5} for the Euler-Maruyama scheme. We consider the following SDE
$$ dX_t= -kX_t^2dt+\sigma X_t dW_t,$$
with $X_0=1$, $k=1$, $\sigma=0.2$. In Figure~\ref{Convergence_EDS}, we have plotted the approximation of $\mathbb{E}[X_T^2]$ with $T=1$ with the orders $\nu\in \{2,3,4 \}$ in function  of~$1/n$. We still denote by $\hat{\xi}^{n,\nu}_T$ the estimator of $\mathbb{E}[X_T^2]$ given by~\eqref{D5}, using the approximation of order $\nu$ with $n$ time-steps. The  half-width of the 95\% confidence interval is about $2\times 10^{-4}$. The approximation of order $\nu=5$ is already at this level of precision for $n=5$, and we have indicated this value as a reference line for the other schemes. The convergence are again in line with what is expected. 
\begin{figure}[h]
  \centering
  \includegraphics[width=\linewidth]{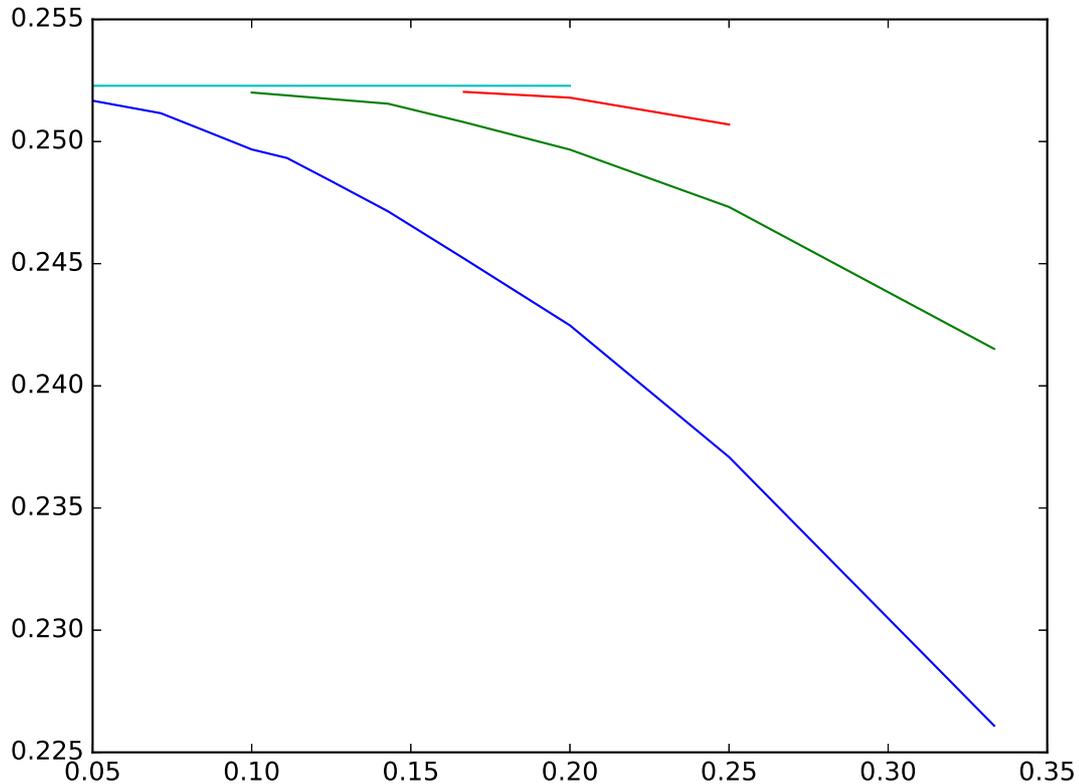}
  \caption{SDE example. Plot of $\hat{\xi}^{n,\nu}_T$ in function of~$1/n$ for $\nu=2$ (blue), $\nu=3$ (green) and $\nu=4$ (red). The value obtained with $\nu=5$ and $n=5$ is given by a cyan line. } \label{Convergence_EDS}
\end{figure}
\subsection{Numerical results for a PDMP}\label{Subsec_Num_PDMP}
We consider the TCP process with infinitesimal generator
$$ Lf(x)=f'(x)+ x(f(x/2)-f(x)),$$
starting from $X_0=1$, and our goal is to approximate $\mathbb{E}[X_T]$, with $T=1$. Since the jumps are only downward, the jump intensity $\lambda(x)$ is bounded by $X_0\times e$ on $[0,1]$. We are thus in the framework of paragraph~\ref{subsubsec_approx_PDMP}, and use the scheme described in~\eqref{scheme_PDMP}. We denote again by $\hat{\xi}^{n,\nu}_T$ the estimator of $\mathbb{E}[X_T]$ given by~\eqref{D5}, using the approximation of order $\nu$ with $n$. In Figure~\ref{Convergence_PDMP}, we have plotted the approximation of $\E[X_T]$ with $T=1$ with the orders $\nu\in \{2,3,4 \}$ in function  of~$1/n$.  The  half-width of the 95\% confidence interval is about $7\times 10^{-4}$. The approximation of order $\nu=5$ is already at this level of precision for $n=5$, and we have indicated this value as a reference line for the other schemes. The plot is very similar to the one obtained in Figure~\ref{Convergence_EDS}. This demonstrates numerically that the approximations described by~\eqref{D5} are relevant for a wide range of processes and applications. 
\begin{figure}[h]
  \centering
  \includegraphics[width=\linewidth]{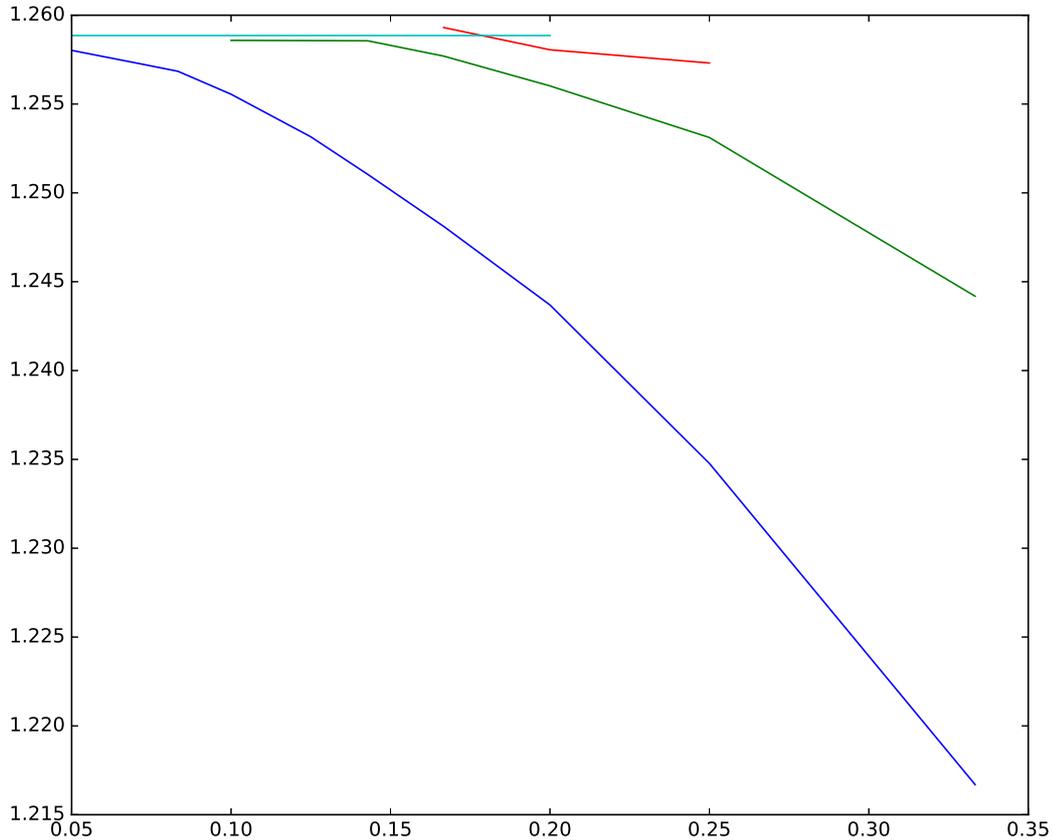}
  \caption{PDMP example. Plot of $\hat{\xi}^{n,\nu}_T$ in function of~$1/n$ for $\nu=2$ (blue), $\nu=3$ (green) and $\nu=4$ (red). The value obtained with $\nu=5$ and $n=5$ is given by a cyan line. } \label{Convergence_PDMP}
\end{figure}

\subsection{A rough complexity analysis}

Now, let us do a rough complexity analysis to understand which order of approximation to use in practice. To make this derivation, we make the assumption for sake of simplicity that the variance corresponding to the term $c(\mathcal{A})\Gamma^{\mathcal{A}}_0$ is equal to~$1$ for all $\mathcal{A}\in \mathbf{F}(\mathcal{T}^\nu_0)$, $\nu\ge 1$. Thus, in this analysis, we will use the same number of samples for all these terms.  We also suppose that we want to achieve a precision of order $\varepsilon>0$, with a standard error which is exactly~$\varepsilon$. Then, we have the following.
\begin{enumerate}
\item For the approximation of order~$1$, we use one Euler scheme with time step $n$ and the standard error is $1/\sqrt{N}$, where $N$ is the number of samples. We take $n=\varepsilon^{-1}$ and  $N=\varepsilon^{-1}$ and the calculation time (counted as the number of Euler iterations used) is $N\times n= \varepsilon^{-3}$.
\item For the approximation of order~$2$, we have two terms corresponding to $\mathcal{A}=\{\emptyset \}$ and $\mathcal{A}=\{\emptyset,1 \}$. The first one requires $n$ calculations of Euler iterations. The second one requires between $2n$ and $3n$ Euler iterations: due to the branching implementation, we only calculate $2n$ iterations when $\kappa(\emptyset)=n-1$ and $3n$ iterations when $\kappa(\emptyset)=0$. For simplicity, we will only consider in this computational cost analysis the worst case and count $3n$ iterations. Since the convergence is of order~$2$, we take $n=\varepsilon^{-1/2}$. The standard error is  $\sqrt{2/N}$, and we take $N=2\varepsilon^{-1}$.  Thus, the calculation time is $N\times(n+3n)=8\varepsilon^{-5/2}$.
\item For the order 3, we have in addition to calculate $\mathbb{E}[\Gamma^{\mathcal{A}}_0]$ for $\mathcal{A}=\{\emptyset,1, 11\}$ and $\mathcal{A}=\{\emptyset,1, 2\}$ that requires respectively $2n+3n=5n$ and $n+2\times 2n+3n=8n$ Euler iterations. We take $n=\varepsilon^{-1/3}$ to have an approximation of order $\varepsilon$. The standard error is  $\sqrt{4/N}$, and we take $N=4\varepsilon^{-1}$. Thus, the calculation time is $N\times (4n+5n+8n)=68\varepsilon^{-7/3}$.
\item  For the order 4, we have in addition to calculate $\mathbb{E}[\Gamma^{\mathcal{A}}_0]$ for $\mathcal{A}=\{\emptyset,1, 11, 111\}$ and $\mathcal{A}=\{\emptyset,1,11, 2\}$, $\mathcal{A}=\{\emptyset,1,2, 21\}$, $\mathcal{A}=\{\emptyset,1,11, 2, 21\}$ and $\mathcal{A}=\{\emptyset,1,2, 3\}$: they require respectively $7n$, $2n+2\times3n+4n=12n$ (see Figure~\ref{Fig_grids}), $12n$, $3n+2\times4n+5n=16n$ and $n+3\times 2n + 3\times 3n +4n=20n$. The overall cost is $17n+7n+2\times 12n+16n+20n=84n$. We then take $n=\varepsilon^{-1/4}$ and $N=9\varepsilon^{-1}$ to have a standard error $\varepsilon$. Thus, the calculation time is $84n\times N=756\varepsilon^{-9/4}$.  
\end{enumerate}
With this rough cost analysis, we would use:
\begin{itemize}
\item the approximation of order 2 rather than the approximation of order 1 if $8\varepsilon^{-5/2}<\varepsilon^{-3}$, i.e. $\varepsilon<1/64$,
\item the approximation of order 3 rather than the approximation of order 2 if $68\varepsilon^{-7/3}<8\varepsilon^{-5/2}$, i.e. $\varepsilon<(8/68)^6\approx 2.6\times 10^{-6}$,
\item the approximation of order 4 rather than the approximation of order 3 if $756\varepsilon^{-9/4}<68\varepsilon^{-7/3}$, i.e. $\varepsilon<(68/756)^{12}\approx 2.8\times 10^{-13}$.
\end{itemize}
This analysis shows that in practice the order 3 may be already sufficient for the precision that is usually needed. However, this cost analysis has to be tempered, because the assumption of a unit variance for each term is rather pessimistic. For ODEs or SDEs with constant diffusion coefficient, we already know from our theoretical results (see Remark~\ref{Rk_variance_terms}) that the variance of $c(\mathcal{A})\Gamma^{\mathcal{A}}_0$ may be much smaller. Also, for SDEs, we see from Table~\ref{Table_variance} that, globally, the terms that are needed for the calculation of order 4 have a smaller variance than the one needed for the order 3, which have also smaller variance than the one needed for the order 2. Of course, there is exception: for example in Table~\ref{Table_variance}, the standard deviation associated to $\{ \emptyset, 1,11, 111 \}$ is of same magnitude as the one associated to $\{\emptyset, 1, 11 \}$ or even $\{\emptyset, 1 \}$. This is why it is better in practice to estimate first the variance of each term and then determine how many samples are needed to achieve a given precision. For the example of Figure~\ref{Convergence_EDS}, to get a precision of $\varepsilon=2\times 10^{-4}$, the approximation of order 2 has required 88s ($n=30$), the order 3 about 89s ($n=10$), the order 4 about  214s ($n=6$) and the order 5 about 345s ($n=5$). Thus, the scheme of order~3 is already competitive for this precision with respect to the order~2.

\begin{table}[h]
  \begin{center}
    \begin{tabular}[h]{|l|c|c| }
      \hline
      $\mathcal{A}$ & Standard deviation of  $c(\mathcal{A})\Gamma^{\mathcal{A}}_0$ & Used for approx of order     \\
      \hline
      $\{\emptyset \}$&  $8.8 \times 10^{-2}$ &  $\nu\ge 1$  \\
      \hline
      $\{\emptyset, 1 \}$  & $3.2 \times 10^{-2}$ & $\nu \ge 2$  \\
      \hline
      $\{\emptyset, 1, 11 \}$  & $1.4 \times 10^{-2}$  & $\nu \ge 3$  \\
      $\{\emptyset, 1, 2 \}$  &  $4.4 \times 10^{-3}$ & $\nu \ge 3$  \\     
      \hline
      $\{\emptyset, 1, 11, 111 \}$  & $1.0 \times 10^{-2}$  & $\nu \ge 4$  \\
      $\{\emptyset, 1, 11, 2 \}$  &  $1.7 \times 10^{-3}$ & $\nu \ge 4$  \\
      $\{\emptyset, 1, 2, 21 \}$  &  $1.5 \times 10^{-3}$ & $\nu \ge 4$  \\
      $\{\emptyset, 1, 11, 2, 21 \}$  & $5.4 \times 10^{-4}$  & $\nu \ge 4$  \\
      $\{\emptyset, 1, 2, 3 \}$  &  $3.7 \times 10^{-4}$ & $\nu \ge 4$  \\
      \hline
   \end{tabular}
   
   \caption{Empirical standard deviation of $c(\mathcal{A})\Gamma^{\mathcal{A}}_0$ for $f(x)=x^2$ and $n=5$, on the SDE example described in Subsection~\ref{Subsec_Num_SDE}.}\label{Table_variance}
 \end{center}
\end{table}

\appendix

\section{Technical results for the variance analysis}

We introduce some notation. We consider smooth random fields, that is
functions $\varphi :\Omega \times \R^{d}\rightarrow \R^{d}$ which are
measurable with respect to $(\omega ,x)$ and such that, for each $\omega ,$
the function $x\mapsto \varphi (\omega ,x)$ is of class $C^{\infty
}(\R^{d})$. For such a random field we denote
\begin{eqnarray}
\left\Vert \varphi \right\Vert _{0,p,\infty } &=&\sup_{x}\left\Vert \varphi
(x)\right\Vert _{p}=\sup_{x}(\int \left\vert \varphi (\omega ,x)\right\vert
^{p} d\Px(\omega ))^{1/p},  \label{APP1a} \\
\left\Vert \varphi \right\Vert _{q,p,\infty } &=&\sum_{\left\vert \alpha
\right\vert \leq q}\left\Vert \partial ^{\alpha }\varphi \right\Vert
_{0,p,\infty }.  \label{APP1b}
\end{eqnarray}%
Moreover, we will say that a sequence of random fields $\varphi
_{i},i=1,...,m$ are independent if there are some independent $\sigma -$%
algebras $\mathcal{G}_{i},i=1,...,m$ such that $\varphi _{i}$ is $\mathcal{G}_{i}\otimes \mathcal{B}(\R^{d})$ measurable. We will use this property as follows. Suppose that $\Phi$ is $\mathcal{G}_{m}\otimes \mathcal{B}(\R^{d})$ measurable and $\Psi $ and $\Theta$ are $\vee _{i=1}^{m-1}\mathcal{G}_{i}\otimes \mathcal{B}(\R^{d})$ measurable. Then, for every $x\in \R^{d}$ and
every $p\geq 1$%
\begin{eqnarray}
\E(\left\vert \Phi (\omega ,\Psi (\omega ,x))\right\vert ^{p} |\Theta (\omega
,x)|) &=&\E\left(|\Theta (\omega ,x)|  \E\left( \left\vert \Phi (\omega ,\Psi (\omega ,x))\right\vert ^{p} \bigg| \vee _{i=1}^{m-1}\mathcal{G}_{i} \right)\right) \notag \\
&\leq &\left\Vert \Phi \right\Vert _{0,p,\infty }^{p}\E(\left\vert \Theta
(\omega ,x)\right\vert ).   \label{APP1e}
\end{eqnarray}

In the sequel we consider a sequence of smooth random fields $\Phi
_{i}:\Omega \times \R^{d}\rightarrow \R^{d}$ and $\phi_{i}:\Omega \times \R^{d}\rightarrow \R^{d}$ , $i\in \N$ and moreover, a vector field $\varphi:\Omega \times \R^{d}\rightarrow \R^{d}$. We assume that $\varphi $ and $(\Phi_{j},\phi _{j}),j\in \N$ are independent. We fix $r\in \N$ and, for a set $
\Lambda \subset \{1,...,r\},$ we define
\begin{eqnarray*}
\theta_{i}^{\Lambda } &=&1_{\Lambda }(i)\phi _{i}+1_{\Lambda ^{c}}(i)\Phi
_{i},\quad i=1,...,r\quad and \\
\theta _{(r)}^{\Lambda } &=&\theta _{r}^{\Lambda }\circ ....\circ \theta
_{1}^{\Lambda }
\end{eqnarray*}%
Moreover, given a multi-index $\alpha ,$ we define%
\begin{equation}
\Gamma _{r}^{\alpha }\varphi (x)=\sum_{\Lambda \subset
\{1,...,r\}}(-1)^{\left\vert \Lambda \right\vert }\partial _{x}^{\alpha
}[\varphi (\theta _{(r)}^{\Lambda })](x)  \label{APP1d}
\end{equation}

\begin{proposition}
\label{appendix}Suppose that for every $p,q\in \N$, there exists $C_{q,p}$ such that
\begin{equation}
\forall i\in \{1,...,r\}, \ \left\Vert \Phi _{i}\right\Vert _{q,p,\infty }+\left\Vert \phi
_{i}\right\Vert _{q,p,\infty }+\left\Vert \varphi \right\Vert _{q+1,p,\infty
}\leq C_{q,p}<\infty .  \label{APP1c}
\end{equation}%
Then, for  every $p\geq 1$ and every multi-index $\alpha $ we have
\begin{equation}
\left\Vert \Gamma _{r}^{\alpha }\varphi \right\Vert _{0,p,\infty }\leq
C\times \prod_{i=1}^{r}\left\Vert \Phi _{i}-\phi _{i}\right\Vert
_{\left\vert \alpha \right\vert +r,2 p,\infty }  \label{APP2}
\end{equation}%
for some $C$ depending on $r$, $\left\vert \alpha \right\vert $ and $C_{\left\vert \alpha \right\vert +r,2|\alpha| p}.$
\end{proposition}

\begin{remark} This proposition says the following: if at each step the error is of
order $\delta _{i}=\left\Vert \Phi _{i}-\phi _{i}\right\Vert _{q,p^{\prime
},\infty }$, then after $r$ steps we have an error of order $\delta
_{1}\times ...\times \delta _{r}.$ This may seem a little surprising, and one may have
expected an error of order $\delta _{1}+...+\delta _{r}$, but this is
due to the way how terms are summed with $\sum_{\Lambda \subset \{1,...,r\}}(-1)^{\left\vert \Lambda
\right\vert }.$
\end{remark}

\begin{proof}
  \textbf{Step 1}. We use the Faà di Bruno formula $\partial ^{\alpha }[f\circ g]=\sum_{\left\vert \beta \right\vert \leq \left\vert \alpha \right\vert }(\partial ^{\beta }f)(g)P_{\alpha ,\beta }(g)$ (see~\eqref{APP3}) and the inequality between geometric and arithmetic means to upper bound the terms $\vert\prod_{i=1}^{k}\partial ^{\gamma _{i}} g^{j_i} \vert$ defining $P_{\alpha ,\beta }(g)$. We then obtain for random functions~$g$
\begin{equation*}
\left\Vert P_{\alpha ,\beta }(g)\right\Vert _{0,p,\infty }^p\leq
C \left\Vert g\right\Vert _{\left\vert \alpha \right\vert ,|\alpha| p,\infty }^{|\alpha| p}.
\end{equation*}%
Besides, for two random fields $g_{1}$ and $g_{2}$, we write
$$ \prod_{i=1}^{k}\partial ^{\gamma _{i}}g_1^{j_i}-\prod_{i=1}^{k}\partial ^{\gamma _{i}}g_2^{j_i}=\sum_{i'=1}^k \left(\prod_{i<i'}\partial ^{\gamma _{i}}g_2^{j_i} \right) (\partial ^{\gamma _{i'}}g_1^{j_{i'}}-\partial ^{\gamma _{i'}}g_2^{j_{i'}} )  \left(\prod_{i>i'}\partial ^{\gamma _{i}}g_1^{j_i}\right). $$
Using the inequality between geometric and arithmetic means for the product on $i\not=j$ and then the Cauchy-Schwarz inequality, we get
\begin{equation}
\left\Vert P_{\alpha ,\beta }(g_{1})-P_{\alpha ,\beta }(g_{2})\right\Vert
_{0,p,\infty }^p\leq C(\left\Vert g_{1}\right\Vert _{\left\vert \alpha
\right\vert ,2(|\alpha|-1)p,\infty }+\left\Vert
g_{2}\right\Vert _{\left\vert \alpha \right\vert ,2(|\alpha|-1)p,\infty })^{(|\alpha|-1)p} \left\Vert
g_{1}-g_{2}\right\Vert^p_{\left\vert \alpha \right\vert ,2 p,\infty }.  \label{APP4a}
\end{equation}

\textbf{Step 2.} We prove (\ref{APP2}) for $r=1.$ In this case $\Lambda
=\varnothing $ or $\Lambda =\{1\}$ so that%
\begin{eqnarray*}
\Gamma _{1}^{\alpha }\varphi (x) &=&\partial ^{\alpha }[\varphi (\Phi
_{1})](x)-\partial ^{\alpha }[\varphi (\phi _{1})](x) \\
&=&\sum_{\left\vert \beta \right\vert \leq \left\vert \alpha \right\vert
}(\partial ^{\beta }\varphi )(\Phi _{1})P_{\alpha ,\beta }(\Phi
_{1})-(\partial ^{\beta }\varphi )(\phi _{1})P_{\alpha ,\beta }(\phi _{1}) \\
&=&\sum_{\left\vert \beta \right\vert \leq \left\vert \alpha \right\vert
}A_{\beta }+B_{\beta }
\end{eqnarray*}%
with 
\begin{eqnarray*}
A_{\beta } &=&((\partial ^{\beta }\varphi )(\Phi _{1})-(\partial ^{\beta
}\varphi )(\phi _{1}))P_{\alpha ,\beta }(\Phi _{1}) \\
&=&P_{\alpha ,\beta }(\Phi _{1})\int_{0}^{1}\left\langle \nabla (\partial
^{\beta }\varphi )(\lambda \Phi _{1}+(1-\lambda )\phi _{1}),\Phi _{1}-\phi
_{1}\right\rangle d\lambda
\end{eqnarray*}%
and 
\begin{equation*}
B_{\beta }=(\partial _{\beta }\varphi )(\phi _{1})(P_{\alpha ,\beta }(\Phi
_{1})-P_{\alpha ,\beta }(\phi _{1})).
\end{equation*}%
Using (\ref{APP1c}) and the fact that $\varphi $ is independent of $\lambda
\Phi _{1}+(1-\lambda )\phi _{1}$ we get (see (\ref{APP1e}))%
\begin{eqnarray*}
\left\Vert \left\langle \nabla (\partial ^{\beta }\varphi )(\lambda \Phi
_{1}+(1-\lambda )\phi _{1}),\Phi _{1}-\phi _{1}\right\rangle \right\Vert
_{p} &\leq &\left\Vert \varphi \right\Vert _{\left\vert \beta \right\vert
+1,p,\infty }\left\Vert \Phi _{1}-\phi _{1}\right\Vert _{0,p,\infty} \\
&\leq &C_{\left\vert \alpha \right\vert ,p}\left\Vert \Phi _{1}-\phi
_{1}\right\Vert _{0,p,\infty}
\end{eqnarray*}%
so that $\left\Vert A_{\beta }\right\Vert _{0,p,\infty }\leq C\left\Vert
\Phi _{1}-\phi _{1}\right\Vert _{0,p,\infty }.$ Moreover, using again (\ref{APP1e}%
) first and then (\ref{APP4a}), we get%
\begin{equation*}
\left\Vert B_{\beta }\right\Vert _{0,p,\infty }\leq \left\Vert \varphi
\right\Vert _{\left\vert \beta \right\vert ,p,\infty }\left\Vert P_{\alpha
,\beta }(\Phi _{1})-P_{\alpha ,\beta }(\phi _{1})\right\Vert _{0,p,\infty
}\leq C\left\Vert \Phi _{1}-\phi _{1}\right\Vert _{|\alpha|,2p,\infty }
\end{equation*}%
so (\ref{APP2}) is proved for $r=1$.

\textbf{Step 3. }Suppose (\ref{APP2}) is true for $r-1,$ for every $\alpha $
and every $p\geq 1.$ We prove it for $r.$ We do it first for $\alpha
=\varnothing $ (without derivatives) because it is simpler. We write%
\begin{eqnarray*}
\Gamma _{r}^{\varnothing }\varphi &=&\sum_{\Lambda \subset
\{1,...,r\}}(-1)^{\left\vert \Lambda \right\vert }\varphi (\theta
_{(r)}^{\Lambda }) \\
&=&\sum_{\Lambda ^{\prime }\subset \{2,...,r\}}(-1)^{\left\vert \Lambda
^{\prime }\right\vert }(\varphi (\theta _{(r-1)}^{\Lambda ^{\prime }})(\Phi
_{1})-\varphi (\theta _{(r-1)}^{\Lambda ^{\prime }})(\phi _{1})) \\
&=&\sum_{i=1}^{d}(\Phi _{1}^{i}-\phi _{1}^{i})\int_{0}^{1}\sum_{\Lambda
^{\prime }\subset \{2,...,r\}}(-1)^{\left\vert \Lambda ^{\prime }\right\vert
}\partial ^{i}[\varphi (\theta _{(r-1)}^{\Lambda ^{\prime }})](\lambda \Phi
_{1}+(1-\lambda )\phi _{1})d\lambda \\
&=&\sum_{i=1}^{d}(\Phi _{1}^{i}-\phi _{1}^{i})\int_{0}^{1}\Gamma
_{r-1}^{(i)}\varphi (\lambda \Phi _{1}+(1-\lambda )\phi _{1})d\lambda.
\end{eqnarray*}%
Note that we have made a slight abuse of notation here: 
the notation $\theta _{(r-1)}^{\Lambda ^{\prime }}$ is used in fact for $\theta^{\Lambda'}_r \circ \dots \circ \theta^{\Lambda'}_2$, not for $\theta^{\Lambda'}_{r-1} \circ \dots \circ \theta^{\Lambda'}_1$.
Since $(\lambda \Phi _{1}+(1-\lambda )\phi _{1})(x)$ is independent of $%
\Gamma _{r-1}^{(i)}\varphi ,$ we have from~\eqref{APP1e} 
\begin{eqnarray*}
\left\Vert \Gamma _{r}^{\varnothing }\varphi (x)\right\Vert _{p} &\leq
&d \left\Vert \Phi _{1}-\phi _{1}\right\Vert _{0,p,\infty}\times \left\Vert \Gamma
_{r-1}^{(i)}\varphi \right\Vert _{0,p,\infty }.
\end{eqnarray*}
Then, by using the induction hypothesis, we get
\begin{eqnarray*}
  \left\Vert \Gamma _{r}^{\varnothing }\varphi (x)\right\Vert _{p} &\leq & C\left\Vert \Phi _{1}-\phi _{1}\right\Vert _{0,p,\infty}\times
\prod_{j=2}^{r}\left\Vert \Phi _{j}-\phi _{j}\right\Vert
_{1+r-1,2 p,\infty } \\
&\le&C\prod_{j=1}^{r}\left\Vert \Phi _{j}-\phi _{j}\right\Vert
_{r,2p,\infty }.
\end{eqnarray*}%

We prove now (\ref{APP2}) for a general multi-index $\alpha$ and make the same abuse of notation for $\theta^{\Lambda'}$. Using (\ref{APP3}) and (\ref{APP4}) for $f=\varphi (\theta _{(r-1)}^{\Lambda ^{\prime}})$ and $g_{1}=\Phi _{1},g_{2}=\phi _{1}$ we obtain%
\begin{eqnarray*}
\Gamma _{r}^{\alpha }\varphi &=&\sum_{\Lambda \subset
\{1,...,r\}}(-1)^{\left\vert \Lambda \right\vert }\partial _{x}^{\alpha
}[\varphi (\theta _{(r)}^{\Lambda })] \\
&=&\sum_{\Lambda ^{\prime }\subset \{2,...,r\}}(-1)^{\left\vert \Lambda
^{\prime }\right\vert }\sum_{\left\vert \beta \right\vert \leq \left\vert
\alpha \right\vert }\partial ^{\beta }[\varphi (\theta _{(r-1)}^{\Lambda
^{\prime }})](\Phi _{1})P_{\alpha ,\beta }(\Phi _{1}) \\
&&-\sum_{\Lambda ^{\prime }\subset \{2,...,r\}}(-1)^{\left\vert \Lambda
^{\prime }\right\vert }\sum_{\left\vert \beta \right\vert \leq \left\vert
\alpha \right\vert }\partial ^{\beta }[\varphi (\theta _{(r-1)}^{\Lambda
^{\prime }})](\phi _{1})P_{\alpha ,\beta }(\phi _{1}) \\
&=&\sum_{\left\vert \beta \right\vert \leq \left\vert \alpha \right\vert
}\Gamma _{r-1}^{\beta }\varphi (\Phi _{1})P_{\alpha ,\beta }(\Phi
_{1})-\Gamma _{r-1}^{\beta }\varphi (\phi _{1})P_{\alpha ,\beta }(\phi _{1})
\\
&=&\sum_{\left\vert \beta \right\vert \leq \left\vert \alpha \right\vert
}A_{\beta }+B_{\beta }
\end{eqnarray*}%
with 
\begin{equation*}
A_{\beta }=(\Gamma _{r-1}^{\beta }\varphi (\Phi _{1})-\Gamma _{r-1}^{\beta
}\varphi (\phi _{1}))P_{\alpha ,\beta }(\Phi _{1})
\end{equation*}%
and 
\begin{equation*}
B_{\beta }=\Gamma _{r-1}^{\beta }\varphi (\Phi _{1})(P_{\alpha ,\beta }(\Phi
_{1})-P_{\alpha ,\beta }(\phi _{1})).
\end{equation*}%
By assumption $\theta_2,\dots,\theta_r$ are independent of $(\Phi_{1},\phi _{1})$. Therefore, $\Gamma _{r-1}^{\beta }\varphi (x)$ is independent of $(\Phi_{1},\phi _{1})$. We use~\eqref{APP1e} first and then the induction hypothesis and~\eqref{APP4a} to obtain%
\begin{eqnarray*}
\left\Vert B_{\beta }\right\Vert _{0,p,\infty } &\leq &\left\Vert \Gamma
_{r-1}^{\beta }\varphi \right\Vert _{0,p,\infty }\left\Vert P_{\alpha ,\beta
}(\Phi _{1})-P_{\alpha ,\beta }(\phi _{1})\right\Vert _{0,p,\infty } \\
&\leq &C    \left\Vert \Phi _{1}-\phi _{1}\right\Vert
_{|\alpha|,2 p,\infty }  \prod_{i=2}^{r}\left\Vert \Phi_{i}-\phi _{i}\right\Vert
_{\left\vert \beta \right\vert +r-1,2 p,\infty} \\
 &\leq &C \prod_{i=1}^{r}\left\Vert \Phi_{i}-\phi _{i}\right\Vert
_{\left\vert \alpha \right\vert +r,2 p,\infty} .
\end{eqnarray*}%
Moreover%
\begin{equation*}
A_{\beta }=P_{\alpha ,\beta }(\Phi _{1})\int_{0}^{1}\left\langle (\nabla
\Gamma _{r-1}^{\beta }\varphi )(\lambda \Phi _{1}+(1-\lambda )\phi
_{1}),\Phi _{1}-\phi _{1}\right\rangle d\lambda .
\end{equation*}%
Notice that $\partial ^{i}\Gamma _{r-1}^{\beta }\varphi =\Gamma
_{r-1}^{(\beta ,i)}\varphi$. Using again~\eqref{APP1e} and the recurrence hypothesis, we get
\begin{equation*}
\left\Vert A_{\beta }\right\Vert _{0,p,\infty }\leq
C\prod_{i=1}^{r}\left\Vert \Phi _{i}-\phi _{i}\right\Vert _{\left\vert \beta
\right\vert +1+r-1,2p,\infty }\le C \prod_{i=1}^{r}\left\Vert \Phi_{i}-\phi _{i}\right\Vert
_{\left\vert \alpha \right\vert +r,2 p,\infty} .\qedhere
\end{equation*}%
\end{proof}

\bigskip
\begin{lemma}\label{lemma_flow_der}
  Let $(X_t(x))_{t\ge0} $ denote the flow of the SDE~\eqref{i1} and $\hat{X}_t(x)=x+b(x)t+\sigma(x)W_t$ the flow of the Euler scheme. We assume that $b$ and $\sigma$ are $C^\infty$, bounded and with bounded derivatives. Then, we have
  $$ \forall p,q\in \N, \exists C_{p,q}, \|\hat{X}_t-X_t\|_{q,p,\infty}\le C_{p,q} t^a,$$
 with $a=2$ if $\sigma=0$, $a=3/2$ if $\sigma(x)$ is a constant function and $a=1$ in the general case. 
\end{lemma}
\begin{proof}
  We show this result by induction on $q$. We only focus on the general case, the cases $\sigma=0$ or $\sigma(x)$ constant can be then easily deduced. For $q=0$, this result is stated for example in Proposition~1.2~\cite{AJK}. For simplicity of notation, we do the proof in dimension~$d=1$ with $b=0$. We note $\sigma^{(q)}$ the $q$-th derivative of~$\sigma$. For $q=1$, we have 
  $\hat{X}_t^{(1)}(x)=1+\sigma^{(1)}(x)W_t$ and \begin{equation}\label{estim_moments}
    X_t^{(1)}(x)=1+\int_0^tX_s^{(1)}(x)  \sigma^{(1)} (X_s(x)) dW_s .
  \end{equation}
  Since $\sigma^{(1)}$ is bounded, we have $\forall t>0, \sup_{x}\E[\sup_{s\in[0,t]} |X_s^{(1)}(x)|^p]<\infty$.  We write
  $$\hat{X}_t^{(1)}(x)- X_t^{(1)}(x)=\int_0^t (X_s^{(1)}(x)-1)  \sigma^{(1)} (X_s(x)) dW_s +\int_0^t  \sigma^{(1)} (X_s(x))-\sigma^{(1)}(x) dW_s. $$
  Since $\sigma^{(1)}$ is bounded and Lipschitz, we get by using the Burkholder-Davis-Gundy inequality and then Jensen inequality
  $$\E[|\hat{X}_t^{(1)}(x)- X_t^{(1)}(x)|^p\le C t^{p/2-1}\int_0^t \E[|X_s^{(1)}(x)-1|^p]+\E[|X_s(x)-x|^p]ds,$$
with a constant $C$ that does not depend on~$x$. We check then again with the BDG inequality that $\E[|X_s(x)-x|^p]\le Cs^{p/2}$ since $\sigma$ is bounded and  $\E[|X_s^{(1)}(x)-1|^p]\le Cs^{p/2}$ since $\sigma^{(1)}$ is bounded and~\eqref{estim_moments}. Thus, we have $\E[|\hat{X}_t^{(1)}(x)- X_t^{(1)}(x)|^p]\le C t^p$.

  We suppose now the result true for $q-1\in \N^*$ and that we have shown that \begin{equation}\label{estim_tempscourt}
    \E[|X_s^{(r)}(x)|^p]\le Cs^{p/2},\text{ for }2\le r\le q-1,
  \end{equation}
for each $p$, with a constant~$C$ that does not depend on $x$.   
 We have
  $\hat{X}_t^{(q)}(x)=\sigma^{(q)}(x)W_t,$
  and by the Faà di Bruno formula
  \begin{align*}
    d X^{(q)}_t(x)&=\sum_{m_1+\dots+ q m_q=q}c_{m_1,\dots,m_q} \prod_{k=1}^q( X^{(k)}_t(x))^{m_k}  \sigma^{(m_1+\dots+m_q)}(X_t(x))dW_t \\
  &=  X^{(q)}_t(x)  \sigma^{(1)}(X_t(x))dW_t +  (X^{(1)}_t(x))^q  \sigma^{(q)}(X_t(x))dW_t +A_tdW_t,
  \end{align*}
  with $A_t=\sum_{m_1+\dots+ q m_q=q, m_1\not = q,m_q\not=0 }c_{m_1,\dots,m_q} \prod_{k=1}^q( X^{(k)}_t(x))^{m_k}  \sigma^{(m_1+\dots+m_q)}(X_t(x))$. Note that in this sum is equal to $0$ for $q=2$ and otherwise there is at least one $k\in\{2,\dots,q-1\}$, such that $m_k\ge 1$. This gives $\E[|A_t|^p]\le Ct^{p/2}$ by using the induction hypothesis~\eqref{estim_tempscourt} and H\"older type inequalities. Since $X^{(q)}_0(x)=0$, $\sigma^{(1)}$ and $ \sigma^{(q)}$ are bounded and  $\sup_{x}\E[\sup_{s\in[0,t]} |X_s^{(1)}(x)|^p]<\infty$ for any $p$, we get  $\E[|X_t^{(q)}(x)|^p]\le Ct^{p/2}$ by using BDG and Gronwall inequalities. Therefore, $\tilde{A}_t=A_t+ X^{(q)}_t(x)  \sigma^{(1)}(X_t(x))$ also satisfies $\E[|\tilde{A}_t|^p]\le Ct^{p/2}$.

  We now repeat the same arguments as for $q=1$: from
  \begin{align*}
    \hat{X}_t^{(q)}(x)- X_t^{(q)}(x)=&\int_0^t ((X_s^{(1)}(x))^q-1)  \sigma^{(q)} (X_s(x)) dW_s +\int_0^t  \sigma^{(q)} (X_s(x))-\sigma^{(q)}(x) dW_s \\&  + \int_0^t \tilde{A}_tdW_t,
  \end{align*}
we get $\E[|\hat{X}_t^{(q)}(x)- X_t^{(q)}(x)|^p]\le C t^p$.

\end{proof}
  
\subsubsection*{Acknowledgements}

 Aur\'elien Alfonsi benefited from the support of the ``Chaire Risques Financiers'', Fondation du Risque. 

\bibliographystyle{plain}
\bibliography{biblio}

\end{document}